\documentclass[11pt]{article}

\usepackage{amsmath, amsfonts, amssymb, amsthm,  graphicx, enumerate}
\usepackage[letterpaper, left=1truein, right=1truein, top = 1truein, bottom = 1truein]{geometry}

\usepackage[numbers, square]{natbib}

\usepackage[%dvips,
            CJKbookmarks=true,
            bookmarksnumbered=true,
            bookmarksopen=true,
            colorlinks=true,
            citecolor=red,
            linkcolor=blue,
            anchorcolor=red,
            urlcolor=blue,
            ]{hyperref}

\usepackage[blocks, affil-it]{authblk}

\usepackage[usenames,dvipsnames]{color}

\usepackage[ruled, vlined, linesnumbered]{algorithm2e}

\usepackage{prettyref,soul,xspace}

\usepackage{tikz}
\usepackage{pgfplots,makecell}

\usepackage{bm}
\usepackage{fancyhdr}
\theoremstyle{remark}

\theoremstyle{plain}

\newtheorem{lemma}{Lemma}[section]
\newtheorem{theorem}{Theorem}[section]

\newtheorem{corollary}{Corollary}[section]
\newtheorem{condition}{Condition}[subsection]

\newcommand{\p}{\mathbb{P}}
\newcommand{\E}{\mathbb{E}}

\newcommand{\iid}{\stackrel{iid}{\sim}}

\newcommand{\br}[1]{\left( #1 \right)}

\newcommand{\cbr}[1]{\left\{ #1 \right\}}
\newcommand{\pbr}[1]{\p\left( #1 \right)}
\newcommand{\ebr}[1]{\exp\left( #1 \right)}
\newcommand{\abs}[1]{\left| #1 \right|}
\newcommand{\sign}[1]{\text{ sign}\left( #1 \right)}

\newcommand{\mathr}{\mathbb{R}}
\newcommand{\mathn}{\mathcal{N}}

\newcommand{\indic}[1]{{\mathbb{I}\left\{{#1}\right\}}}
\newcommand{\iprod}[2]{\left \langle #1, #2 \right\rangle}
\newcommand{\norm}[1]{\left\|{#1} \right\|}

\newcommand{\argmin}{\mathop{\rm argmin}}

\newcommand{\bbr}[1]{\left\{{#1} \right\}}

\newcommand{\normp}[1]{\|{#1} \|}
\newcommand{\snr}{\textsf{SNR}}

\newcommand{\Sigmahatazstar}{\hat{\Sigma}_a(z^*)}
\newcommand{\Sigmahatbzstar}{\hat{\Sigma}_b(z^*)}
\newcommand{\Sigmahatzstar}{\hat{\Sigma}(z^*)}
\newcommand{\Sigmahataz}{\hat{\Sigma}_a(z)}
\newcommand{\Sigmahatbz}{\hat{\Sigma}_b(z)}
\newcommand{\Sigmahatz}{\hat{\Sigma}(z)}
\newcommand{\Sigmaastar}{\Sigma_a^*}
\newcommand{\Sigmabstar}{\Sigma_b^*}
\newcommand{\Sigmastar}{\Sigma^*}
\newcommand{\Sigmazjstarstar}{\Sigma_{z_j^*}^*}
\newcommand{\Sigmahatzjstarzstar}{\hat{\Sigma}_{z_j^*}(z^*)}
\newcommand{\sumzjstareqa}{\sum_{j=1}^{n}\mathbb{I}\{z_j^*=a\}}
\newcommand{\sumzjeqa}{\sum_{j=1}^{n}\mathbb{I}\{z_j=a\}}
\newcommand{\thetahatazstar}{\hat{\theta}_a(z^*)}
\newcommand{\thetahatbzstar}{\hat{\theta}_b(z^*)}
\newcommand{\thetahataz}{\hat{\theta}_a(z)}
\newcommand{\thetahatbz}{\hat{\theta}_b(z)}
\newcommand{\thetaastar}{\theta_a^*}
\newcommand{\thetabstar}{\theta_b^*}
\newcommand{\thetahatzjstarzstar}{\hat{\theta}_{z_j^*}(z^*)}
\newcommand{\romaone}{\uppercase\expandafter{\romannumeral1}}
\newcommand{\romatwo}{\uppercase\expandafter{\romannumeral2}}

\newcommand{\indzja}{\mathbb{I}\{z_j=a\}}
\newcommand{\indzjstara}{\mathbb{I}\{z_j^*=a\}}
\newcommand{\zjstar}{z_j^*}
\newcommand{\thetazjstarstar}{\theta_{\zjstar}^*}

\newcommand{\lambdamin}{\lambda_{\min}}
\newcommand{\lambdamax}{\lambda_{\max}}

\newcommand{\dab}{\Delta_{a,b}}
\newcommand{\xiab}{\Xi_{a,b}}
\renewcommand{\preceq}{\lesssim}

\title{Optimal Clustering in  Anisotropic Gaussian Mixture Models}
\author[1]{Xin Chen}
\author[2]{Anderson Y. Zhang}
\affil[1]{
University of Washington
}

\affil[2]{
University of Pennsylvania
}

\date{}

\begin{document}
\maketitle

\begin{abstract}
We study the clustering task under anisotropic Gaussian Mixture Models where the covariance matrices from different clusters are unknown and are not necessarily the identical matrix. We characterize the dependence of signal-to-noise ratios on the cluster centers and covariance matrices and obtain the minimax lower bound for the  clustering problem. In addition, we propose a computationally feasible procedure and prove it achieves the optimal rate within a few iterations. The proposed procedure is a hard EM type algorithm, and it can also be seen as a variant of the Lloyd's algorithm that is adjusted to the anisotropic covariance matrices.
\end{abstract}

\section{Introduction}

Clustering is a fundamentally important task in statistics and machine learning \cite{friedman2001elements, bishop2006pattern}. The most popular and  studied model for clustering is  the Gaussian Mixture Model (GMM) \cite{pearson1894contributions, titterington1985statistical} which can be written as
\begin{align*}
Y_j  &= \theta^*_{z^*_j} + \epsilon_j, \text{ where }\epsilon_j \stackrel{ind}{\sim}  \mathn(0,\Sigma^*_{z^*_j}), \forall j\in[n].
\end{align*}
Here $Y=(Y_1,\ldots,Y_n)$ are the observations with $n$ being the sample size. Let $k$ be the number of clusters that is assumed to be known. Denote $\{\theta^*_a\}_{a\in[k]}$ to be  unknown centers and $\{\Sigma^*_a\}$ to be  unknown  covariance matrices for the $k$ clusters. Let $z^*\in[k]^n$ be the cluster structure  such that  for each index $j\in[n]$, the value of $z^*_j$ indicates which cluster the $j$th data point belongs to. The goal is to recover $z^*$ from $Y$. For any estimator $\hat z$, its clustering performance is measured by a misclustering error rate $h(\hat z,z^*)$ which will be introduced later in (\ref{eqn:h}).

There have been increasing interests in the theoretical and algorithmic analysis of  clustering under  GMMs. When the GMM is isotropic (that is, all the covariance matrices $\{\Sigma^*_{a}\}_{a\in [k]}$ are equal to the same identity matrix), \cite{lu2016statistical} obtains the minimax rate for clustering which takes a form of $\exp(-(1+o(1)) (\min_{a\neq b} \|\theta^*_a - \theta^*_b\|)^2/8)$ under the loss $h(\hat z,z^*)$.  Various methods have been studied in the isotropic setting. 
$k$-means clustering \cite{macqueen1967some} might be the most natural choice but it is NP-hard \cite{dasgupta2008hardness}. As a local approach to optimize the $k$-mean objects, Lloyd's algorithm  \cite{lloyd1982least} is one of the most popular clustering algorithms  and has achieved many successes in different disciplines \cite{wu2008top}.
 \cite{lu2016statistical, gao2019iterative} establishes computational and statistical guarantees for the Lloyd's algorithm by
showing  it achieves the optimal rates after a few iterations provided with some decent initialization. Another popular  approach to clustering especially for  high-dimensional data is  spectral clustering \cite{von2007tutorial, spielman1996spectral, VempalaWang04}, which is an umbrella
term for clustering after a dimension reduction through a spectral decomposition. \cite{loffler2019optimality, Ndaoud19, AbbeFanWang20}  proves the spectral clustering also achieves the optimality under the isotropic GMM. Another line of work is to consider
semidefinite programming (SDP) as a convex relaxation of the $k$-means objective.
Its statistical properties have been studied in \cite{fei2018hidden, giraud2018partial}.

In spite of  all the  exciting results, most of the existing literature focuses on  isotropic GMMs, and clustering under the anisotropic case where the covariance matrices are not necessarily the identity matrix is not well-understood. The results of some papers \cite{lu2016statistical, fei2018hidden} hold under  sub-Gaussian mixture models, where the errors $\epsilon_j  $ are assumed to follow some sub-Gaussian distribution with variance proxy  $\sigma^2$. It seems that their result already covers the anisotropic case, as $\{\mathn(0,\Sigma^*_a)\}_{a\in[k]}$ are indeed sub-Gaussian distributions. However, from a minimax point of view, among all the sub-Gaussian distributions with  variance proxy  $\sigma^2$, the least favorable case (the case where clustering is the most difficult) is when the errors are $\mathn(0,\sigma^2)$. Therefore, the minimax rates for clustering under the sub-Gaussian  mixture model is essentially the one under  isotropic GMMs, and methods such as the Lloyd's algorithm that requires no covariance matrix information can be rate-optimal.  As a result, the aforementioned results are all for isotropic GMMs.

A few papers have explored the direction of clustering under anisotropic GMMs.
 \cite{brubaker2008isotropic} gives a polynomial-time clustering algorithm that provably works well when the Gaussian distributions are well separated by hyperplanes.  Their idea is further developed in \cite{kalai2010efficiently} which allows the Gaussians to be overlapped with each other but only for two-cluster cases. A recent paper \cite{wang2020efficient} proposes another method for clustering under a balanced mixture of two elliptical distributions. They give a provable upper bound of their clustering performance with respect to an excess risk.  Nevertheless, it remains unknown what is the fundamental limit of  clustering under the anisotropic GMMs and whether there is any polynomial-time  procedure that achieves it.
 
 %clustering under the anisotropic GMM is not well explored in literature.

In this paper, we will investigate the  optimal rates of the clustering task under two anisotropic GMMs. Model 1 is when the covariance matrices are all equal to each other (i.e., homogeneous) and are equal to some unknown matrix $\Sigma^*$. Model 2 is more flexible, where the covariance matrices are unknown and are not necessarily equal to each other (i.e., heterogeneous). The contribution of this paper is two-fold, summarized as follows. 

Our first contribution is on the fundamental limits. We obtain the minimax lower bound for clustering under the anisotropic GMMs with respect to the loss $h(\hat z,z^*)$. We show it takes the form %$\exp(-(1+o(1))(\text{signal-to-noise ratio})^2/8)$
\begin{align*}
\inf_{\hat z}\sup_{z^*\in[k]^n}\E h(z,z^*)\geq \ebr{-(1+o(1)) \frac{(\text{signal-to-noise ratio})^2}{8}},
\end{align*}
where the signal-to-noise ratio under  Model 1 is equal to $ \min_{a,b\in[k]:a\neq b} \|(\theta^*_a - \theta^*_b)^T \Sigma^{*-\frac{1}{2}}\|$ and the one for  Model 2 is more complicated. For both models, we can see the minimax rates depend not only on the centers but also the covariance matrices. This is different from the isotropic case whose signal-to-noise ratio is $\min_{a\neq b} \|\theta^*_a - \theta^*_b\|$. Our results precisely capture the role the covariance matrices play in the clustering problem. It shows covariance matrices impact the fundamental limits of the clustering problem through complicated interaction with the centers especially in  Model 2. The minimax lower bounds are obtained by establishing  connections with  Linear Discriminant Analysis (LDA) and  Quadratic Discriminant Analysis (QDA).

Our second and more important contribution is on the computational side. We propose a computationally feasible and rate-optimal algorithm for the anisotropic GMM. Popular methods including the Lloyd's algorithm and the spectral clustering no longer work well as they are developed under the isotropic case and only consider the distances among the centers \cite{brubaker2008isotropic}. We study an \emph{adjusted Lloyd's algorithm} which estimates the covariance matrices in each iteration and recovers the clusters using the covariance matrix information. It can also be seen as a hard EM algorithm \cite{dempster1977maximum}. As an iterative algorithm, we give  a statistical and computational guarantee and guidance to practitioners  by showing  that it obtains the minimax lower bound within $\log n$ iterations.
% which is a statistical and computational guarantee and a guidance to practitioners. 
That is, let $z^{(t)}$ be the output of the algorithm after $t$ iterations, we have with high probability,
\begin{align*}
h(z^{(t)},z^*)\leq \ebr{-(1+o(1)) \frac{(\text{signal-to-noise ratio})^2}{8}},
\end{align*}
holds for all $t\geq \log n$. The algorithm can be initialized by popular methods such as the spectral clustering or the Lloyd's algorithm. In  numeric studies, we show the proposed algorithm improves greatly from the two aforementioned methods under  anisotropic GMMs, and matches the optimal exponent given  in the minimax lower bound.

\paragraph{Paper Organization.} The remaining paper is organized as follows. In Section \ref{sec:homo}, we study  Model 1 where the covariance matrices are unknown but homogeneous. In Section \ref{sec:hetero}, we consider  Model 2 where covariance matrices are unknown and heterogeneous. For both cases, we obtain the minimax lower bound for clustering and study the adjusted Lloyd's algorithm. In Section \ref{sec:numeric}, we provide a numeric comparison with other popular methods. The proofs of theorems in Section \ref{sec:homo} are given in Section \ref{subsec:proof} and the proofs for Section \ref{sec:hetero} are included in Section \ref{subsect:proof.hetero}. All the technical lemmas are included in Section \ref{sec:technical}.

\paragraph{Notation.} Let $[m]=\{1,2,\ldots, m\}$ for any positive integer $m$. For any set $S$, we denote $\abs{S}$ for its cardinality.
For any matrix $X\in\mathr^{d\times d}$, we denote $\lambda_1(X)$ to be its smallest eigenvalue and $\lambda_d(X)$ to be its largest eigenvalue. In addition, we denote $\norm{X}$ to be its operator norm. For any two vectors $u,v$ of the same dimension, we denote $\iprod{u}{v}=u^Tv$ to be its inner product. For any positive integer $d$, we denote $I_d$ to be the $d\times d$ identity matrix. We denote $\mathcal{N}(\mu,\Sigma)$ to be the normal distribution with mean $\mu$ and covariance matrix $\Sigma$. We denote $\indic{\cdot}$ to be the indicator function. Given two positive sequences $a_n,b_n$, we denote $a_n =o(b_n)$ if $a_n/b_n=o(1)$ when $n\rightarrow\infty$. We write $a_n \lesssim b_n$ if there exists a constant $C>0$ independent of $n$ such that $a_n\leq C b_n$ for all $n$.

\section{GMM with Unknown but Homogeneous Covariance Matrices}
\label{sec:homo}

\subsection{Model}
We first consider a GMM where  covariance matrices of different clusters are unknown but are assumed to be equal to each other. 
%Let $k$ be the number of clusters which is assumed to be known. Denote $\{\theta^*_a\}_{a\in[k]}$ be the unknown centers and $\Sigma^*$ to be the unknown shared covariance matrix. Let $z^*\in[k]^n$ be the cluster assignment vector we want to estimate. For each index $j\in[n]$, the value of $z^*_j$ indicates which cluster the $j$th data point belongs to. 
The data generating progress can be displayed as follow:
%with observations denoted as $\{Y_j\}_{j\in[n]}$.
\begin{flalign}\label{eqn:model_I}
   \textbf{Model 1:} && Y_j  &= \theta^*_{z^*_j} + \epsilon_j, \text{ where }\epsilon_j \stackrel{ind}{\sim}  \mathn(0,\Sigma^*), \forall j\in[n]. &
\end{flalign}
It is called Stretched Mixture Model in \cite{wang2020efficient} as the density of $Y_j$ is elliptical.
Throughout the paper, we call it  \emph{Model 1} for simplicity and to distinguish it from a  more complicated model that will be introduced in Section \ref{sec:hetero}. The goal is to recover the underlying cluster assignment vector $z^*$ from $Y$.

\paragraph{Signal-to-noise Ratio.}
Define  the signal-to-noise ratio
\begin{align}\label{eqn:SNR_I}
\snr = \min_{a,b\in[k]:a\neq b} \|(\theta^*_a - \theta^*_b)^T \Sigma^{*-\frac{1}{2}}\|,
\end{align}
which is a function of all the centers $\{\theta^*_a\}_{a\in[k]}$ and the covariance matrix $\Sigma^*$. As we will show later in Theorem \ref{thm:lower1}, $\snr$ captures the difficulty of the clustering problem and determines the minimax rate. For the geometric interpretation of $\snr$, we defer it after presenting Theorem \ref{thm:main1}.

A quantity closely related to $\snr$ is the minimum distance among the centers. Define $\Delta$ as
\begin{align}\label{eqn:delta}
\Delta = \min_{a,b\in[k]:a\neq b}  \norm{\theta_a^* - \theta_b^*}.
\end{align}
Then we can see $\snr$ and $\Delta$ are in the same order if all eigenvalues of the covariance matrix $\Sigma^*$ is assumed to be constants. If $\Sigma^*$ is further assumed to be an identical matrix, then we have $\snr$ equal to $\Delta$. As a result, in \cite{lu2016statistical, gao2019iterative, loffler2019optimality} where  isotropic GMMs are studied, $\Delta$ plays the role of signal-to-noise ratio and appears in the minimax rates.

\paragraph{Loss Function.}
To measure the clustering performance, we consider the misclustering error rate defined as follows. For any $z,z^*\in [k]^n$, we define
\begin{align}\label{eqn:h}
h(z,z^*) =  \min_{\psi\in\Psi} \frac{1}{n}\sum_{j=1}^n \indic{\psi(z_j)\neq z^*_j},
\end{align}
where $\Psi = \cbr{\psi: \psi \text{ is a bijection from }[k]\text{ to }[k]}$. Here the minimum is over all the permutations over $[k]$ due to the identifiability issue of the labels $1,2,\ldots,k$. Another loss that will be used  is $\ell(z,z^*)$ defined as
\begin{align}\label{eqn:l}
\ell(z,z^*) = \sum_{j=1}^n \norm{\theta^*_{z_j} - \theta^*_{z^*_j}}^2.
\end{align}
It also measures the clustering performance of $z$ considering the distances among the true centers. It is related to $h(z,z^*)$ as $h(z,z^*)\leq \ell(z,z^*)/(n\Delta^2)$ and hence provides more information than $h(z,z^*)$. We will mainly use $\ell(z,z^*)$ in the technical analysis but will eventually present the results using $h(z,z^*)$ which is more interpretable.

%\begin{align*}
%\mathcal{Z}(\alpha,k) = \cbr{z\in[k]^n: \sum_{j=1}^n \indic{z_j^* =a }\geq \frac{\alpha n}{k},\forall  a \in[k]}.
%\end{align*}

%\begin{figure}[htbp]
%	\centering
%	\begin{minipage}[t]{0.5\linewidth}
%		\centering
%		\includegraphics[scale=0.55]{sec2_1.pdf}
%		\caption{\footnotesize Original contour plot}
%	\end{minipage}%
%	\begin{minipage}[t]{0.5\linewidth}
%		\centering
%		\includegraphics[scale=0.55]{sec2_2.pdf}
%		\caption{\footnotesize Transformed contour plot and \snr}
%	\end{minipage}%
%\end{figure}

\subsection{Minimax Lower Bound}\label{sec:1_lower}
We first establish the minimax lower bound for the clustering problem under  Model 1.
\begin{theorem}\label{thm:lower1}
Under the assumption $\frac{\snr}{\sqrt{\log k}} \rightarrow\infty$, we have 
\begin{align}\label{eqn:lower1}
\inf_{\hat z} \sup_{z^*\in [k]^n} \E h(z,z^*) \geq  \ebr{-(1+o(1))\frac{\snr^2}{8}}.
\end{align}
If $\snr=O(1)$ instead, we have $\inf_{\hat z} \sup_{z^*\in[k]^n} \E h(z,z^*) \geq c$ for some constant $c>0$.
\end{theorem}

Theorem \ref{thm:lower1} allows the cluster numbers $k$ to grow with $n$ and shows that $\snr\rightarrow\infty$ is a necessary condition to have a consistent clustering if $k$ is a constant.
%, then the condition can be reduced to $\snr \rightarrow \infty$.
Theorem \ref{thm:lower1} holds for any arbitrary $\{\theta^*_a\}_{a\in[k]}$ and $\Sigma^*$, and the minimax lower bound depend on them through $\snr$. The parameter space is only for $z^*$ while $\{\theta^*_a\}_{a\in[k]}$ and $\Sigma^*$ are fixed. Hence,  (\ref{eqn:lower1}) can be interpreted as a pointwise result, and it captures precisely the explicit dependence of the minimaxity on  $\{\theta^*_a\}_{a\in[k]}$ and $\Sigma^*$.

Theorem \ref{thm:lower1} is closely related to the Linear Discriminant Analysis (LDA). If there are only two clusters, and if the centers and the covariance matrix are known, then estimating each $z^*_j$ is exactly the task of LDA: we want to figure out which normal distribution an observation $Y_j$ is generated from, where the two normal distributions have different means but the same covariance matrix.  In fact, this is also how Theorem \ref{thm:lower1} is proved: we will first reduce the estimation problem of $z^*$ into a two-point hypothesis testing problem for each individual $z^*_j$, the error of which is given   in Lemma \ref{lem:lda} by the analysis of LDA, and then aggregate all the testing errors together.

In the following lemma, we give a sharp and explicit formula for the testing error of the LDA. Here we have two normal distributions $\mathn(\theta^*_1,\Sigma^*)$ and $\mathn(\theta^*_2,\Sigma^*)$ and an observation $X$ that is generated from one of them. We are interested in estimating which distribution it is from. By Neyman-Pearson lemma, it is known that the likelihood ratio test $\indic{2(\theta^*_2  - \theta^*_1)^T (\Sigma^*)^{-1}X \geq \theta_2^{*T} (\Sigma^*)^{-1}\theta^*_2- \theta^{*T}_1( \Sigma^*)^{-1} \theta^*_1}$ is the optimal testing procedure. Then by using the Gaussian tail probability, we are able to obtain the optimal testing error, the lower bound of which is given in Lemma  \ref{lem:lda}.

 %On the contrary, the  result established in \cite{lu2016statistical}

 %bound the impact of the cluster structure on the 

%which is different from the results established in \cite{lu2016statistical}, which 

%\paragraph{Linear Discriminant Analysis}

\begin{lemma}[Testing Error for Linear Discriminant Analysis]\label{lem:lda}
Consider two hypotheses $\mathbb{H}_0: X\sim \mathn(\theta_1^*,\Sigma^*)$ and $\mathbb{H}_1: X \sim \mathn(\theta_2^*,\Sigma^*)$. Define a testing procedure $$\phi =\indic{2(\theta^*_2  - \theta^*_1)^T (\Sigma^*)^{-1}X \geq \theta_2^{*T} (\Sigma^*)^{-1}\theta^*_2- \theta^{*T}_1( \Sigma^*)^{-1} \theta^*_1}.$$
%\begin{align*}
%\phi = \begin{cases}
%1, \text{ if }\; 2(\theta_b  - \theta_a)^T \Sigma^{-1}X \geq \theta_b^T \Sigma^{-1}\theta_b - \theta_a^T \Sigma^{-1} \theta_a\\
%0, \text{ o.w..}
%\end{cases}
%\end{align*}
Then we have
$
\inf_{\hat \phi} (\p_{\mathbb{H}_0} (\hat \phi = 1) +\p_{\mathbb{H}_1} (\hat \phi = 0)) =\p_{\mathbb{H}_0} \br{ \phi = 1} +\p_{\mathbb{H}_1} \br{ \phi = 0}.
$
If $\normp{(\theta^*_2 - \theta^*_1)^T (\Sigma^*)^{-\frac{1}{2}}} \rightarrow \infty$, we have 
\begin{align*}
\inf_{\hat \phi} (\p_{\mathbb{H}_0} (\hat \phi = 1) +\p_{\mathbb{H}_1} (\hat \phi = 0)) \geq \ebr{-(1+o(1))\frac{ \normp{(\theta^*_2 - \theta^*_1)^T (\Sigma^*)^{-\frac{1}{2}}}^2}{8}}.
\end{align*}
Otherwise, $\inf_{\hat \phi} (\p_{\mathbb{H}_0} (\hat \phi = 1) +\p_{\mathbb{H}_1} (\hat \phi = 0)) \geq c$ for some constant $c>0$.
\end{lemma}
%where $\eta \sim \mathn(0,1)$. If in additional $\norm{(\theta_b - \theta_a)^T \Sigma^{-\frac{1}{2}}} \rightarrow \infty$
%\begin{proof}
%let $\eta \sim\mathn(0,1)$  $= \pbr{\eta \geq  \frac{1}{2}\norm{(\theta_b - \theta_a)^T \Sigma^{-\frac{1}{2}}}}$.
%\end{proof}

\begin{figure}[ht]
\centering
\includegraphics[width=0.45\textwidth, trim = 16 40 0 50, clip]{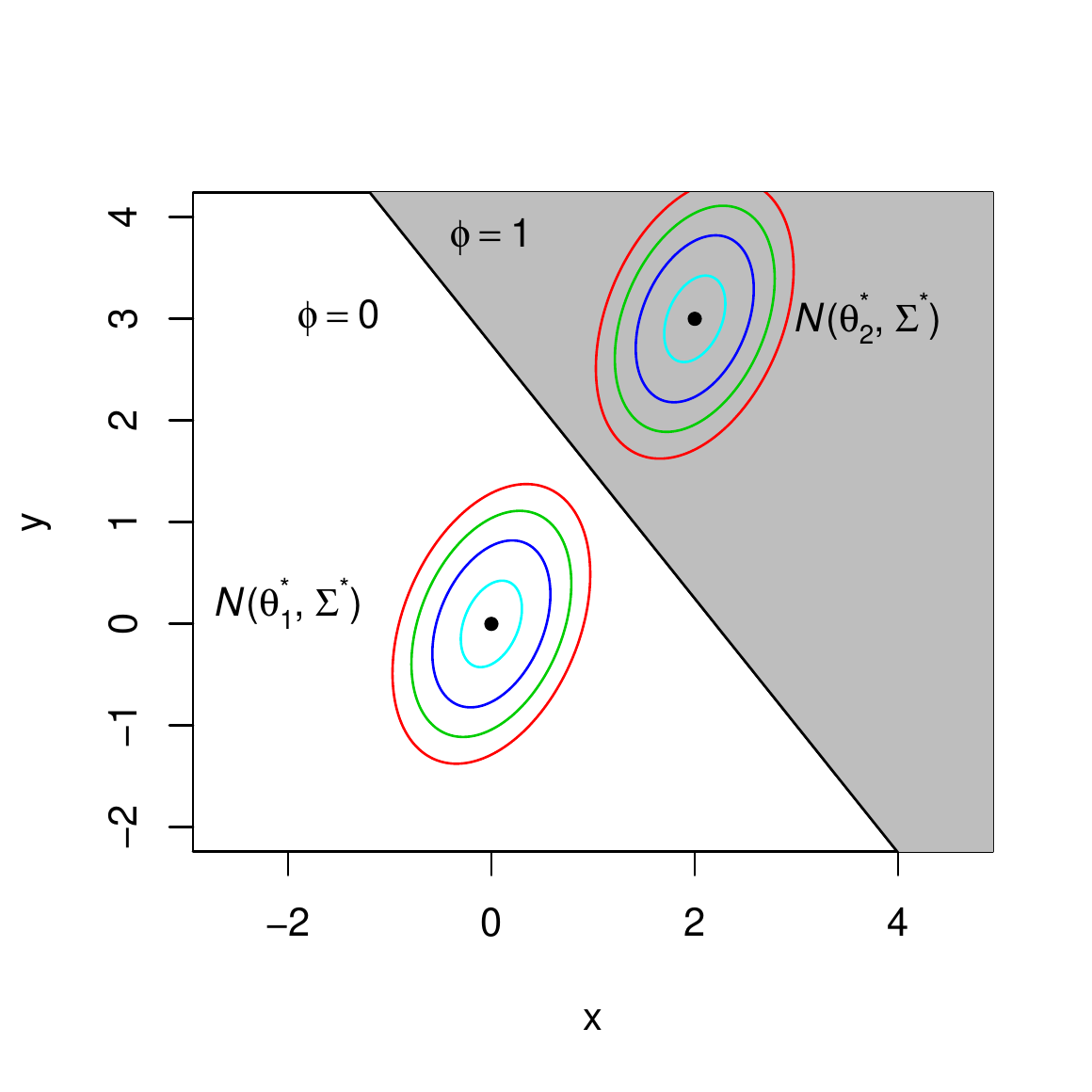}
\includegraphics[width=0.45\textwidth, trim = 16 40 0 50, clip]{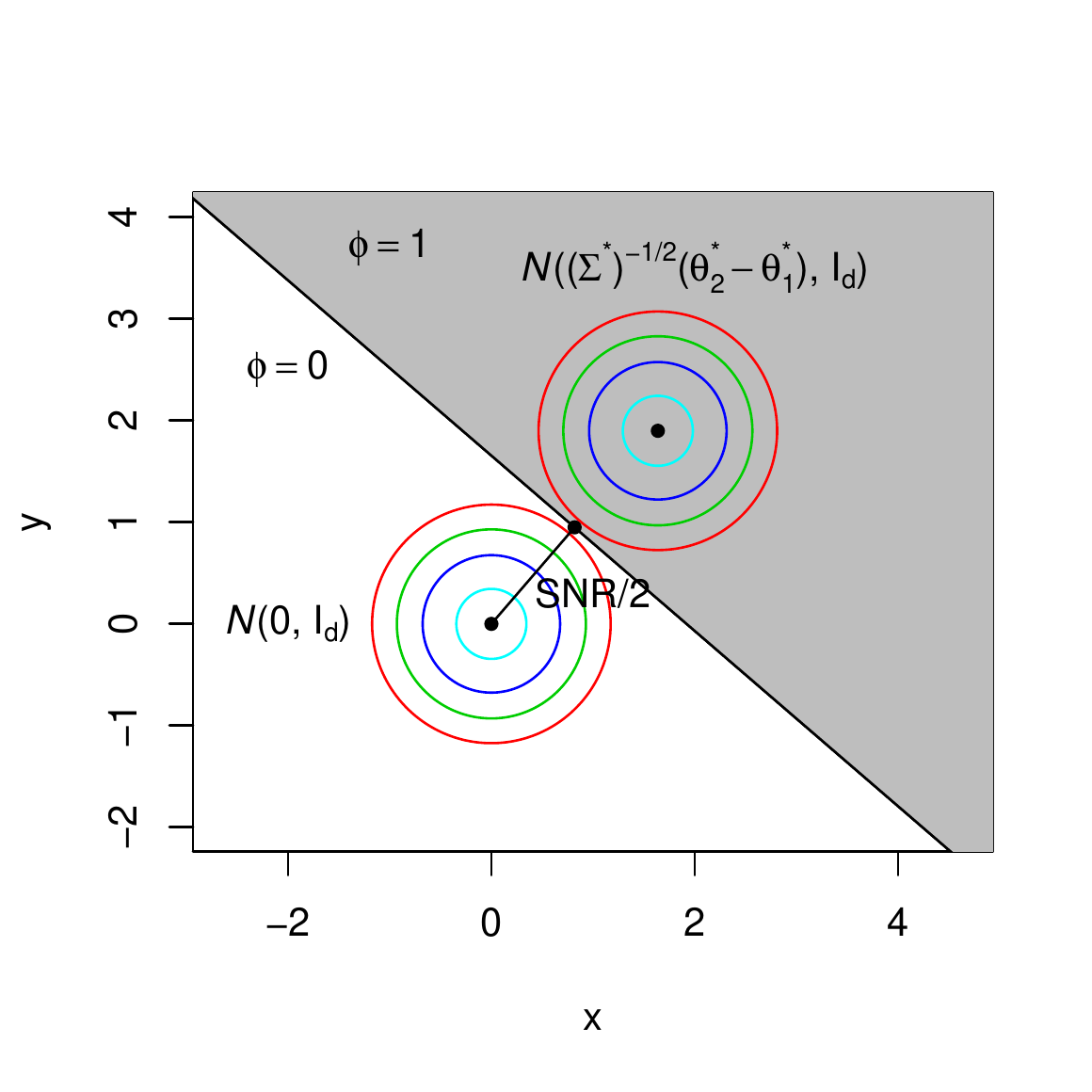}
\caption{A geometric interpretation of  \snr.  \label{fig:1}}
\end{figure}

With the help of Lemma \ref{lem:lda}, we have a geometric interpretation of $\snr$. In the left panel of Figure \ref{fig:1}, we have two normal distributions $\mathn(\theta^*_1,\Sigma^*)$ and $\mathn(\theta^*_2,\Sigma^*)$ for $X$ to be generated from.  The black line represents the optimal testing procedure $\phi$ displayed in Lemma  \ref{lem:lda} that divides the space into two half-spaces. To calculate the testing error, we can make a transformation $X' = (\Sigma^*)^{-\frac{1}{2}}(X-\theta_1^*)$ so that the two normal distributions become isotropic: $\mathn(0,I_d)$ and $\mathcal{N}((\Sigma^*)^{-\frac{1}{2}}(\theta^*_2 - \theta^*_1),I_d)$ as displayed in the right panel. Then the distance between the two centers are $\|(\Sigma^*)^{-\frac{1}{2}}(\theta^*_2 - \theta^*_1)\|$, and the distance between a center and the black curve is  half of it. Then $\p_{H_0} (\hat \phi = 1)$ is  the probability of $\mathn(0,I_d)$  in the grayed area, which is equal to $\exp(-(1+o(1))\|(\Sigma^*)^{-\frac{1}{2}}(\theta^*_2 - \theta^*_1)\|^2/8)$ by Gaussian tail probability.
As a result, $\|(\Sigma^*)^{-\frac{1}{2}}(\theta^*_2 - \theta^*_1)\|$ is the effective distance between the two centers of $\mathn(\theta^*_1,\Sigma^*)$ and $\mathn(\theta^*_2,\Sigma^*)$ for the clustering problem, considering the geometry of the covariance matrix. Since we have multiple clusters, $\snr$ defined in (\ref{eqn:SNR_I}) can be interpreted as  the minimum effective distances among the centers $\{\theta^*_a\}_{a\in[k]}$ considering the anisotropic structure of $\Sigma^*$  and it captures the intrinsic difficulty of the clustering problem.

\subsection{Rate-Optimal Adaptive Procedure}\label{sec:model1_alg}

In this section, we will propose a computationally feasible and rate-optimal procedure for clustering under  Model 1. Summarized in Algorithm \ref{alg:1}, the proposed algorithm is a variant of the Lloyd algorithm. Starting from some initialization, it updates the estimation of the centers $\{\theta^*_a\}_{a\in[k]}$ (in (\ref{eqn:1_theta})), the covariance matrix $\Sigma^*$ (in (\ref{eqn:1_Sigma})), and the cluster assignment vector $z^*$ (in (\ref{eqn:1_z})) iteratively. It differs from the Lloyd's algorithm in the sense that Lloyd's algorithm is for isotropic GMMs without the covariance matrix update (\ref{eqn:1_Sigma}). In addition, in (\ref{eqn:1_z}) it updates the estimation of $z^*_j$  by $ \argmin_{a\in[k]} (Y_j -  \theta_a^{(t)} )^T(Y_j -  \theta_a^{(t)} )$ instead. To distinguish them from each other, we call the classical Lloyd's algorithm as the \emph{vanilla Lloyd's algorithm}, and name Algorithm \ref{alg:1} as the \emph{adjusted Lloyd's algorithm}, as it is adjusted to the unknown and anisotropic covariance matrix.

Algorithm \ref{alg:1} can also be interpreted as a hard EM algorithm. If we apply the Expectation Maximization (EM) for  Model 1, we will have an M step for estimating parameters $\{\theta^*_a\}_{a\in[k]}$ and $\Sigma^*$ and an E step for estimating $z^*$.
It turns out the updates on the parameters (\ref{eqn:1_theta}) - (\ref{eqn:1_Sigma})  are exactly the same as the updates of EM (M step). However, the update on $z^*$ in Algorithm \ref{alg:1} is different from that in the EM. Instead of taking a conditional expectation (E step), we also take a maximization in (\ref{eqn:1_z}). As a result, Algorithm \ref{alg:1} consists solely of M steps for both the parameters and $z^*$, which is known as a hard EM algorithm.

%It serves a greedy algorithm to achieve the global maximizer
%\begin{align*}
%(\{\hat \theta_a\}_{a\in[k]},\hat\Sigma,\hat z) = \argmax_{\{ \theta_a\}_{a\in[k]},\Sigma,z} - 
%\end{align*}

\begin{algorithm}[ht]
\KwIn{Data $Y$, number of clusters $k$, an initialization $z^{(0)}$, number of iterations $T$}
\KwOut{$z^{(T)}$}
\For{ $t=1,\ldots,T$}{

Update the centers:
 \begin{align}\label{eqn:1_theta}
 \theta_a^{(t)} = \frac{\sum_{j\in[n]} Y_j\indic{z^{(t-1)}_j=a}}{\sum_{j\in[n]}\indic{z^{(t-1)}_j=a}},\quad \forall a\in[k]. \;
 \end{align}
 
Update the covariance matrix:
\begin{align}\label{eqn:1_Sigma}
 \Sigma^{(t)}  =  \frac{ \sum_{a\in[k]}\sum_{j\in[n]} (Y_j -   \theta_a^{(t)}) (Y_j -   \theta_a^{(t)})^T \indic{z^{(t-1)}_j=a}}{n}. \;
\end{align}

Update the cluster estimations:
\begin{align}\label{eqn:1_z}
z^{(t)}_j = \argmin_{a\in[k]} (Y_j -  \theta_a^{(t)} )^T (\Sigma^{(t)})^{-1}(Y_j -  \theta_a^{(t)} ),\quad j\in[n].
\end{align}
}
\caption{Adjusted Lloyd's Algorithm for Model 1 (\ref{eqn:model_I}). \label{alg:1}}
\end{algorithm}

In Theorem \ref{thm:main1}, we give a computational and statistical guarantee of the proposed Algorithm \ref{alg:1}. We show that starting from a decent initialization, within $\log n$ iterations, Algorithm \ref{alg:1} achieves an error rate $\ebr{-(1+o(1))\snr^2/8}$ which matches with the minimax lower bound  given in Theorem \ref{thm:lower1}. As a result, Algorithm \ref{alg:1} is a rate-optimal procedure. In addition, the algorithm is fully adaptive to the unknown $\{\theta^*_a\}_{a\in[k]}$ and $\Sigma^*$. The only information assumed to be known is $k$ the number of clusters, which is commonly assumed to be known in clustering literature \cite{lu2016statistical, gao2019iterative, loffler2019optimality}. The theorem also shows that the number of iterations to achieve the optimal rate is at most $\log n$, which provides  implementation guidance to practitioners.

%Denote $\alpha$ to be a quantify for the smallest cluster sizes such that $\min_{a\in k}\sum_{j=1}^n\mathbb{I}\{z^*_j = a\}\geq  \frac{n}{k}$. We have the following upper bound for $z^{(t)}$ from Algorithm \ref{alg:1}.

\begin{theorem}\label{thm:main1}
Assume $kd=O(\sqrt{n})$ and $\min_{a\in k}\sum_{j=1}^n\mathbb{I}\{z^*_j = a\}\geq  \frac{\alpha n}{k}$ for some constant $\alpha>0$. Assume $\frac{\snr}{ k}\rightarrow\infty$ and
%Assume $\snr \rightarrow\infty, k=O(1), d=O(\sqrt{n})$,  and $\lambdamin \leq \lambda_1(\Sigma^*) \leq \lambda_d(\Sigma^*) \leq \lambdamax$ where $\lambdamin,\lambdamax$ are constants. 
$\lambda_{d}(\Sigma^*) / \lambda_{1}(\Sigma^*)=O(1)$.
For Algorithm \ref{alg:1}, suppose $z^{(0)}$ satisfies $\ell(z^{(0)},z^*) = o(n/k)$ with probability at least $1-\eta$. Then with probability at least $1-\eta - n^{-1} - \exp(-\snr)$, we have
\begin{align*}
h(z^{(t)},z^*) \leq \ebr{-(1+o(1))\frac{\snr^2}{8}},\quad \text{for all }t\geq \log n.
\end{align*}
\end{theorem}

We have remarks on the assumptions of
 Theorem \ref{thm:main1}. We allow the number of clusters $k$ to grow with $n$.
 % and the cluster sizes not necessarily in the order of $n/k$. 
 When  $k$ is a constant, the assumption on $\snr\rightarrow\infty$ is the necessary  condition to have a consistent recovery of $z^*$ according to the minimax lower bound presented in Theorem \ref{thm:lower1}. The assumption on $\Sigma^*$ is to make sure the covariance matrix  is well-conditioned. The dimensionality $d$ is assumed to be at most $O(\sqrt{n})$, an assumption that is stronger than that in \cite{lu2016statistical, gao2019iterative, loffler2019optimality} which only needs $d=O(n)$. This is due to that compared to these papers, we need to estimate the covariance matrix $\Sigma^*$ and to have a control on the estimation error $\|\Sigma^{(t)} - \Sigma^*\|$.
 
 % From the minimax lower bound presented in Theorem \ref{thm:lower1}, $\snr \rightarrow\infty$ is a necessary  condition to have a consistent recovery of $z^*$. The assumption on $\Sigma^*$ is to make sure all the eigenvalues of the covariance matrix are in the constant order. Since the parameters can be rescaled, this essentially means the covariance matrix is well-conditioned.
 %the covariance matrix $\Sigma^*$ is well-conditioned. Since the parameters can be rescaled, the eigenvalues of $\Sigma^*$ is not necessarily assumed to be constants. 
% In fact, we can assume $\lambda_{d}(\Sigma^*) / \lambda_{1}(\Sigma^*)=O(1)$ instead
 
% The cluster number $k$ is assumed to be a constant and the cluster sizes are assumed to be in the same order, which is the most common case in practise. The dimensionality $d$ is assumed to be $O(\sqrt{n})$, an assumption that is stronger than that in \cite{lu2016statistical, gao2019iterative, loffler2019optimality} which only needs $d=O(n)$. This is due to that compared to these papers, we need to estimate the covariance matrix $\Sigma^*$ and to have a control on the estimation error $\|\Sigma^{(t)} - \Sigma^*\|$.

The requirement for the initialization $\ell(z^{(0)},z^*) = o(n/k)$ can be fulfilled by simple procedures. A popular choice is the vanilla Lloyd's algorithm the performance of which is studied in \cite{lu2016statistical, gao2019iterative}. Since  $\epsilon_j$ are sub-Gaussian random variables with proxy variance $\lambdamax$, \cite{gao2019iterative} implies the vanilla Lloyd's algorithm output $\hat z$ satisfies $\ell(\hat z,z^*)\leq n\exp(-(1+o(1))\Delta^2/(8\lambdamax))$ with probability at least $1-\exp(-\Delta) -n^{-1}$, under the assumption that $\snr/k\rightarrow\infty$. Note that   \cite{gao2019iterative} is for isotropic GMMs, but its results can be extended to sub-Gaussian mixture models with nearly identical proof.  Then we have $\ell(\hat z,z^*) =o(n/k)$, as $\Delta^2/\lambdamax$ and $\snr^2$ are both in the same order under the assumption $\snr/k\rightarrow\infty$.
%A popular choice is the spectral clustering. For instance, we can use a variant of spectral clustering studied in \cite{loffler2019optimality}. Since Model 1 can be seen as a sub-Gaussian mixture model,
% in their Proposition D.1, they show the spectral clustering output $\hat z^\spec$ achieves
%\begin{align}\label{eqn:spectral}
%h(\hat z^\spec,z^*) =O\br{\frac{k}{\snr^2}},
%\end{align}
%with probability at least $1-\ebr{-0.08n}$, under the same assumption as in Theorem \ref{thm:main1}. Then it can be used as an initialization in Theorem \ref{thm:main1} if we pose a slightly stronger assumption $\snr/k\rightarrow\infty$. 
As a result, we immediately have the following corollary.
\begin{corollary}\label{cor:1}
Assume $kd=O(\sqrt{n})$ and $\min_{a\in k}\sum_{j=1}^n\mathbb{I}\{z^*_j = a\}\geq  \frac{\alpha n}{k}$ for some constant $\alpha>0$. Assume $\frac{\snr}{k}\rightarrow\infty$ and
$\lambda_{d}(\Sigma^*) / \lambda_{1}(\Sigma^*)=O(1)$. Using the vanilla Lloyd's algorithm as the initialization $z^{(0)}$ in Algorithm \ref{alg:1}, we have with probability at least $1- n^{-1} - \exp(-\snr)-\exp(-\Delta)$,
\begin{align*}
h(z^{(t)},z^*) \leq \ebr{-(1+o(1))\frac{\snr^2}{8}},\quad \text{for all }t\geq \log n.
\end{align*}
\end{corollary}

\section{GMM with Unknown and Heterogeneous Covariance Matrices}\label{sec:hetero}

\subsection{Model}

In this section, we study the GMM with covariance matrices from each cluster unknown and not  necessarily equal to each other. The data generation process can be displayed as follow,
\begin{flalign}\label{eqn:model_II}
   \textbf{Model 2:} && Y_j  &= \theta^*_{z^*_j} + \epsilon_j, \text{ where }\epsilon_j \stackrel{ind}{\sim}  \mathn(0,\Sigma^*_{z^*_j}), \forall j\in[n].&
\end{flalign}
We call it  \emph{Model 2} throughout the paper to distinguish it from  Model 1 studied in Section \ref{sec:homo}. The  difference between (\ref{eqn:model_II}) and (\ref{eqn:model_I}) is that we now have $\{\Sigma^*_a\}_{a\in[k]}$ instead of a shared $\Sigma^*$. We consider  the same loss functions  as in (\ref{eqn:h}) and (\ref{eqn:l}).

\paragraph{Signal-to-noise Ratio.}
The signal-to-noise ratio for  Model 2 is defined as follows. We use the notation $\snr'$ to distinguish it from  $\snr$ for  Model 1. Compared to $\snr$, $\snr'$ is much more complicated and  does not have an explicit formula. We first define a space $\mathcal{B}_{a,b} \in \mathr^d$ for any $a,b\in [k]$ such that $a\neq b$:
\begin{align*}
\mathcal{B}_{a,b} = \Bigg\{x\in \mathr^{d}: &x^T \Sigma_a^{*\frac{1}{2}} \Sigma_b^{*-1}(
\theta^*_a -\theta^*_b) + \frac{1}{2} x^T\br{\Sigma_a^{*\frac{1}{2}}\Sigma_b^{*-1}\Sigma_a^{*\frac{1}{2}}-I_d}x  \\
&\leq -\frac{1}{2}(
\theta^*_a -\theta^*_b)^T\Sigma_b^{*-1}(
\theta^*_a -\theta^*_b) + \frac{1}{2}\log \abs{\Sigma_a^*}- \frac{1}{2}\log \abs{\Sigma_b^*} \Bigg\}.
\end{align*}
We then define $\snr'_{a,b} = 2\min_{x\in \mathcal{B}_{a,b}} \norm{x}$ and 
\begin{align}
\snr' = \min_{a,b\in[n]:a\neq b} \snr'_{a,b}.\label{eqn:SNR_II}
\end{align}

The from of $\snr'$ is closely connected to the testing error of the Quadratic Discriminant Analysis (QDA), which we will give in Lemma \ref{lem:qda}. For the  interpretation of the $\snr'$ (especially from a geometric point of view), we defer it after presenting   Lemma \ref{lem:qda}. Here let us consider a few special cases where we are able to simplify $\snr'$:
(1) When $\Sigma^*_a = \Sigma^*$ for all $a\in[k]$, by simple algebra, we have $\snr'_{a,b} =  \|(\theta^*_a - \theta^*_b)^T \Sigma^{*-\frac{1}{2}}\|$ for any $a,b\in[k]$ such that $a\neq b$. Hence, $\snr'=\snr$ and  Model 2 is reduced to the 
%. As a result, the  Model I (\ref{eqn:SNR_I}) is a special case of the Model II 
Model 1. (2) When $\Sigma^* = \sigma_a^2 I_d$ for any $a\in[k]$ where $\sigma_1,\ldots,\sigma_k>0$ are large constants, we have $\snr'_{a,b} ,\snr'_{b,a} $ both close to $2\|\theta^*_a - \theta^*_b\|/(\sigma_a + \sigma_b)$. From these examples, we can see $\snr'$ is determined by both the centers $\{\theta^*_a\}_{a\in[k]}$ and the covariance matrices $\{\Sigma^*_a\}_{a\in[k]}$.

\subsection{Minimax Lower Bound}\label{sec:3.2}

We first establish the minimax lower bound for the clustering problem under  Model 2.
\begin{theorem}\label{thm:lower2}
Under the assumption $\frac{\snr'}{\sqrt{\log k}} \rightarrow\infty$, we have 
\begin{align*}
\inf_{\hat z} \sup_{z^*\in [k]^n} \E h(z,z^*) \geq  \ebr{-(1+o(1))\frac{\snr^{'2}}{8}}.
\end{align*}
If $\snr'=O(1)$ instead, we have $\inf_{\hat z} \sup_{z^*\in[k]^n} \E h(z,z^*) \geq c$ for some constant $c>0$.
\end{theorem}
Despite that the statement of Theorem \ref{thm:lower2} looks similar to that of Theorem \ref{thm:lower1}, the two minimax lower bounds are different from each other due to the discrepancy in the dependence of the centers and the covariance matrices in  $\snr'$ and $\snr$. By the same argument as in Section \ref{sec:1_lower}, the minimax lower bound established in Theorem  \ref{thm:lower2} is closely related to the Quadratic Discriminant Analysis (QDA) between two normal distributions with different means and different covariance matrices.

\begin{lemma}[Testing Error for Quadratic Discriminant Analysis]\label{lem:qda}
Consider two hypotheses $\mathbb{H}_0: X\sim \mathn(\theta_1^*,\Sigma_1^*)$ and $\mathbb{H}_1: X \sim \mathn(\theta_2^*,\Sigma_2^*)$. Define a testing procedure $$\phi =\indic{\log \abs{\Sigma_1^*} + (x - \theta_1^*)^T\Sigma_1^* (x-\theta_1^*) \geq \log \abs{\Sigma_2^*} + (x - \theta_2^*)^T\Sigma_2^* (x-\theta_2^*) }.$$
Then we have
$
\inf_{\hat \phi} (\p_{\mathbb{H}_0} (\hat \phi = 1) +\p_{\mathbb{H}_1} (\hat \phi = 0))  =\p_{\mathbb{H}_0} \br{ \phi = 1} +\p_{\mathbb{H}_1} \br{ \phi = 0}.
$
If $\min\cbr{\snr'_{1,2},\snr'_{2,1}}\rightarrow\infty$, we have
\begin{align*}
\inf_{\hat \phi} (\p_{H_0} (\hat \phi = 1) +\p_{H_1} (\hat \phi = 0)) \geq \ebr{-(1+o(1)) \frac{\min\cbr{\snr'_{1,2},\snr'_{2,1}}^2}{8}}.
\end{align*}
Otherwise, $\inf_{\hat \phi} (\p_{H_0} (\hat \phi = 1) +\p_{H_1} (\hat \phi = 0)) \geq c$ for some constant $c>0$.
\end{lemma}

\begin{figure}[ht]
	\centering
		\includegraphics[width=0.45\textwidth, trim = 0 40 0 50, clip]{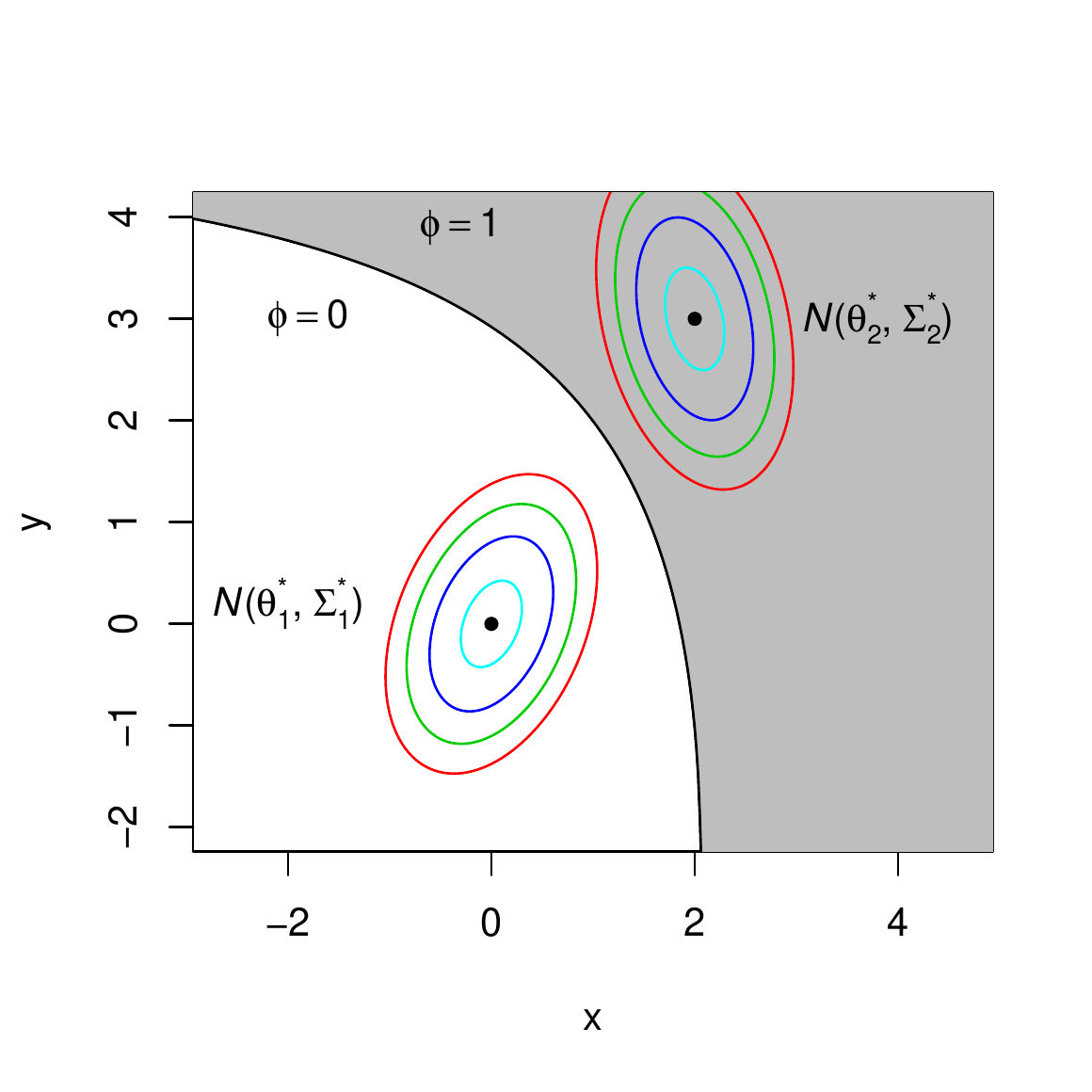}
		\includegraphics[width=0.45\textwidth, trim = 0 40 0 50, clip]{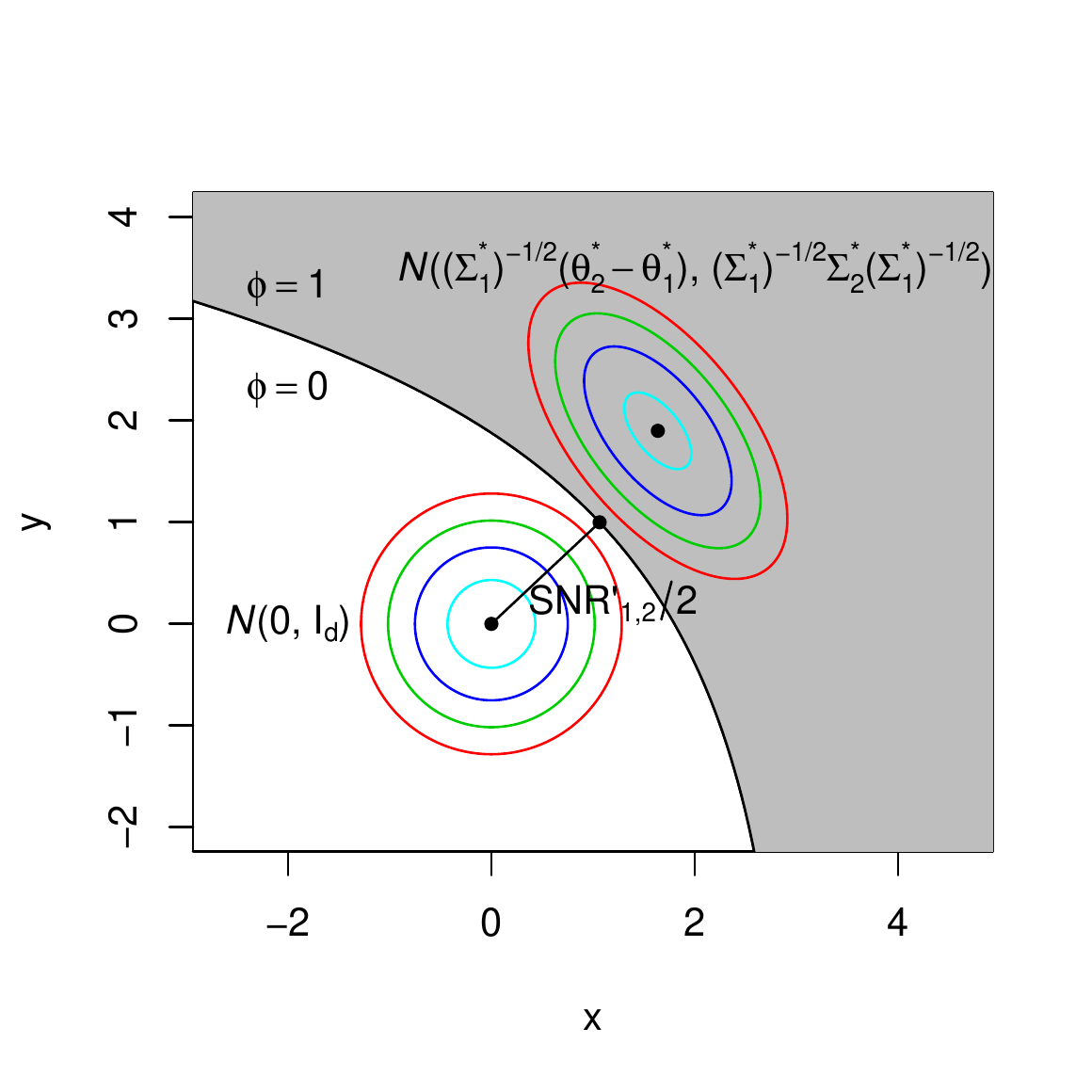}
		\caption{A geometric interpretation of  $\snr'$.  \label{fig:2}}
\end{figure}

From Lemma \ref{lem:qda}, we can have a geometric interpretation of $\snr'$. In the left panel of Figure \ref{fig:2}, we have two normal distributions $\mathn(\theta^*_1,\Sigma_1^*)$ and $\mathn(\theta^*_2,\Sigma_2^*)$ where $X$ can be generated from. The black curve represents the optimal testing procedure $\phi$ displayed in Lemma  \ref{lem:qda}. Since $\Sigma_1^*$ is not necessarily  equal to $\Sigma_2^*$, the black curve is not necessarily a straight line. If $\mathbb{H}_0$ is true, then the probability for $X$ to be incorrectly classified is when $X$ falls in the gray area, which is $\p_{H_0} (\hat \phi = 1)$. To calculate it, we can make a transformation $X'= (\Sigma^*_1)^{-\frac{1}{2}}(X - \theta^*_1)$. Then displayed in the right panel of Figure \ref{fig:2}, the two distributions become $\mathn(0,I_d)$ and $\mathn((\Sigma^*_1)^{-\frac{1}{2}}(\theta^*_2 - \theta^*_1),(\Sigma^*_1)^{-\frac{1}{2}} \Sigma^*_2 (\Sigma^*_1)^{-\frac{1}{2}})$, and the optimal testing procedure  $\indic{X'\in \mathcal{B}_{1,2}}$. As a result, in the right panel of Figure \ref{fig:2}, $\mathcal{B}_{1,2}$ represents the space colored by gray, and the black curve is its boundary. Then $\p_{H_0} (\hat \phi = 1)$ is equal to $\pbr{\mathn(0,I_d)\in \mathcal{B}_{1,2}}$, which can be shown to be determined by the minimum distance between the center of $\mathn(0,I_d)$ and the space $\mathcal{B}_{1,2}$. Denote the minimum distance by $\snr_{1,2}'/2$, by Lemmas \ref{lem:1} and Lemma \ref{lem:2}, we can show $\pbr{\mathn(0,I_d)\in \mathcal{B}_{1,2}} = \exp(-(1+o(1))\snr_{1,2}^{'2}/8)$. As a result, the $\snr'$ can be interpreted as the minimum effective distance among the centers $\{\theta^*_a\}_{a\in[k]}$ considering the anisotropic and heterogeneous structure of $\{\Sigma^*_a\}_{a\in[k]}$  and it captures the intrinsic difficulty of the clustering problem under  Model 2.

\subsection{Optimal Adaptive Procedure}

In this section, we will propose a computationally feasible and rate-optimal procedure for clustering under  Model 2. Similar to  Algorithm \ref{alg:1}, the proposed Algorithm \ref{alg:2} can be seen as a variant of the Lloyd's algorithm that is adjusted to the unknown and heterogeneous covariance matrices. It can also be interpreted as a hard EM algorithm under  Model 2. Algorithm \ref{alg:2}  differs from Algorithm \ref{alg:1} in (\ref{eqn:2_Sigma}) and (\ref{eqn:2_z}),  as now there are $k$ covariance matrices to be estimated. 

\begin{algorithm}[ht]
\SetAlgoLined
\KwIn{Data $Y$, number of clusters $k$, an initialization $z^{(0)}$, number of iterations $T$}
\KwOut{$z^{(T)}$}
\For{ $t=1,\ldots,T$}{
 
 Update the centers:
 \begin{align}\label{eqn:2_theta}
 \theta_a^{(t)} = \frac{\sum_{j\in[n]} Y_j\indic{z^{(t-1)}_j=a}}{\sum_{j\in[n]}\indic{z^{(t-1)}_j=a}},\quad \forall a\in[k]. \;
 \end{align}

Update the covariance matrices:
\begin{align}\label{eqn:2_Sigma}
 \Sigma_a^{(t)}  =  \frac{\sum_{j\in[n]} (Y_j -   \theta_a^{(t)}) (Y_j -   \theta_a^{(t)})^T \indic{z^{(t-1)}_j=a}}{\sum_{j\in[n]}\indic{z^{(t-1)}_j=a}},\quad \forall a\in[k]. \;
\end{align}

Update the cluster estimations:
\begin{align}\label{eqn:2_z}
z^{(t)}_j = \argmin_{a\in[k]} (Y_j -  \theta_a^{(t)} )^T (\Sigma^{(t)}_a)^{-1}(Y_j -  \theta_a^{(t)} ) + \log |\Sigma^{(t)}_a|,\quad j\in[n].
\end{align}
}
\caption{Adjusted Lloyd's Algorithm for Model 2 (\ref{eqn:model_II}). \label{alg:2}}
\end{algorithm}

In Theorem \ref{thm:main2}, we give a computational and statistical guarantee of the proposed Algorithm \ref{alg:2}. We show that provided with some decent initialization, Algorithm \ref{alg:2} is able to achieve the minimax lower bound within $\log n$ iterations. 
The  assumptions needed in Theorem \ref{thm:main2} are similar to those in Theorem \ref{thm:main2}, except that we require stronger assumptions on $k$ and the dimensionality $d$ since now we have $k$ (instead of one) covariance matrices to be estimated. In addition, by $\max_{a,b\in[k]}\lambda_{d}(\Sigma^*_a) / \lambda_{1}(\Sigma^*_b)=O(1)$  we not only assume each of the $k$ covariance matrices is well-conditioned, but also assume they are comparable to each other.

\begin{theorem}\label{thm:main2}
Assume $k,d=O(1)$ and $\min_{a\in k}\sum_{j=1}^n\mathbb{I}\{z^*_j = a\}\geq  \frac{\alpha n}{k}$ for some constant $\alpha>0$. Assume $\snr'\rightarrow\infty$ and $\max_{a,b\in[k]}\lambda_{d}(\Sigma^*_a) / \lambda_{1}(\Sigma^*_b)=O(1)$.
%Assume $\snr \rightarrow\infty, k=O(1), d=O(1)$, $\min_{a\in k}\sum_{j=1}^n\indic{z^*_j = a}\geq \frac{\alpha n}{k}$ for some constant $\alpha>0$, and $\lambdamin \leq \lambda_1(\Sigma^*) \leq \lambda_d(\Sigma^*) \leq \lambdamax$ where $\lambdamin,\lambdamax$ are constants. 
For Algorithm \ref{alg:2} , suppose $z^{(0)}$ satisfies $\ell(z^{(0)},z^*) = o(n/k)$ with probability at least $1-\eta$. Then with probability at least $1-\eta - n^{-1} - \exp(-\snr')$, we have
\begin{align*}
h(z^{(t)},z^*) \leq \ebr{-(1+o(1))\frac{\snr^{'2}}{8}},\quad \text{for all }t\geq \log n.
\end{align*}
\end{theorem}

The vanilla Lloyd's algorithm can be used as the initialization for Algorithm \ref{alg:2}. Under the assumption that $\lambda_{d}(\Sigma^*) / \lambda_{1}(\Sigma^*)=O(1)$,  Model 2 is also a sub-Gaussian mixture model. By the same argument as in Section  \ref{sec:model1_alg} we  have the following corollary.

\begin{corollary}\label{cor:2}
Assume $k,d=O(1)$ and $\min_{a\in k}\sum_{j=1}^n\mathbb{I}\{z^*_j = a\}\geq  \frac{\alpha n}{k}$ for some constant $\alpha>0$. Assume $\snr'\rightarrow\infty$ and $\lambda_{d}(\Sigma^*) / \lambda_{1}(\Sigma^*)=O(1)$.
Using the vanilla Lloyd's algorithm as the initialization $z^{(0)}$ in Algorithm \ref{alg:2}, we have with probability at least $1- n^{-1} - \exp(-\snr)$,
\begin{align*}
h(z^{(t)},z^*) \leq \ebr{-(1+o(1))\frac{\snr^{'2}}{8}},\quad \text{for all }t\geq \log n.
\end{align*}
\end{corollary}

\section{Numerical Studies}\label{sec:numeric}
In this section, we compare the performance of the proposed methods with other popular clustering methods on synthetic datasets under different settings.

%investigate the numerical performance of the adjusted Lloyd's algorithm and compare them with theoretical results. 

\paragraph{Model 1.}
The first simulation is designed for the GMM with unknown but homogeneous covariance matrices (i.e., Model 1). We independently generate $n=1200$ samples with dimension $d=50$ from $k=30$ clusters. Each cluster has 40 samples. We set $\Sigmastar = U^T\Lambda U$, where $\Lambda$ is a $50\times 50$ diagonal matrix with diagonal elements selected from 0.5 to 8 with equal space and $U$ is a randomly generated orthogonal matrix. The centers $\{\theta^*_a\}_{a\in[n]}$ are orthogonal or each other with $\norm{\theta^*_1} =\ldots = \norm{\theta^*_{30}} =9$. We consider four popular clustering methods: (1) the spectral clustering method \cite{loffler2019optimality} (denoted as ``spectral''), (2) the vanilla Lloyd's algorithm \cite{lu2016statistical} (denoted as ``vanilla Lloyd''), (3) the proposed Algorithm \ref{alg:1} initialized by the spectral clustering (denoted as ``spectral + Alg 1''), and (4) Algorithm \ref{alg:1} initialized by the vanilla Lloyd (denoted as ``vanilla Lloyd + Alg 1''). The comparison is presented in the left panel of Figure \ref{fig:sim}.

In the plot, the $x$-axis is the number of iterations and the $y$-axis is the logarithm of the mis-clustering error rate, i.e., $\log (h)$. Each of the curves plotted is an average of 100 independent trials. We can see the proposed Algorithm \ref{alg:1} outperforms the spectral clustering and the vanilla Lloyd's algorithm significantly. What is more, the dashed line represents the optimal exponent $-\snr^2/8$ of the minimax lower bound given in Section \ref{thm:lower1}. Then we can see Algorithm 1 achieves it after 3 iterations. This justifies the conclusion established in Theorem \ref{thm:main1} that Algorithm 1 is rate-optimal.

\paragraph{Model 2.}
We also compare the performances of  four methods (spectral, vanilla Lloyd, spectral + Alg 2, and vanilla Lloyd + Alg 2)  for  the GMM with unknown and heterogeneous covariance matrices (i.e., Model 2).
 In this case, we take $n=1200$, $k=3$ and $d =5$. We set $\Sigma_1^* = I$, $\Sigma_2^* = \Lambda_2$ which is a $5 \times5$ diagonal matrix with elements generated from 0.5 to 8 with equal space and $\Sigma_3^* = U^T\Lambda_3U$, where $\Lambda_3$ is a diagonal matrix with elements selected uniformly from 0.5 to 2 and $U$ is a randomly generated orthogonal matrix. To simplify the calculation of $\snr'$, we take $\theta_1^*$ as a randomly selected unit vector, $\theta_2^* = \theta_1^* + 5e_1$ with $e_1$ denoting the vector with a 1 in the first coordinate and 0's elsewhere and $\theta_3^* = \theta_2^* + v_1$ with $v_1$ randomly selected satisfying $\|v_1\|=10$. The comparison is presented in the right panel of Figure \ref{fig:sim} where each  curve plotted is an average of 100 independent trials.

From  the plot, we can clearly see the  proposed Algorithm \ref{alg:2} improves greatly the spectral clustering and the vanilla Lloyd algorithm. The dashed line represents the optimal exponent $-\snr'^2/8$ of the minimax lower bound given in Section \ref{thm:lower2}, which is achieved  by Algorithm \ref{alg:2}. Hence, this numerically justifies Theorem \ref{thm:main2} that Algorithm 1 is rate-optimal for clustering under  Model 2.

\begin{figure}
\centering
\includegraphics[scale= 0.5]{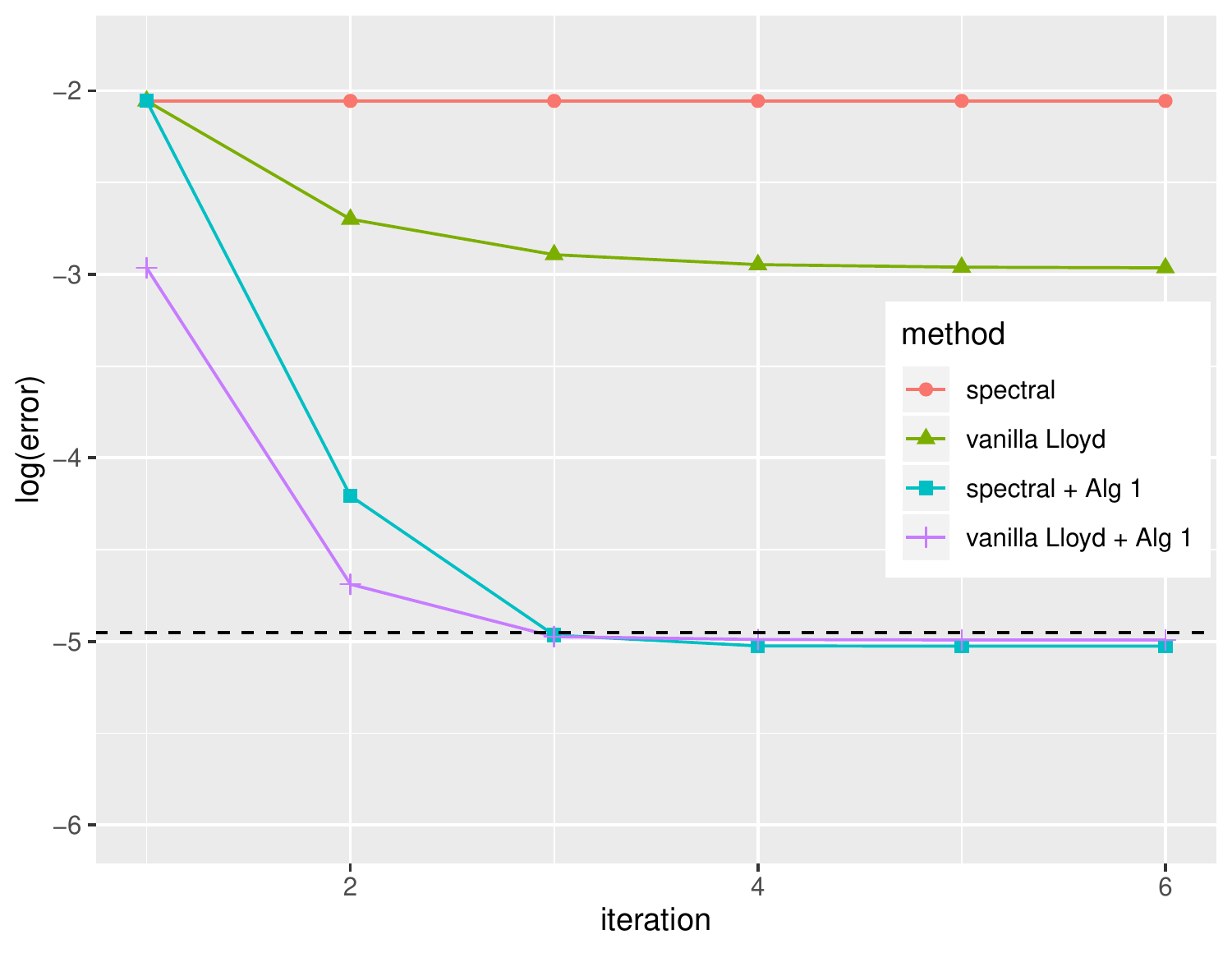}
\includegraphics[scale= 0.5]{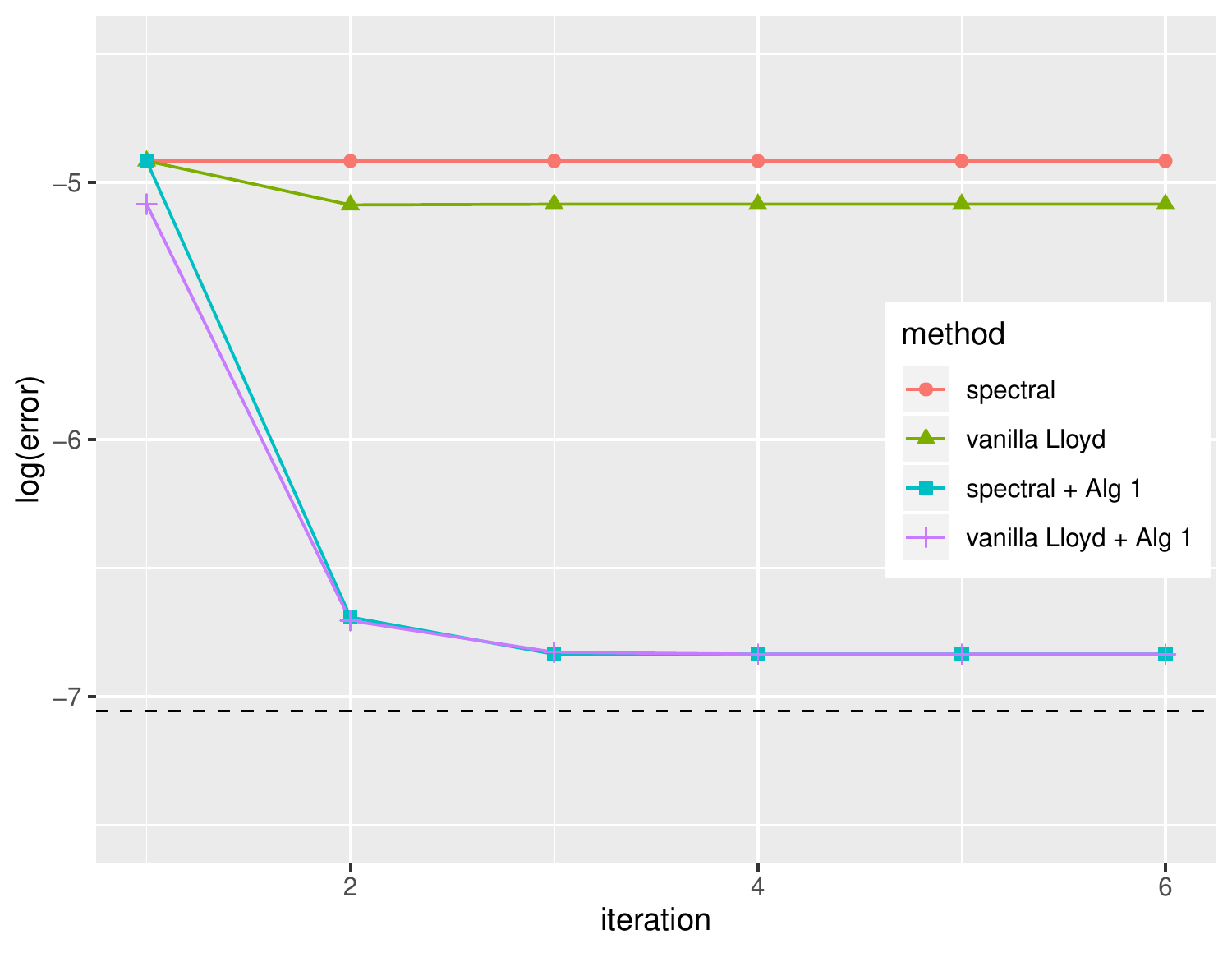}
\caption{Left: Performance of  Algorithm \ref{alg:1} compared with other methods under  Model 1. Right: Performance of  Algorithm \ref{alg:2} compared with other methods under  Model 2.}
        \label{fig:sim}
\end{figure}

%\begin{figure}[!htb]
%    \centering
%    \begin{minipage}{.5\textwidth}
%        \centering
%        \includegraphics[scale= 0.5]{plot_sec2.pdf}
%        \caption{Performance of the Algorithm \ref{alg:1} compared with other methods under the Model 1.}
%        \label{fig:sim.sec2}
%    \end{minipage}%
%    \begin{minipage}{0.5\textwidth}
%        \centering
%        \includegraphics[scale= 0.5]{plot_sec3.pdf}
%	\caption{Performance of the Algorithm \ref{alg:2} compared with other methods under the Model 2.}
%	\label{fig:sim.sec3}
%    \end{minipage}
%\end{figure} 

%\begin{figure}[htbp]
%	\centering    %居中
%	
%	\includegraphics[scale= 0.5]{plot_sec2.pdf}
%	\caption{\footnotesize Performance of the Algorithm \ref{alg:1} compared with other methods under the Model 1.}
%	
%	\label{fig:sim.sec2}  %图片引用标记
%\end{figure}  
%\begin{figure}[htbp]
%	\centering    %居中
%	
%	\includegraphics[scale= 0.5]{plot_sec3.pdf}
%	\caption{\footnotesize Performance of the Algorithm \ref{alg:2} compared with other methods under the Model 2.}
%	
%	\label{fig:sim.sec3}  %图片引用标记
%\end{figure}  

%%!TEX root = 2020917.tex

\section{Proofs in Section \ref{sec:homo}} \label{subsec:proof}

\subsection{Proofs for The Lower Bound}
\begin{proof}[Proof of Lemma \ref{lem:lda}]
Note that $\phi$ is the likelihood ratio test. Hence by the Neyman-Pearson lemma, it is the optimal procedure. Let $\epsilon\sim\mathn(0,I_d)$.
By Gaussian tail probability, we have
\begin{align*}
\p_{\mathbb{H}_0} \br{ \phi = 1} +\p_{\mathbb{H}_1} \br{ \phi = 0} &= \pbr{2(\theta_2^*  - \theta_1^*)^T (\Sigma^*)^{-1}(\theta_1  + \epsilon) \geq \theta_2^{*T} (\Sigma^*)^{-1}\theta_2^* - \theta_1^{*T} (\Sigma^*)^{-1} \theta_1^*} \\
& \quad +  \pbr{2(\theta_2^*  - \theta_1^*)^T (\Sigma^*)^{-1}(\theta_2  + \epsilon) < \theta_2^{*T} (\Sigma^*)^{-1}\theta_2^* - \theta_1^{*T} (\Sigma^*)^{-1} \theta_1^*} \\
& = 2\pbr{2(\theta_2^*  - \theta_1^*)^T (\Sigma^*)^{-1}(\theta_1  + \epsilon) \geq \theta_2^{*T} (\Sigma^*)^{-1}\theta_2^* - \theta_1^{*T} (\Sigma^*)^{-1} \theta_1^*}\\
& =2 \pbr{\epsilon > \frac{1}{2} \| (\theta_2^*-\theta_1^*)^T (\Sigma^*)^{-\frac{1}{2}}\|}\\
& \geq C \min\cbr{1, \frac{1}{\| (\theta_2^*-\theta_1^*)^T (\Sigma^*)^{-\frac{1}{2}}\|}\ebr{ - \frac{\| (\theta_2^*-\theta_1^*)^T (\Sigma^*)^{-\frac{1}{2}}\|^2}{8}}},
\end{align*}
for some constant $C>0$. The proof is complete.
\end{proof}

\begin{proof}[Proof of Theorem \ref{thm:lower1}]
We adopt the idea from \cite{lu2016statistical}. Without loss of generality, assume the minimum in (\ref{eqn:SNR_I}) is achieved at $a=1,b=2$ so that  $\snr = (\theta_1^*-\theta_2^*)^T(\Sigmastar)^{-1}(\theta_1^*-\theta_2^*)$. 
%Assume $\epsilon_i \sim N(0, \Sigmastar)$ for any $i \in [n]$.  
Consider an arbitrary $\bar z\in[k]^n$ such that $|\{i\in[n]:\bar z_i = a\}| \geq \lceil \frac{n}{k} - \frac{ n}{8k^2}\rceil$ for any $a\in[k]$. Then for each $a\in[k]$, we can choose a subset of $\{i\in[n]:\bar z_i = a\}$ with 	cardinality $ \lceil \frac{n}{k} - \frac{ n}{8k^2}\rceil$, denoted by $T_a$. Let $T = \cup_{a\in [k]} T_a$. Then we can define a parameter space
\begin{align*}
\mathcal{Z} = \cbr{z\in[k]^n:z_i = \bar z_i \text{ for all }i\in T \text{ and }z_i \in\{1,2\} \text{ if }i\in T^c}.
\end{align*}
%For any $z \in [k]^n,$ define $n_u(z) = \sum_{i=1}^{n}\mathbb{I}\{z_i=u\}$. It is easy to find  $z^* \in [k]^n$ satisfying $n_1(z^*)\leq n_2(z^*)\leq\cdots\leq n_k(z^*)$ with $n_1(z^*) = n_2(z^*) = \lfloor\frac{n}{2k}\rfloor$. For each $u \in [k]$, we choose a subset of $\{i:z^*(i)=u\}$ with cardinality $\lceil n_u(z^*) - \frac{ n}{8k^2}\rceil$, denoted by $T_u$. Define $T = \cup_{u=1}^kT_u$ and $\mathcal{Z^*} = \{z: z_i = z_i^* \text{ for all }i\in T \}$. 
Notice that for any $z\neq \tilde{z} \in \mathcal{Z}$, we have $\frac{1}{n}\sum_{i=1}^{n}\mathbb{I}\{z_i\neq\tilde{z}_i\} \leq \frac{k}{n}\frac{n}{8k^2} = \frac{1}{8k}$ and $\frac{1}{n}\sum_{i=1}^{n}\mathbb{I}\{\psi(z_i)\neq \tilde{z}_i\} \geq \frac{1}{n}(\frac{n}{2k} - \frac{n}{8k^2})\geq \frac{1}{4k}$ for any permutation $\psi$ on $[k]$. Thus we can conclude 
\begin{align*}
h(z,\tilde{z}) = \frac{1}{n}\sum_{i=1}^{n}\mathbb{I}\{z_i\neq\tilde{z}_i\},\quad \text{for all }z,\tilde{z}\in\mathcal{Z}.
\end{align*}
We notice that
\begin{align*}
\inf_{\hat{z}}\sup_{z^*\in [k]^n}\E h(\hat{z},z^*)\geq~& \inf_{\hat{z}}\sup_{z^*\in \mathcal{Z}}\E h(\hat{z}, z^*)\\
\geq~& \inf_{\hat{z}}\frac{1}{\abs{\mathcal{Z}}}\sum_{z^* \in \mathcal{Z}}\E h(\hat{z},z^*)\\
\geq~ & \frac{1}{n}\sum_{i\in T^c}\inf_{\hat{z}_i}\frac{1}{\abs{\mathcal{Z}}}\sum_{z^* \in \mathcal{Z}}\p_{z^*}(\hat{z}_i \neq z_i).
\end{align*}
%where $\ave$ stands for the arithmetic average. 
Now consider a fixed $i\in T^c$.
Define $\mathcal{Z}_a = \{z\in \mathcal{Z}:z_i=a\}$ for $a=1,2$. Then we can see $\mathcal{Z} = \mathcal{Z}_1 \cup \mathcal{Z}_2$ and $\mathcal{Z}_1\cap \mathcal{Z}_2 = \emptyset$. What is more, there exists a one-to-one mapping $f(\cdot)$ between $\mathcal{Z}_1$ and  $\mathcal{Z}_2$, such that for any $z\in \mathcal{Z}_1$, we have $f(z)\in \mathcal{Z}_2$ with $[f(z)]_j = z_j$ for any $j\neq i$ and $[f(z)]_i = 2$. Hence, we can reduce the problem into a two-point testing probe and then apply Lemma \ref{lem:lda}. We first  consider the case that $\snr \rightarrow\infty$. We have
\begin{align*}
\inf_{\hat{z}_i}\frac{1}{\abs{\mathcal{Z}}}\sum_{z^* \in \mathcal{Z}}\p_{z^*}(\hat{z}_i \neq z_i) & = \inf_{\hat{z}_i}\frac{1}{\abs{\mathcal{Z}}}\sum_{z^* \in \mathcal{Z}_1} \br{ \p_{z^*}(\hat{z}_i \neq 1) + \p_{f(z^*)}(\hat{z}_i \neq 2)} \\
&\geq \frac{1}{\abs{\mathcal{Z}}}\sum_{z^* \in \mathcal{Z}_1}  \inf_{\hat{z}_i}\br{ \p_{z^*}(\hat{z}_i \neq 1) + \p_{f(z^*)}(\hat{z}_i \neq 2)} \\
&\geq   \frac{\abs{\mathcal{Z}_1}}{\mathcal{Z}}  \ebr{-(1+\eta) \frac{\snr^2}{8}} \\
&\geq \frac{1}{2}\ebr{-(1+\eta) \frac{\snr^2}{8}},
\end{align*}
for some $\eta=o(1)$. Here the second inequality is due to Lemma \ref{lem:lda}. Then,
\begin{align*}
\inf_{\hat{z}}\sup_{z^*\in [k]^n}\E h(\hat{z},z^*) &\geq \frac{\abs{T^c}}{2n}\ebr{-(1+\eta) \frac{\snr^2}{8}} = \frac{1}{16k }\ebr{-(1+\eta) \frac{\snr^2}{8}}  \\
& = \ebr{-(1+\eta') \frac{\snr^2}{8}},
\end{align*}
for some other $\eta'=o(1)$, where we use  $\snr^2/\log k\rightarrow\infty$.
%$k=O(1)$ and $\snr\rightarrow\infty$.

The proof for the case $\snr=O(1)$ is similar and hence is omitted here.
\end{proof}
%\begin{align*}
%&\inf_{\hat z}\ave_{z^* \in \mathcal{Z}^*}\p(\hat{z}_i\neq z_i) \\
%\geq~& \frac{1}{k}\inf_{\hat{z}_i}\sum_{a=1}^{k}\ave_{z^* \in \mathcal{Z}^*_a}\p(\hat{z}_i \neq a)\\
%\geq~& \frac{1}{k}\inf_{\hat{z}_i}(\p_1(\hat{z}_i=2)+\p_2(\hat{z}_i=1))\\
%= & \frac{2}{k}\p\left\{2\langle\theta^*_2 - \theta^*_1, (\Sigmastar)^{-1}(\theta_1^*+\epsilon_i)\rangle \geq (\theta_2^*)^T(\Sigmastar)^{-1}\theta_2^* - (\theta_1^*)^T(\Sigmastar)^{-1}\theta_1^*\right\}\\
%= & \frac{2}{k}\pbr{\langle\theta^*_2 - \theta^*_1, (\Sigmastar)^{-1}\epsilon_i\rangle \geq \frac{1}{2}\langle\theta^*_2 - \theta^*_1, (\Sigmastar)^{-1}(\theta^*_2 - \theta^*_1) \rangle}.
%\end{align*}
%Above all, by the Gaussian tail bound we can conclude
%\begin{align*}
%&\inf_{\hat{z}}\sup_{z^*\in [k]^n}\E h(\hat{z},z^*)\\
%\geq~ &\frac{1}{n}\frac{n}{8k}\frac{2}{k}\pbr{\langle\theta^*_2 - \theta^*_1, (\Sigmastar)^{-1}\epsilon_i\rangle \geq \frac{1}{2}\langle\theta^*_2 - \theta^*_1, (\Sigmastar)^{-1}(\theta^*_2 - \theta^*_1) \rangle}\\
%\geq~ &\ebr{-(1+o(1))\frac{\snr^2}{8}}.
%\end{align*}
%The proof is complete.

\subsection{Proofs for The Upper Bound}\label{sec:sec2upper}
In this section, we will prove Theorem \ref{thm:main1} using the framework developed in \cite{gao2019iterative} for analyzing iterative algorithms. The key idea to establish statistical guarantees of the proposed iterative algorithm (i.e., Algorithm \ref{alg:1}) is to perform a ``one-step'' analysis. That is, assume we have an estimation $z$ for $z^*$. Then we can apply (\ref{eqn:1_theta}), (\ref{eqn:1_Sigma}), and (\ref{eqn:1_z}) on $z$ to obtain $\{\hat\theta_a(z)\}_{a\in[k]}$, $\hat\Sigma(z)$, and $\hat z(z)$ sequentially, which all depend on $z$. Then $\hat z(z)$ can be seen as a refined estimation of $z^*$. We will first build the connection between $\ell(z,z^*)$ with $\ell(\hat z(z),z^*)$ as in Lemma \ref{lem:I_one_step}. To establish the connection, we will decompose  the loss $\ell(\hat z(z),z^*)$ into several errors according to the difference in their behaviors. Then we will give  conditions (Condition \ref{con:I1} - \ref{con:I3}), under which we will show these errors are either negligible or well controlled by  $\ell(z,z^*)$.  With Lemma \ref{lem:I_one_step} established, in Lemma \ref{lem:linear convergence} we will show the connection can be extended to multiple iterations, under two more conditions  (Condition \ref{con:I4} - \ref{con:I5}).  Last, we will show all these conditions hold with high probability, and hence prove Theorem \ref{thm:main1}.

In the statement of Theorem \ref{thm:main1}, the covariance matrix $\Sigma^*$ is assumed to satisfy $\lambda_{d}(\Sigma^*) / \lambda_{1}(\Sigma^*)=O(1)$. Without loss of generality, we can replace it by assuming $\Sigma^*$  satisfies
\begin{align}\label{eqn:Sigma_assumption}
\lambda_{\min}\leq \lambda_1(\Sigma^*)\leq \lambda_d(\Sigma^*)\leq \lambda_{\max}
\end{align}
where $\lambda_{\min}, \lambda_{\max}>0$ are two constants. This is due to the following simple argument using the scaling properties of normal distributions. Let $\{Y_j\}$ be some dataset generated according to  Model 1 with parameters $\{\theta^*_a\}_{a\in[k]}$, $\Sigma^*$, and $z^*$.  The assumption $\lambda_{d}(\Sigma^*) / \lambda_{1}(\Sigma^*)=O(1)$ is equivalent to assume there exist some constants $\lambda_{\min}, \lambda_{\max}>0$ and some quantity $\sigma>0$ that may depend on $n$ such that  $\lambda_{\min}\sigma^2\leq \lambda_1(\Sigma^*)\leq \lambda_d(\Sigma^*)\leq \lambda_{\max}\sigma^2$. Then performing a scaling transformation  we  obtain another dataset $Y'_j = Y_j /\sigma$. Note that: 1) $\{Y'_j\}$ can be seen to be generated from  Model 1 with parameters $\{\theta^*_a/\sigma\}_{a\in[k]}$, $\Sigma^*/\sigma^2$, and $z^*$, 2)  clustering on $\{Y_j\}$ is equivalent to clustering on $\{Y'_j\}$,  3) by the definition in (\ref{eqn:SNR_I}),  the $\snr$s that are associated with the data generating processes of $\{Y'_j\}$ and  $\{Y_j\}$ are exactly equal to each other, and 4) we have $\lambda_{\min}\leq \lambda_1(\Sigma^*/\sigma^2)\leq \lambda_d(\Sigma^*/\sigma^2)\leq \lambda_{\max}$. Hence, in this section, we will assume (\ref{eqn:Sigma_assumption}) holds and it will not  lose any generality.

In the proof, we will mainly use the loss $\ell(\cdot,\cdot)$ for convenience.
%another loss function for convenience. For any $z,z^*\in[k]^n$, define $\ell(z,z^*)$ as
%\begin{align}\label{eqn:l}
%\ell(z,z^*) = \sum_{j=1}^n \norm{\theta^*_{z_j} - \theta^*_{z^*_j}}^2.
%\end{align}
%It is closely related to the misclustering error rate $h(z,z^*)$. 
Recall $\Delta$ is defined as the minimum distance among centers in (\ref{eqn:delta}). We have
\begin{align}\label{eqn:delta_h}
h(z,z^*) \leq \frac{\ell(z,z^*)}{n\Delta^2}.
\end{align}
The algorithmic guarantees Lemma \ref{lem:I_one_step} and Lemma \ref{lem:linear convergence} are established with respect to the $\ell(\cdot,\cdot)$ loss. But eventually we will use (\ref{eqn:delta_h}) to convert it into a result with respect to $h(\cdot,\cdot)$ in the proof of Theorem  \ref{thm:main1}.

\paragraph{Error Decomposition for the One-step Analysis:}
Consider an arbitrary $z\in[k]^n$. Apply (\ref{eqn:1_theta}), (\ref{eqn:1_Sigma}), and (\ref{eqn:1_z}) on $z$ to obtain $\{\hat\theta_a(z)\}_{a\in[k]}$, $\hat\Sigma(z)$, and $\hat z(z)$:
\begin{align*}
 \hat\theta_a(z) &= \frac{\sum_{j\in[n]} Y_j\indic{z_j=a}}{\sum_{j\in[n]}\indic{z_j=a}},\quad a\in[k]\\
  \hat\Sigma(z)  &=  \frac{\sum_{a\in[k]}\sum_{j\in[n]} (Y_j -   \hat\theta_a(z)) (Y_j -   \hat\theta_a(z))^T \indic{z_j=a}}{n},\\
 \hat  z_j(z) &= \argmin_{a\in[k]} (Y_j -  \hat\theta_a(z) )^T (\hat\Sigma(z))^{-1}(Y_j -  \hat\theta_a),\quad j\in[n].
\end{align*}
 For simplicity we use $\hat z$ that is short for $\hat z(z)$. Let $j\in[n]$ be an arbitrary index with 
%\begin{align}
%\label{interative}
%z_j^{(t)} = \arg\min_{a \in [k]}\langle Y_j - \theta_a^{(t)}, (\Sigma^{(t)})^{-1}(Y_j - \theta_a^{(t)})\rangle + \log|\Sigma^{(t)}|.
%\end{align}
 $\zjstar = a$. According to (\ref{eqn:1_z}), $z^*_j$ will be incorrectly estimated after on iteration in $\hat z_j$ if $a\neq \argmin_{a\in[k]} (Y_j -  \hat \theta_a(z) )^T (\hat \Sigma(z))^{-1}(Y_j -  \hat \theta_a(z) )$. That is, it is important to analyze the event
\begin{align}
\label{Discriminate}
\langle Y_j - \thetahatbz, (\Sigmahatz)^{-1}(Y_j - \thetahatbz)\rangle \leq \langle Y_j - \thetahataz, (\Sigmahatz)^{-1}(Y_j - \thetahataz)\rangle,
\end{align}
for any $b \in [k] \setminus \{a\} $. Note that $Y_j = \theta^*_a + \epsilon_j$.
After some rearrangements, we can see (\ref{Discriminate}) is equivalent to,
\begin{align}
&\langle \epsilon_j, (\Sigmahatzstar)^{-1}(\thetahatazstar - \thetahatbzstar)\rangle \notag\\
\leq & -\frac{1}{2}\langle \thetaastar - \thetabstar, (\Sigmastar)^{-1}(\thetaastar - \thetabstar)\rangle \notag
+  F_j(a,b,z)  + G_j(a,b,z) + H_j(a,b,z),
\end{align}
where
\begin{align*}
F_j(a,b,z) &= \langle \epsilon_j, (\Sigmahatz)^{-1}(\thetahatbz - \thetahatbzstar) \rangle - \langle \epsilon_j, (\Sigmahatz)^{-1}(\thetahataz - \thetahatazstar) \rangle\\
& + \langle \epsilon_j, ((\Sigmahatz)^{-1} - (\Sigmahatzstar)^{-1})(\thetahatbzstar - \thetahatazstar) \rangle,
\end{align*}
\begin{align*}
G_j(a,b,z) &= \br{\frac{1}{2}\langle \thetaastar-\thetahataz, (\Sigmahatz)^{-1}(\thetaastar-\thetahataz) \rangle - \frac{1}{2}\langle \thetaastar-\thetahatazstar, (\Sigmahatz)^{-1}(\thetaastar-\thetahatazstar) \rangle}\\
& +\br{ \frac{1}{2}\langle \thetaastar-\thetahatazstar, (\Sigmahatz)^{-1}(\thetaastar-\thetahatazstar) \rangle - \frac{1}{2}\langle \thetaastar-\thetahatazstar, (\Sigmahatzstar)^{-1}(\thetaastar-\thetahatazstar) \rangle}\\
& + \br{ - \frac{1}{2}\langle \thetaastar-\thetahatbz, (\Sigmahatz)^{-1}(\thetaastar-\thetahatbz) \rangle +\frac{1}{2} \langle \thetaastar-\thetahatbzstar, (\Sigmahatz)^{-1}(\thetaastar-\thetahatbzstar) \rangle}\\
& + \br{- \frac{1}{2}\langle \thetaastar-\thetahatbzstar, (\Sigmahatz)^{-1}(\thetaastar-\thetahatbzstar) \rangle +\frac{1}{2} \langle \thetaastar-\thetahatbzstar, (\Sigmahatzstar)^{-1}(\thetaastar-\thetahatbzstar) \rangle},
\end{align*}
and
\begin{align*}
H_j(a,b,z) = &\br{-\frac{1}{2}\langle \thetaastar - \thetahatbzstar, (\Sigmahatzstar)^{-1}(\thetaastar - \thetahatbzstar) \rangle + \frac{1}{2} \langle \thetaastar - \thetabstar, (\Sigmahatzstar)^{-1}(\thetaastar - \thetabstar) \rangle}\\
&+\br{-\frac{1}{2}\langle \thetaastar - \thetabstar, (\Sigmahatzstar)^{-1}(\thetaastar - \thetabstar)\rangle + \frac{1}{2}\langle \thetaastar - \thetabstar, (\Sigmastar)^{-1}(\thetaastar - \thetabstar)\rangle}\\
& +\br{ \frac{1}{2}\langle \thetaastar - \thetahatazstar, (\Sigmahatzstar)^{-1}(\thetaastar - \thetahatazstar) \rangle}.
\end{align*}
Here $\langle \epsilon_j, (\Sigmahatzstar)^{-1}(\thetahatazstar - \thetahatbzstar)\rangle \notag\leq  -\frac{1}{2}\langle \thetaastar - \thetabstar, (\Sigmastar)^{-1}(\thetaastar - \thetabstar)\rangle$ is the main term that will lead to the optimal rate. Among all the remaining terms,
$F_j(a,b,z)$ includes all  terms involving $\epsilon_j$.  $G_j(a,b,z)$ includes all terms related to $z$ and $H_j(a,b,z)$ consists of terms that only involves $z^*$. Readers can refer \cite{gao2019iterative} for more information about the decomposition.

\paragraph{Conditions and Guarantees for One-step Analysis.} 
To establish the guarantee for the one-step analysis, we first give several conditions on the error terms $F_j(a,b;z), G_j(a,b;z)$ and $H_j(a,b;z)$. 
%The validity of these conditions would be shown later.

\begin{condition}\label{con:I1}
Assume that
\begin{align*}
\max_{\{z:l(z,z^\ast) \leq \tau\}}\max_{j \in [n]}\max_{b\in[k]\setminus\{\zjstar\}} \frac{|H_j(\zjstar,b,z)|}{\langle \theta_{\zjstar}^*- \thetabstar, (\Sigmastar)^{-1}(\theta_{\zjstar}^* - \thetabstar) \rangle} \leq \frac{\delta}{4}
\end{align*}
holds with probability with at least $1-\eta_1$ for some $\tau, \delta, \eta_1 > 0$.	
\end{condition}

\begin{condition}\label{con:I2}
	Assume that
	\begin{align*}
	\max_{\{z:l(z,z^\ast) \leq \tau\}}\sum_{j=1}^{n}\max_{b\in[k]\setminus\{\zjstar\}}\frac{F_j(\zjstar,b,z)^2\|\thetazjstarstar - \thetabstar\|^2}{\langle \thetazjstarstar - \thetabstar, (\Sigma^\ast)^{-1}(\thetazjstarstar - \thetabstar) \rangle^2l(z,z^\ast)} \leq \frac{\delta^2}{256}
	\end{align*}
	holds with probability with at least $1-\eta_2$ for some $\tau, \delta, \eta_2 > 0$.	
\end{condition}

\begin{condition}\label{con:I3}
	Assume that 
	\begin{align*}
	\max_{\{z:l(z,z^\ast) \leq \tau\}}\max_{j \in [n]}\max_{b\in[k]\setminus\{\zjstar\}}\frac{|G_j(\zjstar,b,z)|}{\langle \thetazjstarstar - \thetabstar, (\Sigma^\ast)^{-1}(\thetazjstarstar - \thetabstar) \rangle} \leq \frac{\delta}{8}
	\end{align*}
	holds with probability with at least $1-\eta_3$ for some $\tau, \delta, \eta_3 > 0$.
\end{condition}
We next define a quantity that we refer it as the ideal error,
\begin{align*}
\xi_{\text{ideal}}(\delta) = \sum_{j=1}^{n}\sum_{b\in[k]\setminus\{\zjstar\}}\|\thetazjstarstar - \thetabstar\|^2\mathbb{I}\{\langle \epsilon_j, (\Sigmahatzstar)^{-1}(\thetahatazstar - \thetahatbzstar)\rangle \notag
\leq  -\frac{1-\delta}{2}\langle \thetaastar - \thetabstar, (\Sigmastar)^{-1}(\thetaastar - \thetabstar)\rangle\}.
\end{align*} 

\begin{lemma}\label{lem:I_one_step}
Assumes Conditions \ref{con:I1} - \ref{con:I3} hold for some $\tau,\delta, \eta_1,\eta_2,\eta_3,>0$. We then have 
	\begin{align*}
	\pbr{\ell(\hat z,z^*) \leq 2\xi_{\text{ideal}}(\delta) + \frac{1}{4}\ell(z, z^*)\text{ for any $z\in[k]^n$ such that $\ell(z,z^*)\leq \tau$}} \geq 1-\eta,
	\end{align*}
	%for any $z\in[k]^n$ such that $\ell(z,z^*)\leq \tau$, with probability at least $1-\eta$, 
	where $\eta = \sum_{i=1}^{3}\eta_i$.
\end{lemma}
\begin{proof}
We notice that
\begin{align*}
\mathbb{I}\left\{\hat{z}_j = b\right\} &\leq \indic{\langle Y_j - \thetahatbz, (\Sigmahatz)^{-1}(Y_j - \thetahatbz)\rangle \leq \langle Y_j - \thetahataz, (\Sigmahatz)^{-1}(Y_j - \thetahataz)\rangle}\\
= ~& \mathbb{I}\bigg\{\langle \epsilon_j, (\Sigmahatzstar)^{-1}(\hat{\theta}_{\zjstar}(z^*) - \thetahatbzstar)\rangle \notag\\
&\leq  -\frac{1}{2}\langle \theta_{\zjstar}^* - \thetabstar, (\Sigmastar)^{-1}(\theta_{\zjstar}^* - \thetabstar)\rangle \notag
+  F_j(\zjstar,b,z)  + G_j(\zjstar,b,z) + H_j(\zjstar,b,z)\bigg\}\\
\leq ~& \mathbb{I}\left\{\langle \epsilon_j, (\Sigmahatzstar)^{-1}(\hat{\theta}_{\zjstar}(z^*) - \thetahatbzstar)\rangle \leq -\frac{1-\delta}{2}\langle \theta_{\zjstar}^* - \thetabstar, (\Sigmastar)^{-1}(\theta_{\zjstar}^* - \thetabstar)\rangle\right\}\\
&+ \mathbb{I}\left\{\frac{\delta}{2}\langle \theta_{\zjstar}^* - \thetabstar, (\Sigmastar)^{-1}(\theta_{\zjstar}^* - \thetabstar)\rangle \leq F_j(\zjstar,b,z)  + G_j(\zjstar,b,z) + H_j(\zjstar,b,z)\right\}\\
\leq ~& \mathbb{I}\left\{\langle \epsilon_j, (\Sigmahatzstar)^{-1}(\hat{\theta}_{\zjstar}(z^*) - \thetahatbzstar)\rangle \leq -\frac{1-\delta}{2}\langle \theta_{\zjstar}^* - \thetabstar, (\Sigmastar)^{-1}(\theta_{\zjstar}^* - \thetabstar)\rangle\right\}\\
&+ \mathbb{I}\left\{\frac{\delta}{8}\langle \theta_{\zjstar}^* - \thetabstar, (\Sigmastar)^{-1}(\theta_{\zjstar}^* - \thetabstar)\rangle \leq F_j(\zjstar,b,z) \right\}\\
\leq ~& \mathbb{I}\left\{\langle \epsilon_j, (\Sigmahatzstar)^{-1}(\hat{\theta}_{\zjstar}(z^*) - \thetahatbzstar)\rangle \leq -\frac{1-\delta}{2}\langle \theta_{\zjstar}^* - \thetabstar, (\Sigmastar)^{-1}(\theta_{\zjstar}^* - \thetabstar)\rangle\right\}\\
&+ \frac{64F_j(\zjstar,b,z)^2}{\delta^2\langle \theta_{\zjstar}^* - \thetabstar, (\Sigmastar)^{-1}(\theta_{\zjstar}^* - \thetabstar)\rangle^2} ,
\end{align*}
where the third inequality comes from Condition \ref{con:I1}. Thus, we have 
\begin{align*}
&\ell(\hat{z},z^*)\\
=~ & \sum_{j=1}^{n}\sum_{b\in[k]\setminus \{a\}}\left\|\thetabstar - \theta_{\zjstar}^*\right\|^2\mathbb{I}\left\{\hat{z}_j=b\right\}\\
\leq~ & \xi_{\text{ideal}}(\delta) + \sum_{j=1}^{n}\sum_{b\in[k]\setminus \{a\}}\left\|\thetabstar - \theta_{\zjstar}^*\right\|^2\mathbb{I}\left\{\hat{z}_j=b\right\}\frac{64F_j(\zjstar,b,z)^2}{\delta^2\langle \theta_{\zjstar}^* - \thetabstar, (\Sigmastar)^{-1}(\theta_{\zjstar}^* - \thetabstar)\rangle^2}\\
%&+ \sum_{j=1}^{n}\sum_{b\in[k]\setminus \{a\}}\left\|\thetabstar - \theta_{\zjstar}^*\right\|^2\mathbb{I}\{\hat{z}_j=b\}\frac{32G_j(\zjstar,b,z)^2}{\delta^2\langle \theta_{\zjstar}^* - \thetabstar, (\Sigmastar)^{-1}(\theta_{\zjstar}^* - \thetabstar)\rangle^2}\\
\leq~ &\xi_{\text{ideal}}(\delta) + \frac{\ell(z,z^*)}{4} ,
\end{align*}
which implies Lemma \ref{lem:I_one_step}. Here the last inequality uses Conditions \ref{con:I2} and \ref{con:I3}.
\end{proof}

\paragraph{Conditions and Guarantees for Multiple Iterations.} 
In the above we establish a statistical guarantee for the one-step analysis. Now we will extend the result to multiple iterations. That is, starting from some initialization $z^{(0)}$, we will characterize how the losses $\ell(z^{(0)},z^*)$, $\ell(z^{(1)},z^*)$, $\ell(z^{(2)},z^*)$, \ldots, decay.
We impose a condition on $\xi_{\text{ideal}}(\delta)$ and a condition for $z^{(0)}$.
\begin{condition}\label{con:I4}
	Assume that
	\begin{align*}
	\xi_{\text{ideal}}(\delta) \leq \frac{3\tau}{8}
	\end{align*}
	holds with probability with at least $1-\eta_4$ for some $\tau, \delta, \eta_4 > 0$.
\end{condition}
Finally, we need a condition on the initialization.
\begin{condition}\label{con:I5}
	Assume that
	\begin{align*}
	\ell(z^{(0)}, z^*) \leq \tau
	\end{align*}
	holds with probability with at least $1-\eta_5$ for some $\tau,  \eta_5 > 0$.
\end{condition}
With these conditions satisfied, we can give a lemma that shows the linear convergence guarantee for our algorithm.

\begin{lemma}\label{lem:linear convergence}
	Assumes Conditions \ref{con:I1} - \ref{con:I5} hold for some $\tau,\delta, \eta_1,\eta_2,\eta_3,\eta_4,\eta_5>0$. We then have 
	\begin{align*}
	\ell(z^{(t)},z^*) \leq 2\xi_{\text{ideal}}(\delta) + \frac{1}{4}\ell(z^{(t-1)}, z^*)
	\end{align*}
	for all $t \geq 1$, with probability at least $1-\eta$, where $\eta = \sum_{i=1}^{5}\eta_i$.
\end{lemma}
\begin{proof}
	By Condition \ref{con:I4}, \ref{con:I5}, and a mathematical induction argument, we can easily know $\ell(z^{(t)},z^*) \leq \tau$ for any $t \geq 0$. Thus, Lemma \ref{lem:linear convergence} is a direct extension of Lemma \ref{lem:I_one_step}.
\end{proof}

\paragraph{With-high-probability Results for the Conditions and The Proof of The Main Theorem.}
Recall the definition of $\Delta$ in \eqref{eqn:delta}. Recall that in (\ref{eqn:Sigma_assumption}) we assume $\lambdamin \leq \lambda_{1}(\Sigma^*) \leq  \lambda_{d}(\Sigma^*)\leq \lambdamax$ for two constants $\lambdamin, \lambdamax>0$. Hence we have $\Delta$ is in the same order of $\snr$. Specifically, we have
\begin{align}
\frac{1}{\lambdamax}\Delta \leq \snr \leq \frac{1}{\lambdamin}\Delta.\label{eqn:snr_delta_1}
\end{align}
Hence the assumption $\snr/k\rightarrow\infty$ in the statement of Theorem \ref{thm:main1} is equivalently $\Delta/k \rightarrow\infty$.
Next, we give two lemmas to clarify the Conditions. The first lemma shows that $\delta$ can  be taken for some $o(1)$ term and the second lemma shows for any $\delta = o(1)$, $\xi_{\text{ideal}}(\delta)$ is upper bounded by the desired minimax rate multiply by the sample size $n$. 
\begin{lemma}\label{lem:conditions}
	Under the same conditions as in Theorem \ref{thm:main1}, for any constant $C'>0$, there exists some constant $C>0$ only depending on $\alpha$ and $C'$ such that
	\begin{align}
	\max_{\{z:\ell(z,z^\ast) \leq \tau\}}\max_{j \in [n]}\max_{b\in[k]\setminus\{\zjstar\}} \frac{|H_j(\zjstar,b,z)|}{\langle \theta_{\zjstar}^*- \thetabstar, (\Sigma^\ast)^{-1}(\theta_{\zjstar}^* - \thetabstar) \rangle} &\leq C\sqrt{\frac{k(d+\log n)}{n}}\label{Hcondi}\\
	\max_{\{z:\ell(z,z^\ast) \leq \tau\}}\sum_{j=1}^{n}\max_{b\in[k]\setminus\{\zjstar\}}\frac{F_j(\zjstar,b,z)^2\|\thetazjstarstar - \thetabstar\|^2}{\langle \thetazjstarstar - \thetabstar, (\Sigma^\ast)^{-1}(\thetazjstarstar - \thetabstar) \rangle^2\ell(z,z^\ast)} &\leq Ck^3\br{\frac{\tau}{n}+\frac{1}{\Delta^2}+\frac{d^2}{n\Delta^2}} \label{Fcondi}\\
	\max_{\{z:\ell(z,z^\ast) \leq \tau\}}\max_{j \in [n]}\max_{b\in[k]\setminus\{\zjstar\}}\frac{|G_j(\zjstar,b,z)|}{\langle \thetazjstarstar - \thetabstar, (\Sigma^\ast)^{-1}(\thetazjstarstar - \thetabstar) \rangle} &\leq Ck\br{\frac{\tau}{n} + \frac{1}{\Delta}\sqrt{\frac{\tau}{n}} + \frac{d\sqrt{\tau}}{n\Delta}}\label{Gcondi}
	\end{align}
	with probability at least $1- n^{-C'}$.
\end{lemma}
\begin{proof}
Under the conditions of Theorem \ref{thm:main1}, the inequalities \eqref{ine:lem:miscel.1}-\eqref{ine:keyine} hold with probability at least $1-n^{-C'}$. In the remaining proof, we will work on the event these inequalities hold.
	%We consider the events that inequalities \eqref{ine:lem:miscel.1}-\eqref{ine:lem:covspec.true.2} and \eqref{ine:keyine} in technical lemmas hold. 
	Denote $ \hat \Sigma_a(z)  =  \frac{\sum_{j\in[n]} (Y_j -   \hat\theta_a(z)) (Y_j -    \hat\theta_a(z))^T \indic{z_j=a}}{\sum_{j\in[n]}\indic{z_j=a}}$ and $\Sigma^*_a = \Sigma^*$ for any $a\in[k]$. Then we have the equivalence
	\begin{align*}
	\Sigmahatzstar - \Sigmastar = \sum_{a=1}^{k}\frac{\sum_{j=1}^{n}\mathbb{I}\{z_j^*=a\}}{n}(\Sigmahatazstar - \Sigmaastar).
	\end{align*}
	Hence, we can use the results  from Lemma \ref{Cov.inv.rate} and Lemma \ref{Cov.inv.rate2}.
	
	 By \eqref{cov_star_to_true} and \eqref{covrate}, we  have
	\begin{align}
	\|\Sigmahatzstar - \Sigmastar\| \preceq~ & \sqrt{\frac{k(d+\log n)}{n}}, \notag
	\end{align}
	and 
	\begin{align*}
	\|\Sigmahatz - \Sigmahatzstar\| =~& \left\|\sum_{a=1}^{k}\frac{\sum_{j=1}^{n}\mathbb{I}\{z_j=a\}}{n}\Sigmahataz - \sum_{a=1}^{k}\frac{\sum_{j=1}^{n}\mathbb{I}\{z_j^*=a\}}{n}\Sigmahatazstar\right\|\\
	\preceq~& \left\|\sum_{a=1}^{k}\frac{\sum_{j=1}^{n}\mathbb{I}\{z_j=a\}}{n}(\Sigmahataz - \Sigmahatzstar)\right\|+ \left\|\sum_{a=1}^k\frac{\sum_{j=1}^{n}(\indzja - \indzjstara)}{n}\Sigmahatazstar\right\|\\
	 \preceq~& \frac{k\sqrt{n\ell(z,z^*)}}{n\Delta} + \frac{k}{n}\ell(z,z^*)+\frac{kd}{n\Delta}\sqrt{\ell(z,z^*)} + \frac{k}{n\Delta^2}\ell(z,z^*)\\
	 \preceq~& \frac{k\sqrt{n\ell(z,z^*)}}{n\Delta} + \frac{k}{n}\ell(z,z^*)+\frac{kd}{n\Delta}\sqrt{\ell(z,z^*)}.
	\end{align*}
	By the assumption that $kd=O(\sqrt{n})$, $\frac{\Delta}{k}\rightarrow\infty$ and $\tau=o(n/k)$, we have $\|\Sigmahatzstar - \Sigmastar\|, \|\Sigmahatz - \Sigmahatzstar\|=o(1)$, which implies $\|(\Sigmahatzstar)^{-1}\|,\|(\hat{\Sigma}(z))^{-1}\|\lesssim 1$. 
	Thus, we have
	\begin{align}
	\|(\Sigmahatzstar)^{-1} - (\Sigmastar)^{-1}\| \leq \|(\Sigmahatzstar)^{-1} \| \|\Sigmahatzstar - \Sigmastar\| \|(\Sigmastar)^{-1}\| \preceq \sqrt{\frac{k(d+\log n)}{n}}, \label{Sigmahatzstartrue.inv}
	\end{align}
	and similarly
	\begin{align}\label{Sigmazzstar.inv}
	\|(\hat{\Sigma}(z))^{-1} - (\hat{\Sigma}(z^*))^{-1}\| \preceq \frac{k}{n}\ell(z,z^*)+\frac{k\sqrt{n\ell(z,z^*)}}{n\Delta}+\frac{kd}{n\Delta}\sqrt{\ell(z,z^*)}.
	\end{align}
	Now we start to prove \eqref{Hcondi}-\eqref{Gcondi}. Let $F_j(a,b,z) = F_j^{(1)}(a,b,z) + F_j^{(2)}(a,b,z)+F_j^{(3)}(a,b,z)$ where 
	\begin{align*}
	F_j^{(1)}(a,b,z)&:=\langle \epsilon_j, (\Sigmahatz)^{-1}(\thetahatbz - \thetahatbzstar) \rangle - \langle \epsilon_j, (\Sigmahatz)^{-1}(\thetahataz - \thetahatazstar)\rangle,\\
	F_j^{(2)}(a,b,z)&:= -\langle \epsilon_j, ((\Sigmahatz)^{-1} - (\Sigmahatzstar)^{-1})(\thetaastar - \thetabstar)\rangle,\\
	F_j^{(3)}(a,b,z)&:= -\langle \epsilon_j, ((\Sigmahatz)^{-1} - (\Sigmahatzstar)^{-1})(\thetabstar - \thetahatbzstar)\rangle + \langle \epsilon_j, ((\Sigmahatz)^{-1} - (\Sigmahatzstar)^{-1})(\thetaastar - \thetahatazstar)\rangle.
	\end{align*}
	Notice that
	\begin{align*}
	&\sum_{j=1}^{n}\max_{b\in[k]\setminus\{\zjstar\}}\frac{F_j^{(2)}(\zjstar,b,z)^2\|\thetazjstarstar - \thetabstar\|^2}{\langle \thetazjstarstar - \thetabstar, (\Sigma^\ast)^{-1}(\thetazjstarstar - \thetabstar) \rangle^2\ell(z,z^\ast)}\\
	\preceq~ &\sum_{j=1}^{n}\sum_{b=1}^{k}\frac{\bigg|\langle \epsilon_j, ((\Sigmahatz)^{-1} - (\Sigmahatzstar)^{-1})(\thetazjstarstar - \thetabstar)\rangle\bigg|^2}{\|\thetazjstarstar - \thetabstar\|^2\ell(z,z^*)}\\
	\leq~ & \sum_{b=1}^{k}\sum_{a\in[k]\setminus \{b\}}\sum_{j=1}^n\indzjstara\frac{\bigg|\langle \epsilon_j, ((\Sigmahatz)^{-1} - (\Sigmahatzstar)^{-1})(\thetaastar - \thetabstar)\rangle\bigg|^2}{\|\thetaastar - \thetabstar\|^2\ell(z,z^*)}\\
	\leq~& \sum_{b=1}^{k}\sum_{a\in[k]\setminus \{b\}}\frac{\|((\Sigmahatz)^{-1} - (\Sigmahatzstar)^{-1})(\thetaastar - \thetabstar)\|^2}{\|\thetaastar - \thetabstar\|^2\ell(z,z^*)}\bigg\|\sum_{j=1}^{n}\indzjstara\epsilon_j\epsilon_j^T\bigg\|\\
	\preceq~& k^3(\frac{\tau}{n}+\frac{1}{\Delta^2}+\frac{d^2}{n\Delta^2}),
	\end{align*}
	where we use \eqref{ine.lem.miscel.3},  \eqref{Sigmazzstar.inv}, and the fact that $\ell(z,z^*) \leq \tau$ and $kd=O(\sqrt{n})$ for the last inequality. From (\ref{thetaazstartrue}) we have  $\max_{a\in[k]}\|\thetaastar - \thetahatazstar\| =o(1)$ under the assumption $kd=O(\sqrt{n})$. By the similar analysis as in $	F_j^{(2)}(a,b,z)$, we  have
	\begin{align*}
	\sum_{j=1}^{n}\max_{b\in[k]\setminus\{\zjstar\}}\frac{F_j^{(3)}(\zjstar,b,z)^2\|\thetazjstarstar - \thetabstar\|^2}{\langle \thetazjstarstar - \thetabstar, (\Sigma^\ast)^{-1}(\thetazjstarstar - \thetabstar) \rangle^2\ell(z,z^\ast)} \preceq k^3(\frac{\tau}{n}+\frac{1}{\Delta^2}+\frac{d^2}{n\Delta^2}).
	\end{align*}
	Similarly, we have
	\begin{align*}
	&\sum_{j=1}^{n}\max_{b\in[k]\setminus\{\zjstar\}}\frac{F_j^{(1)}(\zjstar,b,z)^2\|\thetazjstarstar - \thetabstar\|^2}{\langle \thetazjstarstar - \thetabstar, (\Sigma^\ast)^{-1}(\thetazjstarstar - \thetabstar) \rangle^2\ell(z,z^\ast)}\\
	\preceq~& \sum_{b=1}^{k}\sum_{a\in[k]\setminus \{b\}}\frac{\|(\Sigmahatz)^{-1}(\thetahataz - \thetahatazstar)\|^2}{\|\thetaastar - \thetabstar\|^2\ell(z,z^*)}\bigg\|\sum_{j=1}^{n}\indzjstara\epsilon_j\epsilon_j^T\bigg\|\\
	\preceq~& \frac{k^3}{\Delta^4},
	\end{align*}
	where we use \eqref{thetazzstar} and the fact that $(\Sigmahatz)^{-1}$ has bounded operator norm. Combining these terms together, we obtain \eqref{Fcondi}. 
	
	Next, for \eqref{Hcondi}, by \eqref{thetaazstartrue} we have 
	\begin{align*}
	&|-\langle \thetaastar - \thetahatbzstar, (\Sigmahatzstar)^{-1}(\thetaastar - \thetahatbzstar) \rangle +  \langle \thetaastar - \thetabstar, (\Sigmahatzstar)^{-1}(\thetaastar - \thetabstar) \rangle |\\
	\leq~& |\langle \thetabstar - \thetahatbzstar, (\Sigmahatzstar)^{-1}(\thetabstar - \thetahatbzstar)\rangle|+ 2|\langle \thetabstar - \thetahatbzstar, (\Sigmahatzstar)^{-1}(\thetaastar - \thetabstar)\rangle|\\
	\preceq~& \frac{k(d+\log n)}{n}+ \sqrt{\frac{k(d+\log n)}{n}}\|\thetaastar - \thetabstar\|,
	\end{align*}
	and 
	\begin{align*}
	\langle \thetaastar - \thetahatazstar, (\Sigmahatzstar)^{-1}(\thetaastar - \thetahatazstar) \rangle \preceq \frac{k(d+\log n)}{n}.
	\end{align*}
	By \eqref{Sigmahatzstartrue.inv} we have
	\begin{align*}
	-\langle \thetaastar - \thetabstar, (\Sigmahatzstar)^{-1}(\thetaastar - \thetabstar)\rangle + \langle \thetaastar - \thetabstar, (\Sigmastar)^{-1}(\thetaastar - \thetabstar)\rangle \preceq \sqrt{\frac{k(d+\log n)}{n}}\|\thetaastar - \thetabstar\|^2.
	\end{align*}
	Using the results above we can get \eqref{Hcondi}. 
	
	Finally we are going to establish \eqref{Gcondi}. Recall the definition of $G_j(a,b,z)$ which has four terms. For the third and fourth terms, we have
	\begin{align*}
	&-\langle \thetaastar-\thetahatbz, (\Sigmahatz)^{-1}(\thetaastar-\thetahatbz) \rangle + \langle \thetaastar-\thetahatbzstar, (\Sigmahatz)^{-1}(\thetaastar-\thetahatbzstar) \rangle \\
	\preceq~ & \|\thetahatbz - \thetahatbzstar\|^2 + \|\thetahatbz- \thetahatbzstar\|\|\thetaastar - \thetabstar\|,
	\end{align*}
	and 
	\begin{align*}
	&-\langle \thetaastar-\thetahatbzstar, (\Sigmahatz)^{-1}(\thetaastar-\thetahatbzstar) \rangle +\langle \thetaastar-\thetahatbzstar, (\Sigmahatzstar)^{-1}(\thetaastar-\thetahatbzstar) \rangle \\
	\preceq~& \|\thetaastar - \thetabstar\|^2\|(\Sigmahatz)^{-1} - (\Sigmahatzstar)^{-1}\|.
	\end{align*}
	We can easily verify that the other two terms are smaller than the above two terms. Then, by using (\ref{thetazzstar}) and (\ref{Sigmazzstar.inv}), we have 
	\begin{align*}
	&\frac{|G_j(\zjstar,b,z)|}{\langle \thetazjstarstar - \thetabstar, (\Sigma^\ast)^{-1}(\thetazjstarstar - \thetabstar) \rangle} \\
	&\lesssim \frac{\|\thetahatbz - \thetahatbzstar\|^2 + \|\thetahatbz- \thetahatbzstar\|\|\theta^*_{z^*_j} - \theta^*_b\| + \|\theta^*_{z^*_j} - \theta^*_b\|^2\|(\Sigmahatz)^{-1} - (\Sigmahatzstar)^{-1}\|}{\|\theta^*_{z^*_j} - \theta^*_b\|^2}\\
	&\lesssim \frac{k\tau}{n} + \frac{k}{\Delta}\sqrt{\frac{\tau}{n}} + \frac{kd\sqrt{\tau}}{n\Delta}. 
	\end{align*}
	\end{proof}
%	\begin{align*}
%	&\max_{\{z:\ell(z,z^\ast) \leq \tau\}}\max_{j \in [n]}\max_{b\in[k]\setminus\{\zjstar\}}\frac{G_j(\zjstar,b,z)^2}{\langle \thetazjstarstar - \thetabstar, (\Sigma^\ast)^{-1}(\thetazjstarstar - \thetabstar) \rangle^2} \\
%	\preceq~&\max_{\{z:\ell(z,z^\ast) \leq \tau\}}\max_{j \in [n]}\max_{b\in[k]\setminus\{\zjstar\}}\frac{\|\thetahatbz - \thetahatbzstar\|^4+\|\thetahatbz - \thetahatbzstar\|^2\|\thetazjstarstar - \thetabstar\|^2 }{\langle \thetazjstarstar - \thetabstar, (\Sigma^\ast)^{-1}(\thetazjstarstar - \thetabstar) \rangle^2}\\
%	&+ \max_{\{z:\ell(z,z^\ast) \leq \tau\}}\max_{j \in [n]}\max_{b\in[k]\setminus\{\zjstar\}}\frac{\|\thetazjstarstar - \thetabstar\|^4\|(\Sigmahatz)^{-1} - (\Sigmahatzstar)^{-1}\|^2}{\langle \thetazjstarstar - \thetabstar, (\Sigma^\ast)^{-1}(\thetazjstarstar - \thetabstar) \rangle^2}\\
%	\preceq~& \frac{k^4\tau^2}{n^2\Delta^4} + \frac{k^2\tau}{n\Delta^4} + \frac{k^2\tau^2}{n^2}+\frac{k^2\tau}{n\Delta^2} + \frac{k^2d^2\tau}{n^2\Delta^2}\\
%	\preceq~& k^2(\frac{\tau}{n\Delta^2} + \frac{\tau^2}{n^2}), 
%	\end{align*}
%	where we use \eqref{thetazzstar} and \eqref{Sigmazzstar.inv} for the second inequality and $k = O(1)$ and $d = O(\sqrt{n})$ for the last inequality.

\begin{lemma}\label{lem:upperbound}
	With the same conditions as Theorem \ref{thm:main1}, for any sequence $\delta_n = o(1)$, we have
	\begin{align*}
	\xi_{\text{ideal}}(\delta_n) \leq n\ebr{-(1+o(1))\frac{\snr^2}{8}}.
	\end{align*}
	with probability at least $1 - n^{-C'} - \exp(-\snr)$.
\end{lemma}
\begin{proof}
Under the conditions of Theorem \ref{thm:main1}, the inequalities \eqref{ine:lem:miscel.1}-\eqref{ine:keyine} hold with probability at least $1-n^{-C'}$. In the remaining proof, we will work on the event these inequalities hold.
 	%We consider the events that inequalities \eqref{ine:lem:miscel.1}-\eqref{ine:lem:covspec.true.2} and \eqref{ine:keyine} hold. 
 	Recall the definition of $\xi_{\text{ideal}}$, we can write
	\begin{align*}
	\xi_{\text{ideal}}(\delta) &= \sum_{j=1}^{n}\sum_{b\in[k]\setminus\{\zjstar\}}\|\thetazjstarstar - \thetabstar\|^2\mathbb{I}\left\{\langle \epsilon_j, (\Sigmahatzstar)^{-1}(\thetahatzjstarzstar - \thetahatbzstar)\rangle \notag
	\leq  -\frac{1-\delta}{2}\langle \thetazjstarstar - \thetabstar, (\Sigmastar)^{-1}(\thetazjstarstar - \thetabstar)\rangle\right\}\\
	& \leq \sum_{j=1}^{n}\sum_{b\in[k]\setminus\{\zjstar\}}\|\thetazjstarstar - \thetabstar\|^2\mathbb{I}\left\{\langle \epsilon_j, (\Sigmastar)^{-1}(\thetazjstarstar - \thetabstar)\rangle\leq  -\frac{1-\delta-\bar{\delta}}{2}\langle \thetazjstarstar - \thetabstar, (\Sigmastar)^{-1}(\thetazjstarstar - \thetabstar)\rangle\right\}\\
	&+ \sum_{j=1}^{n}\sum_{b\in[k]\setminus\{\zjstar\}}\|\thetazjstarstar - \thetabstar\|^2\mathbb{I}\left\{\langle \epsilon_j, ((\Sigmahatzstar)^{-1} - (\Sigmastar)^{-1})(\thetazjstarstar - \thetabstar) \rangle \leq -\frac{\bar{\delta}}{6}\langle \thetazjstarstar - \thetabstar, (\Sigmastar)^{-1}(\thetazjstarstar - \thetabstar)\rangle\right\}\\
	&+\sum_{j=1}^{n}\sum_{b\in[k]\setminus\{\zjstar\}}\|\thetazjstarstar - \thetabstar\|^2\mathbb{I}\left\{\langle \epsilon_j, (\Sigmahatzstar)^{-1}(\thetahatzjstarzstar - \thetazjstarstar)\rangle \leq -\frac{\bar{\delta}}{6}\langle \thetazjstarstar - \thetabstar, (\Sigmastar)^{-1}(\thetazjstarstar - \thetabstar)\rangle\right\}\\
	&+ \sum_{j=1}^{n}\sum_{b\in[k]\setminus\{\zjstar\}}\|\thetazjstarstar - \thetabstar\|^2\mathbb{I}\left\{-\langle \epsilon_j, (\Sigmahatzstar)^{-1}(\thetahatbzstar - \thetabstar)\rangle \leq -
	\frac{\bar{\delta}}{6}\langle \thetazjstarstar - \thetabstar, (\Sigmastar)^{-1}(\thetazjstarstar - \thetabstar)\rangle\right\}\\
	&=: M_1 + M_2+M_3+M_4.
	\end{align*} 
	where $\bar{\delta} = \bar{\delta}_n$ is some sequence to be chosen later. We bound the four terms respectively. Suppose $\epsilon_j = (\Sigmastar)^{1/2}\omega_j$, where $w_j \iid \mathn(0, I_d)$. By \eqref{Sigmahatzstartrue.inv}, we know
	\begin{align*}
	M_2 &\leq \sum_{j=1}^{n}\sum_{b\in[k]\setminus\{\zjstar\}}\|\thetazjstarstar - \thetabstar\|^2\mathbb{I}\left\{\frac{\bar{\delta}}{6\lambdamax}\|\thetazjstarstar - \thetabstar\|^2 \leq \lambdamax\|w_j\|\|(\Sigmahatzstar)^{-1} - (\Sigmastar)^{-1}\|\|\thetazjstarstar - \thetabstar\|\right\} \\
	&\leq \sum_{j=1}^{n}\sum_{b\in[k]\setminus\{\zjstar\}}\|\thetazjstarstar - \thetabstar\|^2\mathbb{I}\left\{C\bar{\delta}\|\thetazjstarstar - \thetabstar\|\sqrt{\frac{n}{d+\log n}}\leq \|w_j\|\right\} \\
	&\leq \sum_{j=1}^{n}\sum_{b\in[k]\setminus\{\zjstar\}}\|\thetazjstarstar - \thetabstar\|^2\mathbb{I}\left\{C\bar{\delta}^2\|\thetazjstarstar - \thetabstar\|^2\frac{n}{d+\log n} -2d\leq \|w_j\|^2-2d\right\},
	\end{align*}
	where $C$ is a constant which may differ line by line. Recall that $kd=O(\sqrt{n})$, $\min_{a\neq b}\|\theta^*_a-\theta^*_b\|\rightarrow\infty$, and $\Delta/k\rightarrow\infty$ by assumption. Let $n^{-\frac{1}{4}} =o(\bar \delta)$. Using the $\chi^2$ tail probability in Lemma \ref{lem:chi-square}, we have for any $a \neq b \in[k]$,
	\begin{align*}
	\E M_2 &\leq  \sum_{j=1}^{n}\sum_{b\in[k]\setminus\{\zjstar\}}\|\thetazjstarstar - \thetabstar\|^2 \ebr{-C\bar{\delta}^2\|\thetazjstarstar - \thetabstar\|^2 \sqrt{n}} \leq n\ebr{-(1+o(1))\frac{\snr^2}{8}}.
	\end{align*}
%	\begin{align*}
%	&\p\bigg\{C\bar{\delta}^2\|\thetaastar - \thetabstar\|^2\frac{n}{d+\log n} -2d\leq \|w_j\|^2-2d\bigg\}\\
%	\leq~& \ebr{-\|\thetaastar - \thetabstar\|^2(\frac{C}{3}\bar{\delta}^2\frac{n}{d+\log n} - \frac{2d}{3\|\thetaastar - \thetabstar\|^2})} \\
%	\leq~&\ebr{-(1+o(1))\frac{\snr^2}{8}}, 
%	\end{align*}
%	where the last inequality is obtained under the condition that $d = O(\sqrt{n})$ and $\snr \to \infty$ and the fact that there exists a constant $C$ such that $\snr \leq C\|\thetaastar - \thetabstar\|^2$ for any $a \neq b \in [k]$, and we can choose some $\bar{\delta} = \bar{\delta}_n = o(1)$ that is slowly diverging to zero. This yields 
%	\begin{align*}
%	\E M_2 &\leq \sum_{j=1}^{n}\sum_{b\in[k]\setminus\{\zjstar\}}\|\thetazjstarstar - \thetabstar\|^2 \p\bigg\{C\bar{\delta}^2\|\thetazjstarstar - \thetabstar\|^2\frac{n}{d+\log n} -2d\leq \|w_j\|^2-2d\bigg\}\\
%	&\leq \sum_{j=1}^{n}\sum_{b\in[k]\setminus\{\zjstar\}}\|\thetazjstarstar - \thetabstar\|^2 \ebr{-(1+o(1))\frac{\snr^2}{8}}\\
%	&\leq n\ebr{-(1+o(1))\frac{\snr^2}{8}}.
%	\end{align*}
	We can obtain similar  bounds on  $M_3$ and $M_4$ by using \eqref{thetaazstartrue}.
%	\begin{align*}
%	M_3 \leq \sum_{j=1}^{n}\sum_{b\in[k]\setminus\{\zjstar\}}\|\thetazjstarstar - \thetabstar\|^2\mathbb{I}\{C\bar{\delta}\|\thetazjstarstar - \thetabstar\|^2\sqrt{\frac{n}{d+\log n}} \leq \|w_j\|\}.
%	\end{align*}
%	Similar to the $M_2$, we can obtain for any $a\neq b \in[k]$,
%	\begin{align*}
%	&\p\bigg\{C\bar{\delta}\|\thetaastar - \thetabstar\|^2\sqrt{\frac{n}{d+\log n}} \leq \|w_j\|\bigg\} \\
%	\leq~& \ebr{-\|\thetaastar - \thetabstar\|^2(\frac{C}{3}\bar{\delta}^2\|\thetaastar - \thetabstar\|^2\frac{n}{d+\log n} - \frac{2d}{3\|\thetaastar - \thetabstar\|^2})}\\
%	\leq~ &\ebr{-(1+o(1))\frac{\snr^2}{8}},
%	\end{align*}
%	which implies
%	\begin{align*}
%	\E M_3 \leq n\ebr{-(1+o(1))\frac{\snr^2}{8}}.
%	\end{align*}
%	$M_4$ is essentially the same as $M_3$. 
For $M_1$, the Gaussian tail bound leads to the inequality
	\begin{align*}
	&\p\bigg\{ \langle \epsilon_j, (\Sigmastar)^{-1}(\thetaastar - \thetabstar)\rangle\leq  -\frac{1-\delta-\bar{\delta}}{2}\langle \thetaastar - \thetabstar, (\Sigmastar)^{-1}(\thetaastar - \thetabstar)\rangle\bigg\}\\
	=~& \p\bigg\{ \langle w_j, (\Sigmastar)^{-1/2}(\thetaastar - \thetabstar)\rangle\leq  -\frac{1-\delta-\bar{\delta}}{2}\langle \thetaastar - \thetabstar, (\Sigmastar)^{-1}(\thetaastar - \thetabstar)\rangle\bigg\}\\
	\leq~ & \ebr{-\frac{(1-\delta-\bar{\delta})^2}{8}\langle \thetaastar - \thetabstar, (\Sigmastar)^{-1}(\thetaastar - \thetabstar)\rangle}.
	%\leq~ &\ebr{-(1+o(1))\frac{\snr^2}{8}}, \quad  \text{for any }a \neq b \in [k].
	\end{align*}
	Thus,
	\begin{align*}
	\E M_1 &\leq \sum_{j=1}^{n}\sum_{b\in[k]\setminus\{\zjstar\}}\|\thetazjstarstar - \thetabstar\|^2\ebr{-\frac{(1-\delta-\bar{\delta})^2}{8}\langle \thetaastar - \thetabstar, (\Sigmastar)^{-1}(\thetaastar - \thetabstar)\rangle}\\
	&\leq n\ebr{-(1+o(1))\frac{\snr^2}{8}}.
		\end{align*}
%	\begin{align*}
%	\E M_1 \leq n\ebr{-(1+o(1))\frac{\snr^2}{8}}.
%	\end{align*}
	Overall, we have $\E \xi_{\text{ideal}} \lesssim  n\ebr{-(1+o(1))\frac{\snr^2}{8}}$.
	By the Markov's inequality, we have 
	\begin{align*}
	\pbr{\xi_{\text{ideal}}(\delta_n) \geq \E\xi_{\text{ideal}}\exp(\snr)} \leq \exp(-\snr).
	\end{align*}
	In other words, with probability at least $1 - \exp(-\snr)$, we have
	\begin{align*}
	\xi_{\text{ideal}}(\delta_n) \leq \E\xi_{\text{ideal}}(\delta_n)\exp(\snr) \leq n\ebr{-(1+o(1))\frac{\snr^2}{8}}.
	\end{align*}
\end{proof}

\begin{proof}[Proof of Theorem \ref{thm:main1}]
By Lemmas \ref{lem:linear convergence} - \ref{lem:upperbound}, we have Conditions \ref{con:I1} - \ref{con:I5} satisfied
with probability at least $1-\eta - n^{-1} - \exp(-\snr)$. Then applying Lemma \ref{lem:linear convergence}, we have 
\begin{align*}
\ell(z^{(t)},z^*) \leq n\ebr{-(1+o(1))\frac{\snr^2}{8}} + \frac{1}{4}\ell(z^{(t-1)}, z^*),\quad \text{for all }t\geq1.
\end{align*}
By (\ref{eqn:delta_h}) and there exists a constant C such that $\Delta \leq C\snr$, we can conclude
\begin{align*}
h(z^{(t)},z^*) \leq \ebr{-(1+o(1))\frac{\snr^2}{8}} + 4^{-t}, \quad \text{for all }t\geq1.
\end{align*}
Notice that $h(z,z^*)$ takes value in the set $\{j/n:j\in [n]\cup\{0\}\}$, the term $4^{-t}$ in the above inequality should be negligible as long as $4^{-t} = o(n^{-1})$. Thus, we can claim
\begin{align*}
h(z^{(t)},z^*)\leq \ebr{-(1+o(1))\frac{\snr^2}{8}}, \quad \text{for all }t\geq \log n.
\end{align*}
\end{proof}

\section{Proofs in Section \ref{sec:hetero}}\label{subsect:proof.hetero}

\subsection{Proofs for The Lower Bound}

\begin{proof}[Proof of Lemma \ref{lem:qda}]
The Neyman-Pearson lemma tells us the likelihood ratio test $\phi$ is the optimal procedure. Following the proof of Lemma \ref{lem:lda}, we have
\begin{align*}
\p_{\mathbb{H}_0} \br{ \phi = 1} +\p_{\mathbb{H}_1} \br{ \phi = 0}  &= \pbr{\epsilon\in \mathcal{B}_{1,2}} + \pbr{\epsilon\in \mathcal{B}_{2,1}} \\
& \geq \ebr{-\frac{1+o(1)}{8}\snr_{1,2}^{'2}} + \ebr{-\frac{1+o(1)}{8}\snr_{2,1}^{'2}},
\end{align*}
where the last inequality is by Lemma \ref{lem:2}.
\end{proof}

\begin{proof}[Proof of Theorem \ref{thm:lower2}]
The proof is identical to the proof of Theorem \ref{thm:lower1} and is omitted here.
%We adopt a similar proof procedure as Theorem \ref{thm:lower1} and we only present the difference here. Assume $\snr' = \snr'_{1,2}$. Using the same notation in the proof of Theorem \ref{thm:lower1} and consider the case that $\snr' \to \infty$, we can get
%\begin{align*}
%&\inf_{\hat z} \sup_{z^*\in [k]^n} \E h(z,z^*) \notag\\
%\geq~ &\frac{|T^c|}{2n}\inf_{\hat{z}_i}\br{ \p_{z^*}(\hat{z}_i \neq 1) + \p_{f(z^*)}(\hat{z}_i \neq 2)}\\
%\geq~ & \frac{1}{16k}\ebr{-(1+\eta)\frac{\snr^{'2}}{8}}\\
%=~ &\ebr{-(1+\eta')\frac{\snr^{'2}}{8}},
%\end{align*} 
%for some $\eta,\eta' = o(1)$. Here the second inequality is due to Lemma \ref{lem:qda}.
%The proof for the case $\snr' = O(1)$ is similar and hence is omitted here. 
\end{proof}

\subsection{Proofs for The Upper Bound}
We adopt a similar proof idea as in Section \ref{sec:sec2upper}. We first present an error decomposition for  one-step analysis for Algorithm \ref{alg:2}. In Lemma \ref{lem:II_one_step}, we show the loss decays after a one-step iteration under Conditions \ref{con:II_1} - \ref{con:II_6}. Then in Lemma \ref{lem:linear convergence.2} we extend the result to multiple iterations, under two extra Conditions \ref{con:II_7} - \ref{con:II_8}. Last we show all the conditions are satisfied with high probability and thus prove  Theorem \ref{thm:main2}.

In the statement of Theorem \ref{thm:main2}, we assume $\max_{a,b\in[k]}\lambda_{d}(\Sigma^*_a) / \lambda_{1}(\Sigma^*_b)=O(1)$ for
the covariance matrix $\{\Sigma_a^*\}_{a\in[k]}$. Without loss of generality, we can replace it by assuming $\Sigma^*$  satisfies
\begin{align}\label{eqn:Sigma_assumption_2}
\lambda_{\min}\leq \min_{a\in[k]}\lambda_1(\Sigma^*_a)\leq \max_{a\in[k]}\lambda_d(\Sigma^*_a)\leq \lambda_{\max}
\end{align}
where $\lambda_{\min}, \lambda_{\max}>0$ are two constants.  The is due to the scaling properties of the normal distributions. The reasoning is the same as that of (\ref{eqn:Sigma_assumption}) for Model 1 and hence is omitted here. In the remaining of this section, we will assume (\ref{eqn:Sigma_assumption_2}) holds for the covariance matrices.

\paragraph{Error Decomposition for the One-step Analysis:}
Consider an arbitrary $z\in[k]^n$. Apply (\ref{eqn:2_theta}), (\ref{eqn:2_Sigma}), and (\ref{eqn:2_z}) on $z$ to obtain $\{\hat\theta_a(z)\}_{a\in[k]}$, $\{\hat\Sigma_a(z)\}_{a\in[k]}$, and $\hat z(z)$:
\begin{align*}
 \hat\theta_a(z) &= \frac{\sum_{j\in[n]} Y_j\indic{z_j=a}}{\sum_{j\in[n]}\indic{z_j=a}},\\
  \hat\Sigma_a(z)  &=  \frac{\sum_{j\in[n]} (Y_j -   \hat\theta_a(z)) (Y_j -   \hat\theta_a(z))^T \indic{z_j=a}}{\sum_{j\in[n]}\indic{z_j=a}},\\
  \hat z_j(z) & = \argmin_{a\in[k]} (Y_j -  \hat\theta_a(z) )^T (\hat\Sigma_a(a))^{-1}(Y_j -  \hat\theta_a(z) ) + \log |\hat\Sigma_a(z)|,\quad j\in[n].
\end{align*}
For simplicity we denote $\hat z$ short for $\hat z(z)$. Let $j\in[n]$ be an arbitrary index with 
%\begin{align}
%\label{interative}
%z_j^{(t)} = \arg\min_{a \in [k]}\langle Y_j - \theta_a^{(t)}, (\Sigma^{(t)})^{-1}(Y_j - \theta_a^{(t)})\rangle + \log|\Sigma^{(t)}|.
%\end{align}
 $\zjstar = a$. According to (\ref{eqn:1_z}), $z^*_j$ will be incorrectly estimated after on iteration in $\hat z$ if $a\neq \argmin_{a\in[k]} (Y_j -  \hat\theta_a(z) )^T (\hat\Sigma_a(a))^{-1}(Y_j -  \hat\theta_a(z) ) + \log |\hat\Sigma_a(z)|,$. That is, it is important to analyze the event
\begin{align}
\label{Discriminate.2}
\langle Y_j - \thetahatbz, (\Sigmahatbz)^{-1}(Y_j - \thetahatbz)\rangle + \log|\Sigmahatbz| \leq \langle Y_j - \thetahataz, (\Sigmahataz)^{-1}(Y_j - \thetahataz)\rangle + \log |\Sigmahataz|,
\end{align}
for any $b \in [k] \setminus \{a\} $. After some rearrangements, we can see (\ref{Discriminate.2}) is equivalent to,
\begin{align*}
&\langle \epsilon_j, (\Sigmahatbzstar)^{-1}(\thetaastar - \thetahatbzstar)\rangle- \langle \epsilon_j, (\Sigmahatazstar)^{-1}(\thetaastar - \thetahatazstar)\rangle \notag\\
+~ &\frac{1}{2}\langle \epsilon_j, ((\Sigmahatbzstar)^{-1} - (\Sigmahatazstar)^{-1})\epsilon_j \rangle - \frac{1}{2}\log|\Sigmaastar| + \frac{1}{2}\log|\Sigmabstar|\\
\leq  -~&\frac{1}{2}\langle \thetaastar - \thetabstar, (\Sigmabstar)^{-1}(\thetaastar - \thetabstar)\rangle\\
+~  &F_j(a,b,z)  + Q_j(a,b,z)+ G_j(a,b,z) + H_j(a,b,z)+K_j(a,b,z)+L_j(a,b,z),
\end{align*}
where
\begin{align*}
F_j(a,b,z) &= \langle \epsilon_j, (\Sigmahatbz)^{-1}(\thetahatbz - \thetahatbzstar) \rangle - \langle \epsilon_j, (\Sigmahataz)^{-1}(\thetahataz - \thetahatazstar) \rangle\\
& - \langle \epsilon_j, ((\Sigmahatbz)^{-1} - (\Sigmahatbzstar)^{-1})(\thetaastar - \thetahatbzstar) \rangle\\
& + \langle \epsilon_j, ((\Sigmahataz)^{-1} - (\Sigmahatazstar)^{-1})(\thetaastar - \thetahatazstar) \rangle,
\end{align*}
\begin{align*}
Q_j(a,b,z) = -\frac{1}{2}\langle \epsilon_j, ((\Sigmahatbz)^{-1} - (\Sigmahatbzstar)^{-1})\epsilon_j \rangle + \frac{1}{2}\langle \epsilon_j, ((\Sigmahataz)^{-1} - (\Sigmahatazstar)^{-1})\epsilon_j \rangle,
\end{align*}
\begin{align*}
G_j(a,b,z) &= \frac{1}{2}\langle \thetaastar-\thetahataz, (\Sigmahataz)^{-1}(\thetaastar-\thetahataz) \rangle - \frac{1}{2}\langle \thetaastar-\thetahatazstar, (\Sigmahataz)^{-1}(\thetaastar-\thetahatazstar) \rangle\\
& + \frac{1}{2}\langle \thetaastar-\thetahatazstar, (\Sigmahataz)^{-1}(\thetaastar-\thetahatazstar) \rangle - \frac{1}{2}\langle \thetaastar-\thetahatazstar, (\Sigmahatazstar)^{-1}(\thetaastar-\thetahatazstar) \rangle\\
& - \frac{1}{2}\langle \thetaastar-\thetahatbz, (\Sigmahatbz)^{-1}(\thetaastar-\thetahatbz) \rangle +\frac{1}{2} \langle \thetaastar-\thetahatbzstar, (\Sigmahatbz)^{-1}(\thetaastar-\thetahatbzstar) \rangle\\
&- \frac{1}{2}\langle \thetaastar-\thetahatbzstar, (\Sigmahatbz)^{-1}(\thetaastar-\thetahatbzstar) \rangle +\frac{1}{2} \langle \thetaastar-\thetahatbzstar, (\Sigmahatbzstar)^{-1}(\thetaastar-\thetahatbzstar) \rangle,
\end{align*}
\begin{align*}
H_j(a,b,z) = &-\frac{1}{2}\langle \thetaastar - \thetahatbzstar, (\Sigmahatbzstar)^{-1}(\thetaastar - \thetahatbzstar) \rangle + \frac{1}{2}\langle \thetaastar - \thetabstar, (\Sigmahatbzstar)^{-1}(\thetaastar - \thetabstar) \rangle\\
&-\frac{1}{2}\langle \thetaastar - \thetabstar, (\Sigmahatbzstar)^{-1}(\thetaastar - \thetabstar)\rangle + \frac{1}{2}\langle \thetaastar - \thetabstar, (\Sigmabstar)^{-1}(\thetaastar - \thetabstar)\rangle\\
& + \frac{1}{2}\langle \thetaastar - \thetahatazstar, (\Sigmahatazstar)^{-1}(\thetaastar - \thetahatazstar) \rangle,
\end{align*}
\begin{align*}
K_j(a,b,z):= \frac{1}{2}(\log|\Sigmahataz |- \log|\Sigmahatazstar|) - \frac{1}{2}(\log|\Sigmahatbz |- \log|\Sigmahatbzstar|),
\end{align*}
and
\begin{align*}
L_j(a,b,z):= \frac{1}{2}(\log|\Sigmahatazstar| - \log|\Sigmaastar|) - \frac{1}{2}(\log|\Sigmahatbzstar| - \log|\Sigmabstar|).
\end{align*}	
Among these terms,  $F_j,G_j, H_j$ are nearly identical to their counterparts in Section \ref{sec:sec2upper} with $\hat \Sigma(z)$ replaced by $\hat \Sigma_a(z)$ or $\hat \Sigma_b(z)$. There are three extra terms not appearing in Section \ref{sec:sec2upper}: $Q_j$ is a quadratic term of $\epsilon_j$ and  $K_j,L_j$ are terms involving matrix determinants.
%Compared to the decomposition in Section (\ref{sec:homo}), here we have additional error terms $G_j, H_j, K_j$ due to the heterogeneity of the covariance matrices.

\paragraph{Conditions and Guarantees for One-step Analysis.} 
To establish the guarantee for the one-step analysis, we first give several conditions on the error terms.

\begin{condition}\label{con:II_1}
	Assume that
	\begin{align*}
	\max_{\{z:\ell(z,z^\ast) \leq \tau\}}\max_{j \in [n]}\max_{b\in[k]\setminus\{\zjstar\}} \frac{|H_j(\zjstar,b,z)|}{\langle \theta_{\zjstar}^*- \thetabstar, (\Sigmabstar)^{-1}(\theta_{\zjstar}^* - \thetabstar) \rangle} \leq \frac{\delta}{12}
	\end{align*}
	holds with probability with at least $1-\eta_1$ for some $\tau, \delta, \eta_1 > 0$.	
\end{condition}
\begin{condition}
	Assume that
	\begin{align*}
	\max_{\{z:\ell(z,z^\ast) \leq \tau\}}\sum_{j=1}^{n}\max_{b\in[k]\setminus\{\zjstar\}}\frac{F_j(\zjstar,b,z)^2\|\thetazjstarstar - \thetabstar\|^2}{\langle \thetazjstarstar - \thetabstar, (\Sigmabstar)^{-1}(\thetazjstarstar - \thetabstar) \rangle^2\ell(z,z^\ast)} \leq \frac{\delta^2}{288}
	\end{align*}
	holds with probability with at least $1-\eta_2$ for some $\tau, \delta, \eta_2 > 0$.	
\end{condition}
\begin{condition}
	Assume that 
	\begin{align*}
	\max_{\{z:\ell(z,z^\ast) \leq \tau\}}\max_{j \in [n]}\max_{b\in[k]\setminus\{\zjstar\}}\frac{|G_j(\zjstar,b,z)|}{\langle \thetazjstarstar - \thetabstar, (\Sigmabstar)^{-1}(\thetazjstarstar - \thetabstar) \rangle} \leq \frac{\delta}{12}
	\end{align*}
	holds with probability with at least $1-\eta_3$ for some $\tau, \delta, \eta_3 > 0$.
\end{condition}
\begin{condition}
	Assume that
	\begin{align*}
	\max_{\{z:\ell(z,z^\ast) \leq \tau\}}\sum_{j=1}^{n}\max_{b\in[k]\setminus\{\zjstar\}}\frac{Q_j(\zjstar,b,z)^2\|\thetazjstarstar - \thetabstar\|^2}{\langle \thetazjstarstar - \thetabstar, (\Sigmabstar)^{-1}(\thetazjstarstar - \thetabstar) \rangle^2\ell(z,z^\ast)} \leq \frac{\delta^2}{288}
	\end{align*}
	holds with probability with at least $1-\eta_4$ for some $\tau, \delta, \eta_4 > 0$.	
\end{condition}
\begin{condition}
	Assume that
	\begin{align*}
	\max_{\{z:\ell(z,z^\ast) \leq \tau\}}\sum_{j=1}^{n}\max_{b\in[k]\setminus\{\zjstar\}}\frac{K_j(\zjstar,b,z)^2\|\thetazjstarstar - \thetabstar\|^2}{\langle \thetazjstarstar - \thetabstar, (\Sigmabstar)^{-1}(\thetazjstarstar - \thetabstar) \rangle^2\ell(z,z^\ast)} \leq \frac{\delta^2}{288}
	\end{align*}
	holds with probability with at least $1-\eta_5$ for some $\tau, \delta, \eta_5 > 0$.	
\end{condition}
\begin{condition}\label{con:II_6}
	Assume that 
	\begin{align*}
	\max_{\{z:\ell(z,z^\ast) \leq \tau\}}\max_{j \in [n]}\max_{b\in[k]\setminus\{\zjstar\}}\frac{|L_j(\zjstar,b,z)|}{\langle \thetazjstarstar - \thetabstar, (\Sigmabstar)^{-1}(\thetazjstarstar - \thetabstar) \rangle} \leq \frac{\delta}{12}
	\end{align*}
	holds with probability with at least $1-\eta_6$ for some $\tau, \delta, \eta_6 > 0$.
\end{condition}
We next define a quantity that refers to as the ideal error,
\begin{align*}
&\xi_{\text{ideal}}(\delta) = \sum_{j=1}^{p}\sum_{b\in[k]\setminus\{\zjstar\}}\|\theta^*_{\zjstar} - \thetabstar\|^2\mathbb{I}\{\langle \epsilon_j, (\Sigmahatbzstar)^{-1}(\thetaastar - \thetahatbzstar) \rangle - \langle \epsilon_j, (\Sigmahatazstar)^{-1}(\thetaastar - \thetahatazstar) \rangle \\
&+ \frac{1}{2}\langle \epsilon_j, ((\Sigmahatbzstar)^{-1} - (\Sigmahatazstar)^{-1})\epsilon_j \rangle -\frac{1}{2}\log|\Sigmaastar| + \frac{1}{2}\log|\Sigmabstar|\leq -\frac{1-\delta}{2}\langle \thetaastar - \thetabstar,  (\Sigmabstar)^{-1}(\thetaastar - \thetabstar)\rangle\}.
\end{align*}
\begin{lemma}\label{lem:II_one_step}
Assumes Conditions \ref{con:II_1} - \ref{con:II_6} hold for some $\tau,\delta, \eta_1,\ldots, \eta_6>0$. We then have 
	\begin{align*}
	\pbr{\ell(\hat z,z^*) \leq 2\xi_{\text{ideal}}(\delta) + \frac{1}{4}\ell(z, z^*)\text{ for any $z\in[k]^n$ such that $\ell(z,z^*)\leq \tau$}} \geq 1-\eta,
	\end{align*}
	%for any $z\in[k]^n$ such that $\ell(z,z^*)\leq \tau$, with probability at least $1-\eta$,
	 where $\eta = \sum_{i=1}^{6}\eta_i$.
\end{lemma}
\begin{proof}
The proof of this lemma is quite similar to the proof of Lemma \ref{lem:I_one_step}. The additional terms $Q_j$ and $K_j$ can be dealt with the same way as $F_j$ while $L_j$ can be dealt with the same way as $H_j$. We omit the details here.
\end{proof}

\paragraph{Conditions and Guarantees for Multiple Iterations.} 
In the above we establish a statistical guarantee for the one-step analysis. Now we will extend the result to multiple iterations. That is, starting from some initialization $z^{(0)}$, we will characterize how the losses $\ell(z^{(0)},z^*)$, $\ell(z^{(1)},z^*)$, $\ell(z^{(2)},z^*)$, \ldots, decay.
We impose a condition on $\xi_{\text{ideal}}(\delta)$ and a condition for $z^{(0)}$.

%We impose a condition on $\xi_{\text{ideal}}(\delta)$ and would also show its proof later.
\begin{condition}\label{con:II_7}
	Assume that
	\begin{align*}
	\xi_{\text{ideal}}(\delta) \leq \frac{\tau}{2}
	\end{align*}
	holds with probability with at least $1-\eta_7$ for some $\tau, \delta, \eta_7 > 0$.
\end{condition}
Finally, we need a condition on the initialization.
\begin{condition}\label{con:II_8}
	Assume that
	\begin{align*}
	\ell(z^{(0)}, z^*) \leq \tau
	\end{align*}
	holds with probability with at least $1-\eta_8$ for some $\tau,  \eta_8 > 0$.
\end{condition}
With these conditions satisfied, we can give a lemma that shows the linear convergence guarantee for our algorithm.

\begin{lemma}\label{lem:linear convergence.2}
	Assumes Conditions \ref{con:II_1} - \ref{con:II_8} hold for some $\tau,\delta, \eta_1,\cdots,\eta_8>0$. We then have 
	\begin{align*}
	\ell(z^{(t)},z^*) \leq \xi_{\text{ideal}}(\delta) + \frac{1}{2}\ell(z^{(t-1)}, z^*)
	\end{align*}
	for all $t \geq 1$, with probability at least $1-\eta$, where $\eta = \sum_{i=1}^{8}\eta_i$.
\end{lemma}
\begin{proof}
	The proof of this lemma is the same as the proof of Lemma \ref{lem:linear convergence}. 
\end{proof}

\paragraph{With-high-probability Results for the Conditions and The Proof of The Main Theorem.} Recall the definition of $\Delta$ in (\ref{eqn:delta}).
Lemma \ref{lem:sameorder} shows $\snr'$ is in the same order with $\Delta$, which will play a similar role as (\ref{eqn:snr_delta_1}) in Section \ref{sec:sec2upper}. It immediately implies  the assumption $\snr'\rightarrow\infty$ in the statement of Theorem \ref{thm:main2} is equivalently $\Delta \rightarrow\infty$.
The proof of Lemma \ref{lem:sameorder} is deferred to Section \ref{sec:technical}.
\begin{lemma}\label{lem:sameorder}
	Assume $\snr' \to \infty$ and $d = O(1)$. Further assume there exist constants $\lambdamin , \lambdamax>0 $ such that  $\lambdamin \leq \lambda_{1}(\Sigma_a^*) \leq  \lambda_{d}(\Sigma_a^*)\leq \lambdamax$ for any $a\in[k]$. Then, there exist constants $C_1,C_0>0$ only depending on $\lambdamin,\lambdamax,d$ such that
	\begin{align*}
	C_1 \norm{\thetaastar - \thetabstar}\leq \snr'_{a,b} \leq C_2 \norm{\thetaastar - \thetabstar},
	\end{align*}
	for any $a \neq b \in [k]$. As a result, $\snr'$ is in the same order of $\Delta$.
	%Then, we have $\snr'$ is in the same order of $\Delta$. To be more specific, for any $a \neq b \in [k]$,
\end{lemma}

Lemma \ref{lem:conditions.2} and Lemma \ref{lem:upperbound.2} are  counterparts of Lemmas \ref{lem:conditions} and \ref{lem:upperbound} in Section \ref{sec:sec2upper}.
\begin{lemma}\label{lem:conditions.2}
	Under the same conditions as in Theorem \ref{thm:main2}, for any constant $C'>0$, there exists some constant $C>0$ only depending on $\alpha,C',\lambdamin,\lambdamax$ such that
	\begin{align}
	\max_{\{z:\ell(z,z^\ast) \leq \tau\}}\max_{j \in [n]}\max_{b\in[k]\setminus\{\zjstar\}} \frac{|H_j(\zjstar,b,z)|}{\langle \theta_{\zjstar}^*- \thetabstar, (\Sigmabstar)^{-1}(\theta_{\zjstar}^* - \thetabstar) \rangle} &\leq C\sqrt{\frac{k(d+\log n)}{n}}\label{Hcondi.2}\\
	\max_{\{z:\ell(z,z^\ast) \leq \tau\}}\sum_{j=1}^{n}\max_{b\in[k]\setminus\{\zjstar\}}\frac{F_j(\zjstar,b,z)^2\|\thetazjstarstar - \thetabstar\|^2}{\langle \thetazjstarstar - \thetabstar, (\Sigmabstar)^{-1}(\thetazjstarstar - \thetabstar) \rangle^2\ell(z,z^\ast)} &\leq Ck^3\br{\frac{\tau}{n}+\frac{1}{\Delta^2}+\frac{d^2}{n\Delta^2}}  \label{Fcondi.2}\\
	\max_{\{z:\ell(z,z^\ast) \leq \tau\}}\max_{j \in [n]}\max_{b\in[k]\setminus\{\zjstar\}}\frac{|G_j(\zjstar,b,z)|}{\langle \thetazjstarstar - \thetabstar, (\Sigmabstar)^{-1}(\thetazjstarstar - \thetabstar) \rangle} &\leq Ck\br{\frac{\tau}{n} + \frac{1}{\Delta}\sqrt{\frac{\tau}{n}} + \frac{d\sqrt{\tau}}{n\Delta}} \label{Gcondi.2}\\
	\max_{\{z:\ell(z,z^\ast) \leq \tau\}}\sum_{j=1}^{n}\max_{b\in[k]\setminus\{\zjstar\}}\frac{Q_j(\zjstar,b,z)^2\|\thetazjstarstar - \thetabstar\|^2}{\langle \thetazjstarstar - \thetabstar, (\Sigmabstar)^{-1}(\thetazjstarstar - \thetabstar) \rangle^2\ell(z,z^\ast)} &\leq  C\frac{k^3d^2}{\Delta^2}\br{\frac{\tau}{n} + \frac{1}{\Delta^2} + \frac{d^2}{n\Delta^2}} \label{Qcondi.2}\\
	\max_{\{z:\ell(z,z^\ast) \leq \tau\}}\sum_{j=1}^{n}\max_{b\in[k]\setminus\{\zjstar\}}\frac{K_j(\zjstar,b,z)^2\|\thetazjstarstar - \thetabstar\|^2}{\langle \thetazjstarstar - \thetabstar, (\Sigmabstar)^{-1}(\thetazjstarstar - \thetabstar) \rangle^2\ell(z,z^\ast)} &\leq C\frac{k^3d^2}{\Delta^2}\br{\frac{\tau}{n} + \frac{1}{\Delta^2} + \frac{d^2}{n\Delta^2}}   \label{Kcondi.2}\\
	\max_{\{z:\ell(z,z^\ast) \leq \tau\}}\max_{j \in [n]}\max_{b\in[k]\setminus\{\zjstar\}}\frac{|L_j(\zjstar,b,z)|}{\langle \thetazjstarstar - \thetabstar, (\Sigmabstar)^{-1}(\thetazjstarstar - \thetabstar) \rangle} &\leq C\frac{d}{\Delta^2}\sqrt{\frac{k(d+\log n)}{n}} \label{Lcondi.2}
	\end{align}
	with probability at least $1- n^{-C'}-\frac{4}{nd}$.
\end{lemma}
\begin{proof}
Under the conditions of Theorem \ref{thm:main2}, the inequalities \eqref{ine:lem:miscel.1}-\eqref{ine:keyine} hold with probability at least $1-n^{-C'}$. In the remaining proof, we will work on the event these inequalities hold.
	%We consider the events that inequalities \eqref{ine:lem:miscel.1}-\eqref{ine:lem:covspec.true.2} and \eqref{ine:keyine} in technical lemmas hold. 
	Hence, we can use the results from Lemma \ref{Cov.inv.rate} and \ref{Cov.inv.rate2}.
	Using the same arguments as in the proof of Lemma \ref{lem:conditions}, we can get \eqref{Hcondi.2},  \eqref{Fcondi.2} and \eqref{Gcondi.2}. 
	
	As for \eqref{Qcondi.2}, we first use Lemma \ref{lem:chi_4}  to have $\sum_{j=1}^{n}\|\epsilon_j\|^4 \leq 3nd$ with probability at least $1-4/(nd)$. Then, we have
	\begin{align*}
	\sum_{j=1}^{n}\max_{b\in[k]\setminus\{\zjstar\}}\frac{Q_j(\zjstar,b,z)^2\|\thetazjstarstar - \thetabstar\|^2}{\langle \thetazjstarstar - \thetabstar, (\Sigmabstar)^{-1}(\thetazjstarstar - \thetabstar) \rangle^2\ell(z,z^\ast)}\preceq~ &\sum_{j=1}^{n}\sum_{b=1}^{k}\frac{Q_j(\zjstar,b,z)^2}{\Delta^2\ell(z,z^*)}\\
	\leq~ &k\sum_{j=1}^{n}\|\epsilon_j\|^4\frac{\max_{a\in [k]}\|(\Sigmahataz)^{-1} - (\Sigmahatazstar)^{-1}\|^2}{\Delta^2\ell(z,z^*)}\\
	\preceq~ & \frac{k^3d^2}{\Delta^2}\br{\frac{\tau}{n} + \frac{1}{\Delta^2} + \frac{d^2}{n\Delta^2}},
	\end{align*}
	where the last inequality is due to \eqref{Sigmahatazazstar.inv} and the fact that $\ell(z,z^*)\leq \tau$. 
	
	Next for \eqref{Kcondi.2}, notice that by \eqref{cov_star_to_true}, \eqref{covrate}, and $\snr' \to \infty$, we have for any $1 \leq i \leq d$, $\frac{\lambdamin}{2} \leq \lambda_i(\Sigmahatazstar) \leq 2\lambdamax$ and 
	\begin{align*}
	\left|\log (1+\max_{a\in[k]} \frac{\|\Sigmahataz - \Sigmahatazstar\|_2}{\lambda_i(\Sigmahatazstar)})\right| \leq \left|\log (1-\max_{a\in[k]} \frac{\|\Sigmahataz - \Sigmahatazstar\|_2}{\lambda_i(\Sigmahatazstar)})\right|.
	\end{align*} 
	Thus by Lemma \ref{lem:Weyl}, we know
	\begin{align}
	&\max_{a\in[k]}\left|\log|\Sigmahataz |- \log|\Sigmahatazstar|\right| \notag\\
	=~ &\max_{a\in[k]}\left|\log \frac{|\Sigmahataz|}{|\Sigmahatazstar|}\right| \notag\\
	\leq~ & \left|\sum_{i=1}^{d}\log(1-\frac{\max_{a\in[k]}\|\Sigmahataz - \Sigmahatazstar\|_2}{\lambda_i(\Sigmahatazstar)})\right| \notag\\
	\leq & \sum_{i=1}^{d} \log\left(1 + \max_{a\in[k]}\frac{\|\Sigmahataz - \Sigmahatazstar\|_2}{\lambda_i(\Sigmahatazstar)}+ \frac{\max_{a\in[k]}\frac{\|\Sigmahataz - \Sigmahatazstar\|_2^2}{\lambda_i^2(\Sigmahatazstar)}}{1-\max_{a\in[k]}\frac{\|\Sigmahataz - \Sigmahatazstar\|_2}{\lambda_i(\Sigmahatazstar)}}\right) \notag\\
	\preceq & d\|\Sigmahataz - \Sigmahatazstar\|_2, \label{logdeterzzstar}
	\end{align}
	where the last inequality is due to the fact that $\lambda_i(\Sigmahatazstar)$ is at the constant rate, $\|\Sigmahataz - \Sigmahatazstar\|_2=o(1)$ and the inequality $\log(1+x) \leq x$ for any $x > 0$. \eqref{logdeterzzstar} yields to the inequality
	\begin{align*}
	\sum_{j=1}^{n}\max_{b\in[k]\setminus\{\zjstar\}}\frac{K_j(\zjstar,b,z)^2\|\thetazjstarstar - \thetabstar\|^2}{\langle \thetazjstarstar - \thetabstar, (\Sigmabstar)^{-1}(\thetazjstarstar - \thetabstar) \rangle^2\ell(z,z^\ast)}\preceq~ &\sum_{j=1}^{n}\frac{d^2\max_{a\in [k]}\|\Sigmahataz - \Sigmahatazstar\|^2}{\Delta^2\ell(z,z^*)}\\
	\preceq~ & \frac{k^2d^2}{\Delta^2}\br{\frac{\tau}{n} + \frac{1}{\Delta^2} + \frac{d^2}{n\Delta^2}}.
	\end{align*}
	Finally for \eqref{Lcondi.2}, by \eqref{cov_star_to_true} and the similar argument as \eqref{logdeterzzstar}, we can get
	\begin{align*}
	\max_{a\in[k]}\left|\log|\Sigmahatazstar| - \log|\Sigmaastar|\right| \preceq d\sqrt{\frac{k(d+\log n)}{n}}
	\end{align*}
	which implies \eqref{Lcondi.2}. We complete the proof.
\end{proof}
\begin{lemma}\label{lem:upperbound.2}
	With the same conditions as Theorem \ref{thm:main2}, for any sequence $\delta_n = o(1)$, we have
	\begin{align*}
	\xi_{\text{ideal}}(\delta_n) \leq n\ebr{-(1+o(1))\frac{\snr'^2}{8}}.
	\end{align*}
	with probability at least $1 - n^{-C'} - \exp(-\snr')$.
\end{lemma}
\begin{proof}
	Under the conditions of Theorem \ref{thm:main2}, the inequalities \eqref{ine:lem:miscel.1}-\eqref{ine:keyine} hold with probability at least $1-n^{-C'}$. In the remaining proof, we will work on the event these inequalities hold. Similar to the proof of Lemma \ref{lem:upperbound}, we have a decomposition $\xi_{\text{ideal}} \leq \sum_{i=1}^{6}M_i$ where
	\begin{align*}
	M_1 := \sum_{j=1}^{n}\sum_{b\in[k]\setminus\{\zjstar\}}\left\|\theta^*_{\zjstar} - \thetabstar\right\|^2\mathbb{I}\bigg\{\langle &\epsilon_j, (\Sigmabstar)^{-1}(  \thetazjstarstar   - \thetabstar) \rangle 
	+ \frac{1}{2}\langle \epsilon_j, ((\Sigmabstar)^{-1} - (\Sigmazjstarstar)^{-1})\epsilon_j \rangle \\
	&-\frac{1}{2}\log|\Sigmazjstarstar| + \frac{1}{2}\log|\Sigmabstar| 
	\leq -\frac{1-\delta-\bar{\delta}}{2}\langle   \thetazjstarstar   - \thetabstar,  (\Sigmabstar)^{-1}(  \thetazjstarstar   - \thetabstar)\rangle\bigg\}
	\end{align*}
	is the main term and 
	\begin{align*}
	M_2 := \sum_{j=1}^{n}\sum_{b\in[k]\setminus\{\zjstar\}}\left\|\theta^*_{\zjstar} - \thetabstar\right\|^2\mathbb{I}\bigg\{\langle &\epsilon_j, ((\Sigmahatbzstar)^{-1} - (\Sigmabstar)^{-1})(  \thetazjstarstar   - \thetabstar)  \rangle\leq -\frac{\bar{\delta}}{10} \langle   \thetazjstarstar   - \thetabstar,  (\Sigmabstar)^{-1}(  \thetazjstarstar   - \thetabstar)\rangle\bigg\}
	\end{align*}
	\begin{align*}
	M_3:= \sum_{j=1}^{n}\sum_{b\in[k]\setminus\{\zjstar\}}\left\|\theta^*_{\zjstar} - \thetabstar\right\|^2\mathbb{I}\left\{-\langle\epsilon_j, (\Sigmahatzjstarzstar)^{-1}(  \thetazjstarstar   -   \thetahatzjstarzstar  )\rangle\leq -\frac{\bar{\delta}}{10} \langle   \thetazjstarstar   - \thetabstar,  (\Sigmabstar)^{-1}(  \thetazjstarstar   - \thetabstar)\rangle\right\}
	\end{align*}
	\begin{align*}
	M_4:=\sum_{j=1}^{n}\sum_{b\in[k]\setminus\{\zjstar\}}\left\|\theta^*_{\zjstar} - \thetabstar\right\|^2\mathbb{I}\left\{-\langle\epsilon_j, (\Sigmahatbzstar)^{-1}(\thetahatbzstar - \thetabstar)\rangle\leq -\frac{\bar{\delta}}{10} \langle   \thetazjstarstar   - \thetabstar,  (\Sigmabstar)^{-1}(  \thetazjstarstar   - \thetabstar)\rangle\right\}
	\end{align*}
	\begin{align*}
	M_5:= \sum_{j=1}^{n}\sum_{b\in[k]\setminus\{\zjstar\}}\left\|\theta^*_{\zjstar} - \thetabstar\right\|^2\mathbb{I}\left\{\frac{1}{2}\langle \epsilon_j, ((\Sigmahatbzstar)^{-1}-(\Sigmabstar)^{-1})\epsilon_j\rangle \leq - \frac{\bar{\delta}}{10}\langle   \thetazjstarstar   - \thetabstar,  (\Sigmabstar)^{-1}(  \thetazjstarstar   - \thetabstar)\rangle\right\}
	\end{align*}
	\begin{align*}
	M_6:=\sum_{j=1}^{n}\sum_{b\in[k]\setminus\{\zjstar\}}\left\|\theta^*_{\zjstar} - \thetabstar\right\|^2\mathbb{I}\left\{-\frac{1}{2}\langle \epsilon_j, ((\Sigmahatzjstarzstar)^{-1}-(\Sigmazjstarstar)^{-1})\epsilon_j\rangle \leq -\frac{\bar{\delta}}{10} \langle   \thetazjstarstar   - \thetabstar,  (\Sigmabstar)^{-1}(  \thetazjstarstar   - \thetabstar)\rangle\right\}.
	\end{align*}
%	Then we have
%	\begin{align*}
%	\xi_{\text{ideal}} \leq \sum_{i=1}^{6}N_i.
%	\end{align*}
	Using the same arguments as the proof of Lemma \ref{lem:upperbound}, we can choose some $\bar{\delta} = \bar{\delta}_n=o(1)$ which is slowly diverging to zero satisfying 
	\begin{align*}
	\E M_i \leq n\ebr{-(1+o(1))\frac{\snr^{'2}}{2}} \quad \text{for }i=2,3,4.
	\end{align*} 
	As for $M_5$, by \eqref{cov_star_to_true} we have
	\begin{align*}
	M_5 \leq \sum_{j=1}^{n}\sum_{b\in[k]\setminus\{\zjstar\}}\left\|\theta^*_{\zjstar} - \thetabstar\right\|^2\mathbb{I}\left\{C\bar{\delta}\left\|  \thetazjstarstar   - \thetabstar\right\|^2 \leq \|w_j\|^2\sqrt{\frac{\log n}{n}}\right\},
	\end{align*} 
	where $C$ is a constant and $w_j \iid \mathn(0, I_d)$. Since there exists some constant $C'$ such that $\snr' \leq C'\Delta $, we can choose appropriate $\bar{\delta} = o(1)$ such that
	\begin{align*}
	\E M_5 &\leq \sum_{j=1}^{n}\sum_{b\in[k]\setminus\{\zjstar\}}\left\|\theta^*_{\zjstar} - \thetabstar\right\|^2\p\left\{C\bar{\delta}\left\|  \thetazjstarstar   - \thetabstar\right\|^2\sqrt{\frac{n}{\log n}} \leq \|w_j\|^2\right\}\\
	&\leq n\ebr{-(1+o(1))\frac{\snr^{'2}}{8}}.
	\end{align*}
	$M_6$ is essentially the same with $M_5$. Finally for $M_1$, using Lemma \ref{lem:1}, we have
	\begin{align*}
	&\p\bigg(\langle \epsilon_j, (\Sigmabstar)^{-1}(  \thetazjstarstar   - \thetabstar) \rangle 
	+ \frac{1}{2}\langle \epsilon_j, ((\Sigmabstar)^{-1} - (\Sigmazjstarstar)^{-1})\epsilon_j \rangle \\
	&\quad \quad -\frac{1}{2}\log|\Sigmazjstarstar| + \frac{1}{2}\log|\Sigmabstar| 
	\leq -\frac{1-\delta-\bar{\delta}}{2}\langle   \thetazjstarstar   - \thetabstar,  (\Sigmabstar)^{-1}(  \thetazjstarstar   - \thetabstar)\rangle\bigg)\\
	=& \p\bigg(\langle w_j, (\Sigma^*_{z^*_j})^\frac{1}{2}(\Sigmabstar)^{-1}(  \thetazjstarstar   - \thetabstar) \rangle 
	+ \frac{1}{2}\langle w_j, ( (\Sigma^*_{z^*_j})^\frac{1}{2}(\Sigmabstar)^{-1} (\Sigma^*_{z^*_j})^\frac{1}{2} - I_d)w_j \rangle \\
	&\quad \quad -\frac{1}{2}\log|\Sigmazjstarstar| + \frac{1}{2}\log|\Sigmabstar| 
	\leq -\frac{1-\delta-\bar{\delta}}{2}\langle   \thetazjstarstar   - \thetabstar,  (\Sigmabstar)^{-1}(  \thetazjstarstar   - \thetabstar)\rangle\bigg)\\
	&\leq \ebr{-(1-o(1)) \frac{\snr'_{z^*_j,b}}{8}}.
	\end{align*}
	%by Lemma \ref{lem:1}, we can know that
	Then we have 
	\begin{align*}
	\E M_1 \leq n\ebr{-(1+o(1))\frac{\snr^{'2}}{8}}.
	\end{align*}
	Using the Markov's inequality we complete the proof of Lemma \ref{lem:upperbound.2}. 
\end{proof}
\begin{proof}[Proof of Theorem \ref{thm:main2}]
By Lemmas \ref{lem:linear convergence.2}-\ref{lem:upperbound.2}, we can obtain the result by arguments used in the proof of Theorem \ref{thm:main1} and hence is omitted here.
\end{proof}

\section{Technical Lemmas}\label{sec:technical}

Here are the technical lemmas.
\begin{lemma}\label{lem:chi-square}
	For any $x>0$, we have 
\begin{align*}
\p(\chi_d^2 \geq d+2\sqrt{dx}+2x) \leq e^{-x},\\
\p(\chi_d^2 \leq d-2\sqrt{dx}) \leq e^{-x}.
\end{align*}
\end{lemma}
\begin{proof}
	These results are Lemma 1 of \citep{laurent2000adaptive}.
\end{proof}
\begin{lemma}\label{lem:miscel}
	For any $z^* \in [k]^n$ and $k \in [n]$, consider independent vectors $\epsilon_j \sim \mathn(0,\Sigma^*_{z^*_j})$ for any $j\in[n]$. Assume there exists a constant $\lambdamax >0$ such that $  \|\Sigma_a^*\|\leq \lambdamax$ for any $a\in[k]$. Then, for any constant $C'>0$, there exists some constant $C>0$ only depending on $C',\lambda_{\max}$ such that
	\begin{align}
	\max_{a\in[k]}\left\|\frac{\sumzjstareqa\epsilon_j}{\sqrt{\sumzjstareqa}}\right\| &\leq  C\sqrt{d+\log n}, \label{ine:lem:miscel.1}\\ 
	\max_{a\in[k]}\frac{1}{d+\sumzjstareqa}\left\| \sumzjstareqa\epsilon_j\epsilon_j^T\right\| &\leq C, \label{ine.lem.miscel.3}\\
		\max_{T \subset [n]}\left\| \frac{1}{\sqrt{|T|}}\sum_{j\in T}\epsilon_j\right\| &\leq C\sqrt{d+n}, \label{ine:lem:miscel.2}\\
		\max_{a\in[k]}\max_{T\subset\{j:z_j^*=a\}}\left\| \frac{1}{\sqrt{|T|(d+\sumzjstareqa)}}\sum_{j\in T}\epsilon_j\right\|&\leq C, \label{ine:lem:covspec.true.2}
	\end{align}
	with probability at least $1-n^{-C'}$. We have used the convention that $0/0 = 0$.
\end{lemma}
\begin{proof}
	Note that $\epsilon_j$ is sub-Gaussian with parameter $\lambdamax$ which is a constant.
%	 and it's suffices to prove the results with respect to the vectors $\frac{\epsilon_j}{\sqrt{\lambdamax}}$ which is sub-Gaussian with parameter 1. 
%	Thus, we can assume $\lambdamax = 1$ and then, 
	The inequalities \eqref{ine:lem:miscel.1} and \eqref{ine:lem:miscel.2} are respectively Lemmas A.4, A.1 in \citep{lu2016statistical}. The inequality \eqref{ine.lem.miscel.3} is a slight extension of Lemma A.2 in \citep{lu2016statistical}. This extension can be done by a standard union bound argument. The proof of \eqref{ine:lem:covspec.true.2} is identical to that of \eqref{ine:lem:miscel.2}.
\end{proof}

\begin{lemma}\label{lem:covspec.true}
	%For $z^* \in [k]^n$ and $k \in [p]$, consider independent vectors $\epsilon_j \sim \mathn(0,\Sigma^*_{z^*_j})$ for any $j\in[n]$. Assume there exist a constant $\lambdamax >0$ such that $  \lambda_{d}(\Sigma_a^*)\leq \lambdamax$ for any $a\in[k]$, and a constant $\alpha>0$ such that $\min_{a\in k}\sum_{j=1}^n\{z^*_j = a\}\geq \frac{\alpha n}{k}$. 
Consider the same assumptions as in Lemma \ref{lem:miscel}.	
	Assume additionally $\min_{a\in k}\sum_{j=1}^n\mathbb{I}\{z^*_j = a\}\geq \frac{\alpha n}{k}$ for some constant $\alpha>0$ and $\frac{k(d+\log n)}{n}=o(1)$.
	Then, for any constant $C'>0$, there exists some constant $C>0$ only depending on $\alpha,C',\lambdamax$ such that
	\begin{align}\label{ine:lem:covspec.true}
	\max_{a\in[k]}\left\|\frac{1}{\sum_{j=1}^{n}\mathbb{I}\{z_j^*=a\}}\sum_{j=1}^{n}\mathbb{I}\{z_j^*=a\}\epsilon_j\epsilon_j^T - \Sigmaastar\right\| &\leq C\sqrt{\frac{k(d+\log n)}{n}},
	\end{align}
	with probability at least $1-n^{-C'}$.
\end{lemma}
\begin{proof}
Note that we have $\epsilon_j =\Sigma^{*\frac{1}{2}}_{z^*_j} \eta_j$ where $\eta_j\iid \mathn(0,I_d)$ for any $j\in[n]$. Since $\max_a\|\Sigma^*_a\|\leq \lambdamax$, we have
\begin{align*}
\max_{a\in[k]}\left\|\frac{1}{\sum_{j=1}^{n}\mathbb{I}\{z_j^*=a\}}\sum_{j=1}^{n}\mathbb{I}\{z_j^*=a\}\epsilon_j\epsilon_j^T - \Sigmaastar\right\|  \leq  \lambdamax\max_{a\in[k]}\left\|\frac{1}{\sum_{j=1}^{n}\mathbb{I}\{z_j^*=a\}}\sum_{j=1}^{n}\mathbb{I}\{z_j^*=a\}\eta_j\eta_j^T - I_d\right\|.
\end{align*}
Define
		\begin{align*}
		Q_a = \frac{1}{\sum_{j=1}^{n}\mathbb{I}\{z_j^*=a\}}\sum_{j=1}^{n}\mathbb{I}\{z_j^*=a\}\eta_j\eta_j^T - I_d.
		\end{align*}
Take $S^{d-1} = \{y \in \mathbb{R}^d:\|y\| = 1\}$ and $N_\epsilon = \{v_1, \cdots, v_{|N_\epsilon|}\}$ is an $\epsilon$-covering  of $S^{d-1}$. In particular, we pick $\epsilon < \frac{1}{4}$, then $|N_\epsilon| \leq 9^d $. By the definition of the $\epsilon$-covering, we have
		\begin{align*}
		\left\|Q_a\right\| \leq \frac{1}{1-2\epsilon}\max_{i = 1,\cdots, |N_\epsilon|}|v_i^TQ_av_i| \leq 2\max_{i = 1,\cdots, |N_\epsilon|}|v_i^TQ_av_i|.
		\end{align*}
		For any $v \in N_\epsilon$,
		\begin{align*}
		v^TQ_av = \frac{1}{\sum_{j=1}^{n}\mathbb{I}\{z_j^*=a\}}\sum_{j=1}^{n}\mathbb{I}\{z_j^*=a\}(v^T\eta_j\eta_j^Tv - 1).
		\end{align*} 
		Denote $n_a = \sum_{j=1}^{n}\mathbb{I}\{z_j^*=a\}$. Then $\sum_{j=1}^{n}\mathbb{I}\{z_j^*=a\}v^T\eta_j\eta_j^Tv \sim \chi^2_{n_a}$.
Using Lemma \ref{lem:chi-square}, we have
\begin{align*}
		P(\max_{a\in[k]}\|Q_a\| \geq t) &\leq \sum_{a=1}^{k}P(\|Q_a\| \geq t) \notag \\ 
		&\leq \sum_{a=1}^{k}\sum_{i=1}^{|N_\epsilon|}P(|v_i^TQ_av_i| \geq t/2)  \notag \\ 
		%\label{ine:lem:covspec.ture.1}
		&\leq \sum_{a=1}^k 2\exp\Biggl\{-\frac{n_a}{8} \min\{t,t^2\}+ d\log 9\Biggr\}.
		\end{align*}  
		Since $\frac{k(d+\log n)}{n}=o(1)$ and $n_a \geq \alpha n/k$ where $\alpha$ is a  constant, we can take $t= C''\sqrt{\frac{k(d+\log n)}{n}}$ for some large constant $C''$ and the proof is complete.
\end{proof}

\begin{lemma}
	Consider the same assumptions as in Lemma \ref{lem:miscel}.	
	%For $z^* \in [k]^n$ and $k \in [p]$, consider independent vectors  $\epsilon_j \sim \mathn(0,\Sigma^*_{z^*_j})$ for any $j=1,\cdots,n$. Assume there exist constants $0<\lambdamin \leq \lambdamax $ such that for any $a\in[k]$, $\lambdamin \leq \lambda_{1}(\Sigma_a^*) \leq  \lambda_{d}(\Sigma_a^*)\leq \lambdamax$. 
	Then, for any $ s  = o(n)$ and for any constant $C'>0$, there exists some constant $C>0$ only depending on $C',\lambdamax$ such that
	\begin{align}
	\max_{T \subset [n]: |T| \leq  s }\frac{1}{|T|\log\frac{n}{|T|}+\min\{1,\sqrt{\abs{T}}\}d}\left\|\sum_{j \in T}\epsilon_j\epsilon_j^T\right\| \leq C \label{ine:keyine},
	\end{align}
	with probability at least $1-n^{-C'}$. We have used the convention that $0/0 = 0$.
\end{lemma}
\begin{proof}
Consider any $a\in [ s ]$ and a fixed $T\subset[n]$ such that $\abs{T}=a$. Similar to the proof of Lemma \ref{lem:covspec.true}, we can take $S^{d-1} = \{y \in \mathbb{R}^d:\|y\| = 1\}$ and its $\epsilon$-covering $N_\epsilon$ with $\epsilon < \frac{1}{4}$ and $|N_\epsilon| \leq 9^d$. Then we have
\begin{align*}
\|\sum_{j \in T}\epsilon_j\epsilon_j^T\|  = \sup_{\norm{w}=1} \sum_{j\in T} (w^T\epsilon_j)^2 \leq 2 \max_{w\in N_\epsilon}\sum_{j\in T} (w^T\epsilon_j)^2.
\end{align*}
Note that $w^T\epsilon_j/\sqrt{\lambdamax}$ is a sub-Gaussian random variable with parameter 1. By \cite{hsu2012tail}, for any fixed $w\in N_\epsilon$, we have
\begin{align*}
\pbr{\sum_{j\in T} (w^T\epsilon_j)^2 \geq \lambdamax \br{a+ 2\sqrt{at} + 2t}}\leq \ebr{-t}.
\end{align*}
Since $a = o(n)$, there exists a constant $C_0$ such that $2a \leq C_0a\log\frac{n}{a}$. We can take $t = \tilde{C}(a\log \frac{n}{a}+d)$ with $\tilde{C} = \frac{C}{16} - \frac{C_0}{4}$, then $a+2\sqrt{at}+2t \leq \frac{C}{4}(a\log\frac{n}{a}+d)$. Thus, 
	 \begin{align*}
	 \pbr{ \sum_{j\in T} (w^T\epsilon_j)^2 \geq \frac{C}{4}(a\log\frac{n}{a}+d)} \leq \exp\bigg(-\tilde{C}(a\log\frac{n}{a}+d)\bigg).
	 \end{align*}
	 Hence, we have
	 \begin{align*}
	 \pbr{ \|\sum_{j \in T}\epsilon_j\epsilon_j^T\| \geq \frac{C}{2}(a\log\frac{n}{a}+d)}  \leq 9^d\exp\bigg(-\tilde{C}(a\log\frac{n}{a}+d)\bigg).
	 \end{align*}
	 As a result,
\begin{align*}
	 \p\bigg\{\max_{T \subset [n], 1 \leq |T| \leq  s }\frac{1}{|T|\log\frac{n}{|T|}+d}\|\sum_{j \in T}\epsilon_j\epsilon_j^T\| \geq C  \bigg\}\leq &\sum_{a=1}^{ s }\p\biggl\{ \max_{|T|=a}\|\sum_{j \in T}\epsilon_j\epsilon_j^T\| \geq C(a\log\frac{n}{a}+d)\biggr\}\\
	 \leq  & \sum_{a=1}^{ s } \binom{n}{a} \max_{|T|=a} \p\biggl\{ \|\sum_{j \in T}\epsilon_j\epsilon_j^T\| \geq C(a\log\frac{n}{a}+d)\biggr\}\\
	 &\leq   \sum_{a=1}^{ s }  \binom{n}{a} 9^d\exp\bigg(-\tilde{C}(a\log\frac{n}{a}+d)\bigg).
	 \end{align*}
	 Since $a\log\frac{n}{a}$ is an increasing function when $a \in [1, s ]$ and $a\log\frac{n}{a} \geq \log n \geq \log  s $, a choice of $\tilde{C} = 3 + C'$, that is $C = 16C'+4C_0+48$, can yield the desired result. 
	 
	 Finally, to allow $\abs{T}=0$, we note that $d\leq \min\{1,\sqrt{\abs{T}}\}d$. The proof is complete.
\end{proof}

\begin{lemma}
	For  any $z^* \in [k]^n$ and $k \in [n]$, assume $\min_{a\in k}\sum_{j=1}^n\mathbb{I}\{z^*_j = a\}\geq \frac{\alpha n}{k}$ and $\ell(z,z^*) = o(\frac{n\Delta^2}{k})$, then 
	\begin{align}\label{ine:lem:compare.num}
	\max_{a\in [k]}\frac{\sumzjstareqa}{\sumzjeqa} \leq 2.
	\end{align}
\end{lemma}
\begin{proof}
	For any $z \in [k]^n$ such that $\ell(z,z^*) = o(n)$ and any $a\in[k]$, we have
	\begin{align}
	\sumzjeqa &\geq \sumzjstareqa - \sum_{j=1}^{n}\mathbb{I}\{z_j \neq z_j^*\} \notag\\
	&\geq \sumzjstareqa - \frac{\ell(z,z^*)}{\Delta^2} \notag\\
	&\geq  \frac{\alpha n}{2k} \label{num.zja},
	\end{align}
	which implies
	\begin{align*}
	\frac{\sumzjstareqa}{\sumzjeqa} &\leq \frac{\sumzjeqa + \sum_{j=1}^{n}\mathbb{I}\{z_j \neq z_j^*\}}{\sumzjeqa}\\
	&\leq 1+\frac{\alpha n/2k}{\sumzjeqa}\\
	& \leq 2.
	\end{align*}
	Thus, we obtain \eqref{ine:lem:compare.num}.
\end{proof}
The next lemma is the famous Weyl's Theorem and we omit the proof here.
\begin{lemma}[Weyl's Theorem]\label{lem:Weyl}
	Let $A$ and $B$ be  any two $d\times d$ symmetric real matrix. Then for any $1 \leq i \leq d$, we have
	\begin{align*}
	\lambda_i(A + B) \leq \lambda_d(A) + \lambda_i(B).
	\end{align*}
\end{lemma}

In the following lemma, we are going to analyze estimation errors of $\{\Sigma^*_a\}_{a\in[k]}$ under the anisotropic GMMs. For any $z\in[k]^n$ and for any $z\in[k]$, recall the definitions
\begin{align*}
 \hat\theta_a(z) &= \frac{\sum_{j\in[n]} Y_j\indic{z_j=a}}{\sum_{j\in[n]}\indic{z_j=a}},\\
  \hat\Sigma_a(z)  &=  \frac{\sum_{j\in[n]} (Y_j -   \hat\theta_a(z)) (Y_j -   \hat\theta_a(z))^T \indic{z_j=a}}{\sum_{j\in[n]}\indic{z_j=a}}.
\end{align*}

\begin{lemma}\label{Cov.inv.rate}
	 For any $z^* \in [k]^n$ and $k \in [n]$, consider independent vectors $Y_j = \theta^*_{z^*_j} + \epsilon_j$ where $\epsilon_j \sim \mathn(0,\Sigma^*_{z^*_j})$ for any $j\in[n]$. Assume there exist constants $\lambdamin, \lambdamax>0$ such that $\lambdamin \leq \lambda_{1}(\Sigma_a^*) \leq  \lambda_{d}(\Sigma_a^*)\leq \lambdamax$ for any $a\in[k]$, and a constant $\alpha>0$ such that $\min_{a\in k}\sum_{j=1}^n\mathbb{I}\{z^*_j = a\}\geq \frac{\alpha n}{k}$. Assume $\frac{k(d+\log n)}{n}=o(1)$ and $\frac{\Delta}{k} \rightarrow\infty$. Assume \eqref{ine:lem:miscel.1}-\eqref{ine:keyine} hold.
	 Then for any $\tau=o(n)$ and for any constant $C'>0$, there exists some constant $C>0$ only depending on $\alpha,\lambdamax,C'$ such that
	 \begin{align}
	\max_{a\in [k]}\left\|\thetahatazstar - \thetaastar\right\| & \leq C\sqrt{\frac{k(d+\log n)}{n}},	 \label{thetaazstartrue}\\
	 \max_{a\in[k]}\left\|\thetahataz - \thetahatazstar\right\| & \leq C \br{\frac{k}{n\Delta}\ell(z,z^*) + \frac{k\sqrt{d+n}}{n\Delta}\sqrt{\ell(z,z^*)}}, \label{thetazzstar}\\
	 \max_{a\in[k]}\left\|\Sigmahatazstar - \Sigmaastar\right\|& \leq C \sqrt{\frac{k(d+\log n)}{n}}, \label{cov_star_to_true}\\
	\max_{a\in[k]}\left\|\hat{\Sigma}_a(z) - \hat{\Sigma}_a(z^*)\right\| &\leq C\br{\frac{k}{n}\ell(z,z^*)+\frac{k\sqrt{n\ell(z,z^*)}}{n\Delta}+\frac{kd}{n\Delta}\sqrt{\ell(z,z^*)}},\label{covrate}
	\end{align}
	 for all $z$  such that $\ell(z,z^*)\leq \tau$.
\end{lemma}
	 \begin{proof}
	 Using (\ref{ine:lem:miscel.1}) we obtain (\ref{thetaazstartrue}).
	By the same argument of (118) in \citep{gao2019iterative}, we can obtain (\ref{thetazzstar}). By \eqref{ine:lem:miscel.1} and \eqref{ine:lem:covspec.true} and \eqref{thetaazstartrue}, we can obtain (\ref{cov_star_to_true}). In the remaining of the proof, we will establish (\ref{Sigmahatazazstar.inv}).

	Since $\frac{k(d+\log n)}{n}=o(1)$, we have  $\|\Sigmahatazstar\|\lesssim 1$ for any $a\in[k]$. 
	%To bound $\|(\hat{\Sigma}_a(z))^{-1} - (\hat{\Sigma}_a(z^*))^{-1}\|$, we will first 
	The difference $\hat{\Sigma}_a(z) - \hat{\Sigma}_a(z^*)$  will be decomposed into several terms.
	We  notice that
	\begin{align}\label{covzzstar}
	\left\|\hat{\Sigma}_a(z) - \hat{\Sigma}_a(z^*)\right\| \leq S_1+S_2,
	\end{align}
	where
	\begin{align*}
	S_1 := \bigg\|\frac{1}{\sum \mathbb{I}\{z_j=a\}}\sum_{j=1}^{n}\mathbb{I}\{z_j=a\}\bigg((Y_j-\thetahataz)(Y_j-\thetahataz)^T-(Y_j-\thetahatazstar)(Y_j-\thetahatazstar)^T\bigg)\bigg\|,
	\end{align*}
	and
	\begin{align*}
	S_2 := \bigg\|&\br{\frac{1}{\sum \indzja}-\frac{1}{\sum \indzjstara}}\sum_{j=1}^{n}\indzjstara(Y_j-\thetahatazstar)(Y_j-\thetahatazstar)^T\bigg\|.
	%&-\frac{1}{\sum \indzjstara}\sum_{j=1}^{n}\indzjstara(Y_j-\thetahatazstar)(Y_j-\thetahatazstar)^T\bigg\|.
	\end{align*}
	Also, we  notice that 
	\begin{align}\label{1o}
	S_1 \leq L_1 + L_2 + L_3,
	\end{align}
	where
	\begin{align*}
	L_1 &:= \bigg\|\frac{1}{\sumzjeqa}\sum_{j=1}^{n}\mathbb{I}\{z_j = z_j^*=a\}\bigg((Y_j-\thetahataz)(Y_j-\thetahataz)^T - (Y_j-\thetahatazstar)(Y_j-\thetahatazstar)^T\bigg)\bigg\|,\\
	L_2 &:= \bigg\|\frac{1}{\sumzjeqa}\sum_{j=1}^{n}\mathbb{I}\{z_j = a, z_j^*\neq a\}(Y_j-\thetahataz)(Y_j-\thetahataz)^T\bigg\|,\\
	L_3 &:= \bigg\|\frac{1}{\sumzjeqa}\sum_{j=1}^{n}\mathbb{I}\{z_j \neq a, z_j^*= a\}(Y_j-\thetahatazstar)(Y_j-\thetahatazstar)^T\bigg\|.
	\end{align*}
	For $L_1$, we have 
	\begin{align}
	L_1 &\leq \bigg\|\frac{1}{\sumzjeqa}\sum_{j=1}^{n}\mathbb{I}\{z_j = z_j^*=a\}(\thetahataz - \thetahatazstar)(\thetahataz - \thetahatazstar)^T\bigg\| \notag\\
	&\quad + 2\bigg\|\frac{1}{\sumzjeqa}\sum_{j=1}^{n}\mathbb{I}\{z_j = z_j^*=a\}(Y_j - \thetahatazstar)(\thetahataz - \thetahatazstar)^T\bigg\| \notag\\
	%&\quad + \bigg\|\frac{1}{\sumzjeqa}\sum_{j=1}^{n}\mathbb{I}\{z_j = z_j^*=a\}(\thetahataz - \thetahatazstar)(Y_j - \thetahatazstar)^T\bigg\| \notag\\
	& \preceq \left\|\thetahataz - \thetahatazstar\right\|^2  \frac{\sumzjstareqa}{\sumzjeqa} + \left\|\thetaastar - \thetahatazstar\right\|\left\|\thetahataz - \thetahatazstar\right\|\frac{\sumzjstareqa}{\sumzjeqa} \notag \\
	&\quad + \left\|\thetahataz - \thetahatazstar\right\|\left\|\frac{1}{\sumzjeqa}\sum_{j=1}^{n}\mathbb{I}\{z_j = z_j^*=a\}\epsilon_j\right\| \label{quanone}.
	\end{align}
	By \eqref{ine:lem:covspec.true.2}, \eqref{ine:lem:compare.num}, \eqref{num.zja}, we have uniformly for any $a \in [k]$,
	\begin{align}
	\left\|\frac{1}{\sumzjeqa}\sum_{j=1}^{n}\mathbb{I}\{z_j = z_j^*=a\}\epsilon_j\right\| &\preceq \frac{\sqrt{\sum_{j=1}^{n}\mathbb{I}\{z_j = z_j^*=a\}}}{\sumzjeqa}\sqrt{d + \sumzjstareqa} \notag\\
	&\preceq 1 \label{const}.
	\end{align}
	Since  $\max_{a\in [k]}\left\|\thetahatazstar - \thetaastar\right\| =o(1)$, by \eqref{ine:lem:compare.num}, \eqref{thetazzstar}, \eqref{thetaazstartrue}, \eqref{quanone}, and \eqref{const}, we have uniformly for any $a \in [k]$,
	\begin{align}\label{circ1}
	L_1 \preceq \left\|\thetahataz - \thetahatazstar\right\| \preceq  \frac{k}{n\Delta}\ell(z,z^*) + \frac{k\sqrt{d+n}}{n\Delta}\sqrt{\ell(z,z^*)}.
	\end{align}
	To bound $L_2$, we first give the following simple fact. For any positive integer $m$ and any $\{u_j\}_{j\in[m]},\{v_j\}_{j\in[m]}\in\mathr^d$, we have $\|\sum_{j\in[m]}(u_j + v_j)(u_j + v_j)^T\| \leq 2\|\sum_{j\in[m]}u_j u_j^T\| + 2\|\sum_{j\in[m]}v_j v_j^T\|$. Hence,
	for $L_2$, we have  the following decomposition
	\begin{align}
	L_2 \leq 2R_1 +2R_2,
	\end{align}
	where
	\begin{align*}
	R_1&:= \bigg\|\frac{1}{\sumzjeqa}\sum_{j=1}^{n}\mathbb{I}\{z_j = a,z_j^*\neq a\}(Y_j - \thetaastar)(Y_j - \thetaastar)^T\bigg\|,\\
	R_2&:= \bigg\|\frac{1}{\sumzjeqa}\sum_{j=1}^{n}\mathbb{I}\{z_j = a,z_j^*\neq a\}(\thetaastar - \thetahataz)(\thetaastar - \thetahataz)^T\bigg\|.
	%R_3&:=\bigg\|\frac{1}{\sumzjeqa}\sum_{j=1}^{n}\mathbb{I}\{z_j = a,z_j^*\neq a\}(Y_j - \thetaastar)(\thetaastar - \thetahatazstar)^T\bigg\|.
	\end{align*}
%	\begin{align*}
%	R_4:= \bigg\|\frac{1}{\sumzjeqa}\sum_{j=1}^{n}\mathbb{I}\{z_j = a,z_j^*\neq a\}(\thetaastar - \thetahatazstar)(Y_j - \thetaastar)^T\bigg\|.
%	\end{align*}
	Since $\max_{a\in [k]}\sum_{j=1}^{n}\mathbb{I}\{z_j = a,z_j^*\neq a\} \leq \frac{\ell(z,z^*)}{\Delta^2}$, we have
	\begin{align}
	R_2 &\leq \left\|\thetaastar - \thetahataz\right\|^2\frac{\sum_{j=1}^{n}\mathbb{I}\{z_j = a,z_j^*\neq a\}}{\sumzjeqa} \notag\\
	&\preceq \br{\left\|\thetahataz - \thetahatazstar\right\|^2 + \left\|\thetahatazstar - \thetaastar\right\|^2}\frac{k\ell(z,z^*)}{n\Delta^2}.
	\end{align}
By \eqref{ine:keyine} and the fact that $\max_{a\in[k]}\sum_{j=1}^{n}\mathbb{I}\{z_i = a,z_i^*\neq a\} \leq \frac{\ell(z,z^*)}{\Delta^2}$, we also have
\begin{align*}
R_1&\leq 2 \bigg\|\frac{1}{\sumzjeqa}\sum_{j=1}^{n}\mathbb{I}\{z_j = a,z_j^*\neq a\}(\theta^*_{z^*_j} - \theta^*_{z_j})(\theta^*_{z^*_j} - \theta^*_{z_j})^T\bigg\| \\
&\quad + 2 \bigg\|\frac{1}{\sumzjeqa}\sum_{j=1}^{n}\mathbb{I}\{z_j = a,z_j^*\neq a\}\epsilon_j\epsilon_j^T\bigg\|  \\
&\leq 2 \frac{\sum_{j=1}^{n}\mathbb{I}\{z_j = a,z_j^*\neq a\}\|\theta^*_{z^*_j} - \theta^*_{z_j}\|^2 }{\sumzjeqa} + 2 \bigg\|\frac{1}{\sumzjeqa}\sum_{j=1}^{n}\mathbb{I}\{z_j = a,z_j^*\neq a\}\epsilon_j\epsilon_j^T\bigg\|  \\
&\lesssim  \frac{k\ell(z,z^*)}{n} +  \frac{\frac{\ell(z,z^*)}{\Delta^2} \log \frac{n\Delta^2}{\ell(z,z^*)}+d\sqrt{\frac{\ell(z,z^*)}{\Delta^2}}}{n/k}.
\end{align*}

We are going to simplify the above bounds for $R_1,R_2$. Under the assumption that $\frac{k(d+\log n)}{n}=o(1)$, $\Delta/k \rightarrow\infty$, and $\ell(z,z^*)\leq \tau=o(n)$, we have $\max_{a\in [k]}\|\thetahataz - \thetahatazstar\| = o(1)$, $\max_{a\in [k]}\|\thetahatazstar - \thetaastar\| = o(1)$, and $\frac{k\ell(z,z^*)}{n\Delta^2} =o(1)$. Hence $R_2\lesssim \frac{k\ell(z,z^*)}{n\Delta^2}$. Also we have
\begin{align*}
\frac{k\ell(z,z^*)}{n\Delta^2} \log \frac{n\Delta^2}{\ell(z,z^*)} = \frac{k\sqrt{\ell(z,z^*)}}{n\Delta} \sqrt{\frac{\ell(z,z^*)}{\Delta^2} \br{\log \frac{n\Delta^2}{\ell(z,z^*)} }^2} \leq  \frac{k\sqrt{n\ell(z,z^*)}}{n\Delta}.
\end{align*}
where in the last inequality, we use the fact that $x(\log(n/x))^2$ is an increasing function of $x$ when $0<x =o(n)$. Then,
\begin{align*}
L_2 \lesssim \frac{k\sqrt{n\ell(z,z^*)}}{n\Delta} + \frac{k}{n}\ell(z,z^*)+\frac{kd}{n\Delta}\sqrt{\ell(z,z^*)}.
\end{align*}
Since $L_3$ is similar to $L_2$, by \eqref{1o} we have uniformly for any $a \in [k]$
	\begin{align}
	S_1 \preceq \frac{k\sqrt{n\ell(z,z^*)}}{n\Delta} + \frac{k}{n}\ell(z,z^*)+\frac{kd}{n\Delta}\sqrt{\ell(z,z^*)}.
	\end{align}
	
	To bound $S_2$, by (70) in \citep{gao2019iterative}, we have uniformly for any $a \in [k]$,
	\begin{align*}
	S_2 =  \frac{\left|\sumzjstareqa - \sumzjeqa\right|}{\sumzjeqa}  \norm{\hat \Sigma_a(z^*)}^2\lesssim  \frac{k}{n}\frac{\ell(z,z^*)}{\Delta^2},
	\end{align*}
	where we use (\ref{cov_star_to_true}).
	Since $\frac{k}{n}\frac{\ell(z,z^*)}{\Delta^2} \preceq \frac{k\sqrt{n\ell(z,z^*)}}{n\Delta}$, by \eqref{covzzstar} and the facts that $\ell(z,z^*) \leq \tau = o(n)$ we have
	\begin{align*} 
	\max_{a\in[k]}\left\|\hat{\Sigma}_a(z) - \hat{\Sigma}_a(z^*)\right\|  \preceq \frac{k\sqrt{n\ell(z,z^*)}}{n\Delta} + \frac{k}{n}\ell(z,z^*)+\frac{kd}{n\Delta}\sqrt{\ell(z,z^*)}.
	\end{align*}
	\end{proof}
	
\begin{lemma}\label{Cov.inv.rate2}
Under the same assumption  as in Lemma \ref{Cov.inv.rate}, if additional we assume $kd=O(\sqrt{n})$ and $\tau =o(n/k)$, there exists some constant $C>0$ only depending on $\alpha,\lambdamin,\lambdamax,C'$ such that
\begin{align}
\max_{a\in [k]}\left\|(\hat{\Sigma}_a(z))^{-1} - (\hat{\Sigma}_a(z^*))^{-1}\right\| &\leq C\br{\frac{k}{n}\ell(z,z^*)+\frac{k\sqrt{n\ell(z,z^*)}}{n\Delta}+\frac{kd}{n\Delta}\sqrt{\ell(z,z^*)}}.\label{Sigmahatazazstar.inv}
\end{align}
\end{lemma}	
	\begin{proof}
	By (\ref{cov_star_to_true}) we have $\max_{a\in[k]}\|\hat{\Sigma}_a(z^*)\|,\max_{a\in[k]}\|(\hat{\Sigma}_a(z^*))^{-1}\|\lesssim 1$.
	By \eqref{covrate} we also have $\max_{a\in[k]}\|\Sigmahataz\|, \max_{a\in[k]}\|(\Sigmahataz)^{-1}\|\lesssim 1$. Hence, 
	\begin{align}\label{keyorder}
	\max_{a\in[k]}\left\|(\hat{\Sigma}_a(z))^{-1} - (\hat{\Sigma}_a(z^*))^{-1}\right\| \leq &\max_{a\in[k]}\left\|(\Sigmahatazstar)^{-1}\right\|\left\|\hat{\Sigma}_a(z) - \hat{\Sigma}_a(z^*)\right\|\left\|(\Sigmahataz)^{-1}\right\| \notag\\
	\preceq &\frac{k\sqrt{n\ell(z,z^*)}}{n\Delta} + \frac{k}{n}\ell(z,z^*)+\frac{kd}{n\Delta}\sqrt{\ell(z,z^*)}.
	\end{align}
	\end{proof}

\begin{lemma}\label{lem:chi_4}
Let $W_i\iid\chi^2_d$ for any $i\in[n]$ where $n,d$ are positive integers. Then we have 
\begin{align*}
\pbr{\sum_{i=1}^n W_i^2 \geq 3nd^2} \leq \frac{4}{nd}.
\end{align*}
\end{lemma}
\begin{proof}
We have $\E \sum_{i=1}^n W_i^2 =nd(d+2)$ and $\E  \sum_{i=1}^n W_i^4 =nd(d+2)(d+4)(d+6)$. Then we have $\text{Var}\br{ \sum_{i=1}^n W_i^2} =8nd(d+2)(d+3)$. Then we obtain the desired result by Chebyshev's inequality.
\end{proof}

\begin{proof}[Proof of Lemma \ref{lem:sameorder}]
Consider any $a \neq b \in [k]$.
We are going to prove
\begin{align}\label{eqn:snr_delta_bound}
	\frac{-\sqrt{\lambdamax}+\sqrt{\lambdamax+\frac{\lambdamin(\lambdamin+\lambdamax)}{2\lambdamax}}}{\lambdamin+\lambdamax}\norm{\thetaastar - \thetabstar}\leq \snr'_{a,b} \leq \lambdamin^{-1/2}\norm{\thetaastar - \thetabstar} + \sqrt{\frac{3}{2}d} + \sqrt{d\log\frac{\lambdamax}{\lambdamin}}.
	\end{align}
	
	We first prove the upper bound. Denote $\Xi_{a,b} = \thetaastar - \thetabstar$.  Since we have assumed $\snr' \to \infty$, we have that $0 \notin \mathcal{B}_{a,b}$.
	%We first prove the upper bound. Since we have assumed $\snr' \to \infty$, it comes that $0 \notin \mathcal{B}_{a,b}$. Recall that $\Xi_{a,b} = \thetaastar - \thetabstar$ and let $x = -(\Sigmaastar)^{-\frac{1}{2}}\Xi_{a,b}+y$. Then, we have
	Note that we have an equivalent expression of $\mathcal{B}_{a,b} $:
	\begin{align*}
	\mathcal{B}_{a,b} = \bigg\{-(\Sigmaastar)^{-\frac{1}{2}}\Xi_{a,b}+y \in \mathbb{R}^d: &2y^T(\Sigmaastar)^{-\frac{1}{2}}\Xi_{a,b}+y^T\br{\Sigma_a^{*\frac{1}{2}}\Sigma_b^{*-1}\Sigma_a^{*\frac{1}{2}}-I}y \\
	&- \log|\Sigmaastar|+\log|\Sigmabstar| - \Xi_{a,b}^T(\Sigmaastar)^{-1}\Xi_{a,b} \leq 0\bigg\}.
	\end{align*}
	We consider the following scenarios.
	
	\emph{(1).} If $\lambda_{1}\br{\Sigma_a^{*\frac{1}{2}}\Sigma_b^{*-1}\Sigma_a^{*\frac{1}{2}}-I} \geq 0 $, we have $|\Sigmaastar| \geq |\Sigmabstar|$. Let $y=0$, we can know $-(\Sigmaastar)^{-\frac{1}{2}}\Xi_{a,b} \in \mathcal{B}_{a,b}$. This tells us $\snr'_{a,b} \leq \left\|-(\Sigmaastar)^{-\frac{1}{2}}\Xi_{a,b}\right\| \leq \lambdamin^{-1/2}\norm{\Xi_{a,b}}$.
	
	\emph{(2).} If  $\lambda_{1}\br{\Sigma_a^{*\frac{1}{2}}\Sigma_b^{*-1}\Sigma_a^{*\frac{1}{2}}-I} \leq -1 $, let $A:= \Sigma_a^{*\frac{1}{2}}\Sigma_b^{*-1}\Sigma_a^{*\frac{1}{2}}-I$ and assume $U^TAU = V$, where $U$ is an orthogonal matrix and $V:= \text{diag}\bbr{v_1, \cdots,v_d}$ is a diagonal matrix with diagonal elements $v_1 \leq v_2\leq \cdots \leq v_d$ and $v_1 \leq -1$. We can rewrite $y = Uz$ with $z = (z_1,\cdots,z_d)^T$ and $U^T(\Sigmaastar)^{-\frac{1}{2}}\Xi_{a,b} = (\tau_1,\cdots,\tau_d)^T$. Since $-\log |\Sigmaastar| + \log|\Sigmabstar| \leq d\log\frac{\lambdamax}{\lambdamin}$, we can take $z_1 = -\sign{\tau_1}\sqrt{d\log\frac{\lambdamax}{\lambdamin}}$ and $z_i =0$ for $i\geq2$. Then, we have 
	\begin{align*}
	&2y^T(\Sigmaastar)^{-\frac{1}{2}}\Xi_{a,b}+y^T\br{\Sigma_a^{*\frac{1}{2}}\Sigma_b^{*-1}\Sigma_a^{*\frac{1}{2}}-I}y 
	- \log|\Sigmaastar|+\log|\Sigmabstar| - \Xi_{a,b}^T(\Sigmaastar)^{-1}\Xi_{a,b}\\
	=~& 2z_1\tau_1 + v_1z_1^2 - \log|\Sigmaastar|+\log|\Sigmabstar| - \Xi_{a,b}^T(\Sigmaastar)^{-1}\Xi_{a,b}\\
	\leq~ &-2\sqrt{d\log\frac{\lambdamax}{\lambdamin}}|\tau_1| - d\log\frac{\lambdamax}{\lambdamin} -\log |\Sigmaastar| + \log|\Sigmabstar| - \Xi_{a,b}^T(\Sigmaastar)^{-1}\Xi_{a,b}\\
	\leq~ &0.
	\end{align*}
	It means $-(\Sigmaastar)^{-\frac{1}{2}}\Xi_{a,b}+y \notin \mathcal{B}_{a,b}$ and hence $\snr'_{a,b}\leq \norm{-(\Sigmaastar)^{-\frac{1}{2}}\Xi_{a,b}} + \norm{y}$. Then
	Thus we have $\snr'_{a,b} \leq \left\|-(\Sigmaastar)^{-\frac{1}{2}}\Xi_{a,b}\right\| + \sqrt{d\log\frac{\lambdamax}{\lambdamin}} \leq \lambdamin^{-1/2}\norm{\Xi_{a,b}} + \sqrt{d\log\frac{\lambdamax}{\lambdamin}}$.
	
	\emph{(3).} If $-1<\lambda_{1}\br{\Sigma_a^{*\frac{1}{2}}\Sigma_b^{*-1}\Sigma_a^{*\frac{1}{2}}-I} < 0 $, we still use the notations in scenario (2). Notice that $\Sigma_a^{*\frac{1}{2}}\Sigma_b^{*-1}\Sigma_a^{*\frac{1}{2}} = A+I$, we have
	\begin{align*}
	\log\frac{|\Sigmaastar|}{|\Sigmabstar|} &= \log(1+v_1)\cdot\cdots\cdot(1+v_d)\\
	& \geq d\log (1+v_1)\\
	& \geq \frac{3}{2}dv_1.
	\end{align*}
	Now we take $z_1 = -\sign{\tau_1}\sqrt{\frac{3}{2}d}$ and $z_i =0$ for $i\geq2$, then we have
	\begin{align*}
	&2y^T(\Sigmaastar)^{-\frac{1}{2}}\Xi_{a,b}+y^T\br{\Sigma_a^{*\frac{1}{2}}\Sigma_b^{*-1}\Sigma_a^{*\frac{1}{2}}-I}y 
	- \log|\Sigmaastar|+\log|\Sigmabstar| - \Xi_{a,b}^T(\Sigmaastar)^{-1}\Xi_{a,b}\\
	=~& 2z_1\tau_1 + v_1z_1^2 - \log|\Sigmaastar|+\log|\Sigmabstar| - \Xi_{a,b}^T(\Sigmaastar)^{-1}\Xi_{a,b}\\
	\leq~ &-2|\tau_1|\sqrt{\frac{3}{2}d} - \Xi_{a,b}^T(\Sigmaastar)^{-1}\Xi_{a,b}\\
	\leq~ & 0.
	\end{align*}
	Thus we have $\snr'_{a,b} \leq \left\|-(\Sigmaastar)^{-\frac{1}{2}}\Xi_{a,b}\right\| + \sqrt{\frac{3}{2}d} \leq \lambdamin^{-1/2}\norm{\Xi_{a,b}} + \sqrt{\frac{3}{2}d}$.
	
	Overall, we have $\snr'_{a,b} \leq \lambdamin^{-1/2}\norm{\Xi_{a,b}} + \sqrt{\frac{3}{2}d} + \sqrt{d\log\frac{\lambdamax}{\lambdamin}}$ for all the three scenarios.
	
	To prove the lower bound, we have
	\begin{align*}
	x^T \Sigma_a^{*\frac{1}{2}} \Sigma_b^{*-1}(
	\theta^*_a -\theta^*_b) + \frac{1}{2} x^T\br{\Sigma_a^{*\frac{1}{2}}\Sigma_b^{*-1}\Sigma_a^{*\frac{1}{2}}-I}x &\geq -\norm{\Sigma_a^{*\frac{1}{2}} \Sigma_b^{*-1}}\norm{x}\norm{\Xi_{a,b}} - \frac{1}{2}\norm{x}^2\norm{\Sigma_a^{*\frac{1}{2}}\Sigma_b^{*-1}\Sigma_a^{*\frac{1}{2}}-I}\\
	&\geq -\frac{\sqrt{\lambdamax}}{\lambdamin}\norm{x}\norm{\Xi_{a,b}} - \frac{1}{2}\br{\frac{\lambdamax}{\lambdamin}+1}\norm{x}^2.
	\end{align*} 
	By the upper bound we know for any $a \neq b\in [k]$, $\norm{\Xi_{a,b}}\to \infty$  when $\snr' \to \infty$. Thus, we have
	\begin{align*}
	\frac{\sqrt{\lambdamax}}{\lambdamin}\norm{x}\norm{\Xi_{a,b}} + \frac{1}{2}\br{\frac{\lambdamax}{\lambdamin}+1}\norm{x}^2 &\geq \frac{1}{2}\Xi_{a,b}^T\Sigma_b^{*-1}\Xi_{a,b} - \log|\Sigmaastar|+\log|\Sigmabstar|\\
	&\geq \frac{1}{4\lambdamax}\norm{\Xi_{a,b}}^2.
	\end{align*}
	Hence,
	\begin{align*}
	\norm{x} \geq \frac{-\sqrt{\lambdamax}+\sqrt{\lambdamax+\frac{\lambdamin(\lambdamin+\lambdamax)}{2\lambdamax}}}{\lambdamin+\lambdamax}\norm{\Xi_{a,b}}.
	\end{align*}
\end{proof}

In the following lemmas, we are going to establish connections between testing errors and  $\{\snr'_{a,b}\}_{a\neq b}$. Consider any  $a,b\in[k]$ such that $a\neq b$. 
Let $\eta \sim \mathn(0,I_d)$, $\Xi_{a,b}=\theta_a^* - \theta_b^*$, and $\Delta_{a,b}=\norm{\Xi_{a,b}}$. Define
\begin{align*}
\mathcal{B}_{a,b}(\delta) = \Bigg\{x\in \mathr^{d}: x^T \Sigma_a^{*\frac{1}{2}} (\Sigma_b^*)^{-1}\Xi_{a,b} &+ \frac{1}{2} x^T\br{\Sigma_a^{*\frac{1}{2}}(\Sigma_b^*)^{-1}\Sigma_a^{*\frac{1}{2}}-I_d}x  \\
&\leq -\frac{1-\delta}{2}\Xi_{a,b}^T(\Sigma_b^*)^{-1}\Xi_{a,b}  + \frac{1}{2}\log \abs{\Sigma_a^*}- \frac{1}{2}\log \abs{\Sigma_b^*}\Bigg\},
\end{align*}
for any $\delta\in\mathr$. In addition, we define
\begin{align*}
&\snr_{a,b}'(\delta) = \min_{x\in \mathcal{B}_{a,b}(\delta)} 2\norm{x}, \\
\text{and }&P_{a,b}(\delta) = \pbr{\eta \in \mathcal{B}_{a,b}(\delta)}.
\end{align*}
Recall the definitions of $\mathcal{B}_{a,b}$ and $\snr'_{a,b}$ in Section \ref{sec:hetero}. Then they are a special case of $\mathcal{B}_{a,b}(\delta)$ and $\snr_{a,b}'(\delta)$ with $\delta=0$. That is,  we have $\mathcal{B}_{a,b} = \mathcal{B}_{a,b}(0)$ and $\snr'_{a,b} = \snr'_{a,b}(0)$.

\begin{lemma}\label{lem:1}
Assume $d=O(1)$ and $\lambda_{\min}\leq \lambda_1(\Sigma_a^*),\lambda_1(\Sigma_b^*) \leq \lambda_d(\Sigma_a^*), \lambda_d(\Sigma_b^*)\leq \lambda_{\max}$ where $\lambda_{\min},\lambda_{\max}>0$ are constants. Under the condition $\snr'_{a,b}\rightarrow\infty$,
for any positive sequence $\delta=o(1)$, there exists a $\tilde\delta=o(1)$ that  depends on $\delta,d,\dab,\lambdamin,\lambdamax$ such that
\begin{align*}
P_{a,b}(\delta)\leq \ebr{ - \frac{1-\tilde \delta}{8}\snr_{a,b}^{'2}}
\end{align*}
\end{lemma}

\begin{proof}
For convenience and conciseness, we will use  the notation $\theta_a,\theta_b,\Sigma_a,\Sigma_b$ instead of $\theta_a^*,\theta_b^*,\Sigma_a^*,\Sigma_b^*$ throughout the proof. By Lemma \ref{lem:sameorder}, we have $\snr_{a,b}'$ in the same order of $\dab$, which means $\dab\rightarrow\infty$.

%Since $\delta>0$, we have $\snr_{a,b}'(\delta) >\snr_{a,b}'$. 

Assume we had obtained $\snr_{a,b}'(\delta) \geq (1-o(1))\snr_{a,b}'$. 
Then by Lemma \ref{lem:sameorder}, we have $\snr_{a,b}'(\delta)$ in the same order of $\dab$ which is far bigger than $d$ by assumption.
Since $\norm{\eta}^2 \sim \chi^2_d$, using Lemma \ref{lem:chi-square}, we have
\begin{align*}
P_{a,b}(\delta)& \leq P\br{\norm{\eta}^2 \geq \frac{\snr_{a,b}^{'2}(\delta)}{4}} \leq \ebr{-\br{1-O\br{\frac{d}{\dab^2}}} \frac{\snr_{a,b}^{'2}(\delta)}{8}  }\\
&\leq \br{1-O\br{\frac{d}{\dab^2}} \frac{(1-o(1))\snr_{a,b}^{'2}}{8}  }
\end{align*}
which is the desired result.
Hence, the proof of this lemma is all about establishing $\snr_{a,b}'(\delta) \geq (1-o(1))\snr_{a,b}'$. 

To prove it, we first simplify $\snr'_{a,b}(\delta)$. In spite of some abuse of notation, denote $\lambda_1\leq \ldots \leq \lambda_d$ to be  the eigenvalues of $\Sigma_a^{\frac{1}{2}}(\Sigma_b)^{-1}\Sigma_a^{\frac{1}{2}}-I_d$ such that its eigen-decomposition can be written as 
$\Sigma_a^{\frac{1}{2}}(\Sigma_b)^{-1}\Sigma_a^{\frac{1}{2}}-I_d = \sum_{i=1}^d \lambda_i u_iu_i^T,$
where $\{u_i\}$ are orthogonal vectors. Denote  $U=(u_1,\ldots,u_d)$ and 
\begin{align*}
v = U^T \Sigma_a^\frac{1}{2} \Sigma_b^{-1}\xiab
\end{align*}
 and
\begin{align*}
 \mathcal{B}'_{a,b}(\delta)  & = \cbr{y\in \mathr^d: \sum_i y_i v_i + \frac{1}{2}\sum_i \lambda_i y_i^2 \leq -\frac{1-\delta}{2}\xiab^T\Sigma_b^{-1}\xiab  + \frac{1}{2}\log \frac{|\Sigma_a|}{|\Sigma_b|}}.
\end{align*}
Then $ \mathcal{B}'_{a,b}(\delta) $ can be seen a reflection-rotation of $ \mathcal{B}_{a,b}(\delta) $ by the transformation $y=U^tx$. Hence we have $\snr'_{a,b}(\delta) =  \min_{  y  \in \mathcal{B}'_{a,b}(\delta)} 2\norm{  y  }$ for any $\delta$. What is more, let $\bar{\mathcal{B}}'_{a,b}(\delta) $ to be its boundary, i.e.,
\begin{align*}
\bar{\mathcal{B}}'_{a,b}(\delta) & = \cbr{y\in \mathr^d: \sum_i y_i v_i + \frac{1}{2}\sum_i \lambda_i y_i^2 = -\frac{1-\delta}{2}\xiab^T\Sigma_b^{-1}\xiab  + \frac{1}{2}\log \frac{|\Sigma_a|}{|\Sigma_b|}}.
\end{align*}
Since $0\notin  \mathcal{B}'_{a,b}(\delta)  $, we have
$
\snr'_{a,b}(\delta) =  2\min_{  y  \in \bar{\mathcal{B}}'_{a,b}(\delta)} \norm{  y  }
$
As a result, we only need to work on $\bar{\mathcal{B}}'_{a,b}(\delta) $ instead of $ \mathcal{B}_{a,b}(\delta) $. Denote $\bar{\mathcal{B}}'_{a,b} $ to be $\bar{\mathcal{B}}'_{a,b}(0) $ for simplicity.

We then give an equivalent expression of $ \mathcal{B}'_{a,b}(\delta) $. From (\ref{eqn:snr_delta_bound}), we have an upper bound of $\snr'_{a,b}$: $\snr'_{a,b}\leq 2\lambdamin^{-1/2}\dab$ where we use $\dab\gg \lambdamin,\lambdamax,d$. The same upper bound actually holds for $ \mathcal{B}'_{a,b}(\delta) $ for any $\delta=o(1)$ following the same proof. Define $S = \{  y  \in\mathr^d:\|  y  \|\leq 2\lambdamin^{-1/2}\dab\}$. We then have
\begin{align*}
\snr'_{a,b}(\delta) = 2 \min_{  y  \in \bar{\mathcal{B}}'_{a,b}(\delta) \cap S} \norm{  y  }.
\end{align*}

We have the following inequality. Let $g(  y  ):\bar{\mathcal{B}}_{a,b}(\delta) \rightarrow \bar{\mathcal{B}}_{a,b}(0)$ be any mapping. 
%such that $\norm{x-g(x)} = \inf_{y\in \bar{\mathcal{B}}_{a,b}(0)} \norm{x-y}$. 
By the triangle inequality, we have $\norm{  y  } \geq \norm{g(  y  )} - \norm{  y  -g(  y  )}$. We have
\begin{align}
2^{-1}\snr'_{a,b}(\delta) &= \min_{  y  \in \bar{\mathcal{B}}'_{a,b}(\delta) \cap S} \norm{  y  } \nonumber \\
& \geq  \min_{  y  \in \bar{\mathcal{B}}'_{a,b}(\delta) \cap S}  \br{\norm{g(  y  )} - \norm{  y  -g(  y  )}} \nonumber \\
& \geq  \min_{  y  \in \bar{\mathcal{B}}'_{a,b}(\delta) \cap S}  \norm{g(  y  )} -  \max_{  y  \in \bar{\mathcal{B}}'_{a,b}(\delta) \cap S}  \norm{  y  -g(  y  )} \nonumber \\
& \geq \min_{  y  \in \bar{\mathcal{B}}'_{a,b}(0)} \norm{  y  }  -  \max_{  y  \in \bar{\mathcal{B}}'_{a,b}(\delta) \cap S}  \norm{  y  -g(  y  )} \nonumber \\
&\geq \snr'_{a,b}-  \max_{  y  \in \bar{\mathcal{B}}'_{a,b}(\delta) \cap S}  \norm{  y  -g(  y  )}.\label{eqn:snr_delta_snr}
\end{align}
As a result, if we are able to find some $g$ such that $ \max_{  y  \in \bar{\mathcal{B}}'_{a,b}(\delta) \cap S}  \norm{  y  -g(  y  )} =o(1)\snr'_{a,b}$, we will immediately have $\snr'_{a,b}(\delta) \geq (1-o(1))\snr'_{a,b}$ and  the proof will be complete. 

Let $w\in\mathr^d$ be some vector. Define $g(  y  ) =  y   + w \argmin_{t\in\mathr:   y   + tw \in \bar{\mathcal{B}}'_{a,b} } |t|$. If $g(  y  )$ is a well-defined mapping, we have 
\begin{align}\label{eqn:x_gx_diff}
 \max_{  y  \in \bar{\mathcal{B}}'_{a,b}(\delta) \cap S}  \norm{  y  -g(  y  )} & = \max_{  y  \in \bar{\mathcal{B}}'_{a,b}(\delta) \cap S} \min_{t\in\mathr:   y   + tw \in \bar{\mathcal{B}}'_{a,b} } \norm{w} \abs{t},
\end{align}
which can be used to derive an upper bound. However, to make $g(  y  )$ well-defined, we need for any $  y  \in  \bar{\mathcal{B}}'_{a,b}(\delta) \cap S$, there exits some $t\in\mathr$ such that $  y  + tw \in  \bar{\mathcal{B}}'_{a,b}$. This means we have the following two equations:
\begin{align}
&\sum_i y_i v_i + \frac{1}{2}\sum_i \lambda_i y_i^2 = -\frac{1-\delta}{2}\xiab^T\Sigma_b^{-1}\xiab  + \frac{1}{2}\log \frac{|\Sigma_a|}{|\Sigma_b|},  \label{eqn:y_satisfy}\\
\text{and }&\sum_i (y_i+tw_i) v_i + \frac{1}{2}\sum_i \lambda_i (y_i+tw_i)^2 = -\frac{1}{2}\xiab^T\Sigma_b^{-1}\xiab  + \frac{1}{2}\log \frac{|\Sigma_a|}{|\Sigma_b|}. \nonumber
\end{align}
It is equivalent to require $t$ to satisfy	
\begin{align}\label{eqn:t_satisfy}
t\sum_i \br{ w_iv_i+ \lambda_i y_i w_i} +\frac{t^2}{2} \sum_i \lambda_i w_i^2 = -\frac{\delta}{2}\xiab^T\Sigma_b^{-1}\xiab.
\end{align}
Hence, all we need is to find a decent vector $w$ such that: for any $y\in \bar{\mathcal{B}}'_{a,b}(\delta)\cap S$ there exists a $t$ satisfying (\ref{eqn:t_satisfy}), and we can obtain the desired upper bound for (\ref{eqn:x_gx_diff}).

In the following, we will consider four different scenarios according to the spectral $\{\lambda_i\}$. For each scenario, we will construct a $w$ with decent bounds for (\ref{eqn:x_gx_diff}).  Denote $\delta'=\sqrt{\delta}$.
%Let $\delta' >0$ be some sequence that goes to 0 whose value will be determined later.

~\\
\emph{Scenario 1: $\abs{\lambda_1},\abs{\lambda_d}\leq \delta'$.} We choose $w = v/\|v\|$. Note that we have  $\|v\|$ in the same order of $\dab$ and $\|\xiab^T\Sigma_b^{-1}\xiab\|$  in the same order of $\dab^2$. Note that we have
\begin{align*}
t\sum_i \br{ w_iv_i+ \lambda_i y_i w_i} +\frac{t^2}{2} \sum_i \lambda_i w_i^2  &\leq t\norm{v} + \abs{t}\norm{y}\sqrt{\sum_i \lambda_i^2 w_i^2}+\frac{t^2}{2} \sum_i \lambda_i w_i^2\\
&\leq  t\norm{v} +\abs{t}\delta'\norm{y} + \frac{t^2\delta'}{2}\\
&\leq t\norm{v} + 2\abs{t}\delta'\lambdamin^{-1/2}\dab + \frac{t^2\delta'}{2},
\end{align*}
where in the last inequality we use $y\in S$. Define $t_0 = -\delta^{1/2}\dab$. Then we have
\begin{align*}
t_0\sum_i \br{ w_iv_i+ \lambda_i y_i w_i} +\frac{t_0^2}{2} \sum_i \lambda_i w_i^2 &\lesssim - \delta^\frac{1}{2}\dab^2 +  \delta^\frac{1}{2}\delta'\dab^2+ \delta\delta'\dab^2 \ll - \delta\dab^2 \lesssim -\frac{\delta}{2}\xiab^T\Sigma_b^{-1}\xiab.
\end{align*}
Hence for any $y\in S$ there exists a $t\in(t_0,0)$ such that (\ref{eqn:t_satisfy}) is satisfied. Hence, $\abs{t_0} = \delta^{1/2}\dab$ is an upper bound for (\ref{eqn:x_gx_diff}).
 
 ~\\
\emph{Scenario 2: $\lambda_1<-\delta'$.} We choose $w=e_1$ which is the first standard basis of $\mathr^d$. Then,  (\ref{eqn:t_satisfy}) can be written as
\begin{align*}
\lambda_1t^2+2(v_1  + \lambda_1 y_1) t + \delta \xiab^T\Sigma_b^{-1}\xiab=0.
\end{align*}
Since $\lambda_1 <0$, the above equation has two different solutions $t_1,t_2\in\mathr$. Simple algebra leads to
\begin{align*}
\min\{\abs{t_1},\abs{t_2}\} \leq \sqrt{\frac{\delta \xiab^T\Sigma_b^{-1}\xiab}{-\lambda_1}}  \leq \sqrt{\frac{\delta \xiab^T\Sigma_b^{-1}\xiab}{\delta'}} \lesssim \delta^\frac{1}{4}\dab.
\end{align*}
Hence, an upper bound for (\ref{eqn:x_gx_diff}) is $O(\delta^\frac{1}{4}\dab)$.

~\\
\emph{Scenario 3: $\lambda_1 \geq -\delta'$ and there exists a $j\in[d]$ such that $\lambda_j \leq \delta'$ and $\abs{v_j}\geq \sqrt{\delta'}\dab$.} We choose $w=e_j$. Then
(\ref{eqn:t_satisfy}) can be written as
\begin{align*}
\lambda_jt^2+2(v_j  + \lambda_j y_j) t + \delta \xiab^T\Sigma_b^{-1}\xiab=0.
\end{align*}
Note that for any $y\in S$, we have $\abs{v_j  + \lambda_j y_j} \geq \abs{v_j} - \abs{\lambda_j y_j} \geq  \sqrt{\delta'}\dab - \delta'  (2\lambdamin^{-1/2}\dab) \geq \sqrt{\delta'}\dab/2$. Denote $t_0 = -\text{sign}(v_j  + \lambda_j y_j)\sqrt{\delta'}\dab$. Then we have
\begin{align*}
\lambda_jt_0^2+2(v_j  + \lambda_j y_j) t_0 + \delta \xiab^T\Sigma_b^{-1}\xiab & = \lambda_j \delta'\dab^2 -\abs{v_j  + \lambda_j y_j}\sqrt{\delta'}\dab + \delta \xiab^T\Sigma_b^{-1}\xiab \\
&\leq -\br{\frac{\delta'}{2}-\delta^{'2}}\dab^2 + \delta O(\dab^2)\\
&\leq 0.
\end{align*}
As a result, there exists some $t\in(t_0,0)$ satisfying (\ref{eqn:t_satisfy}). Hence, $\abs{t_0} = \delta^{'1/2}\dab$ is an upper bound for (\ref{eqn:x_gx_diff}).

~\\
\emph{Scenario 4: $\lambda_1 \geq -\delta'$  and $\abs{v_j}< \sqrt{\delta'}\dab$ for all $j\in[d]$ such that $\lambda_j \leq \delta'$.} This scenario is slightly more complicated as we need $w$ to be dependent on $y$. Denote it as $w(y)$. Then (\ref{eqn:snr_delta_snr}) still holds and (\ref{eqn:x_gx_diff}) can be changed into
\begin{align}\label{eqn:scanerio_4_0}
 \max_{  y  \in \bar{\mathcal{B}}'_{a,b}(\delta) \cap S}  \norm{  y  -g(  y  )}   = \max_{  y  \in \bar{\mathcal{B}}'_{a,b}(\delta) \cap S} \min_{t\in\mathr:   y   + tw \in \bar{\mathcal{B}}'_{a,b} } \norm{w(y)} \abs{t}.
\end{align} 
Denote $m\in[d]$ to be the integer such that $\lambda_j \leq \delta'$ for all $j\leq m$ and $\lambda_j >\delta'$ for all $j >m$. We can have $m<d$ otherwise this scenario can be reduced to Scenario 1. Define
\begin{align*}
[w(y)]_i = -\br{y_i + \frac{v_i}{\lambda_i}}\indic{i>m}.
\end{align*}
for any $i\in[d]$. Instead of using (\ref{eqn:t_satisfy}), we will analyze it slightly differently.

For $y\in \bar{\mathcal{B}}'_{a,b}(\delta)$, (\ref{eqn:y_satisfy}) can be rewritten as
\begin{align}\label{eqn:scanerio_4_1}
\sum_{i>m} \lambda_i \br{y_i + \frac{v_i}{\lambda_i}}^2 = \sum_{i>m}\frac{v_i^2}{\lambda_i} -(1-\delta)\xiab^T\Sigma_b^{-1}\xiab  + \log \frac{|\Sigma_a|}{|\Sigma_b|} - \br{2\sum_{i\leq m} y_i v_i +\sum_{i\leq m} \lambda_i y_i^2 }.
\end{align}
On the other hand, if $g(y)$ is well-defined,  we need $g(y)\in \bar{\mathcal{B}}'_{a,b}$ which means
\begin{align*}
\sum_{i>m} \lambda_i \br{[g(y)]_i + \frac{v_i}{\lambda_i}}^2 = \sum_{i>m}\frac{v_i^2}{\lambda_i} -\xiab^T\Sigma_b^{-1}\xiab  + \log \frac{|\Sigma_a|}{|\Sigma_b|} - \br{2\sum_{i\leq m} y_i v_i +\sum_{i\leq m} \lambda_i y_i^2 }.
\end{align*}
Note that we have $(y_i + v_i/\lambda_i)(1-t) = [g(y)]_i+v_i/\lambda_i$ for $i>m$ and $[g(y)]_i = y_i$ for $i\leq m$. Then the above display can be written as
\begin{align*}
(1-t)^2\sum_{i>m} \lambda_i \br{y_i + \frac{v_i}{\lambda_i}}^2 = \sum_{i>m}\frac{v_i^2}{\lambda_i} -\xiab^T\Sigma_b^{-1}\xiab  + \log \frac{|\Sigma_a|}{|\Sigma_b|} - \br{2\sum_{i\leq m} y_i v_i +\sum_{i\leq m} \lambda_i y_i^2 }.
\end{align*}
Together with (\ref{eqn:scanerio_4_1}) multiplied, the above equation leads to
\begin{align}\label{eqn:scanerio_4_2}
(1-t)^2\delta \xiab^T\Sigma_b^{-1}\xiab =(2t-t^2) \br{\sum_{i>m}\frac{v_i^2}{\lambda_i} -\xiab^T\Sigma_b^{-1}\xiab  + \log \frac{|\Sigma_a|}{|\Sigma_b|} - \br{2\sum_{i\leq m} y_i v_i +\sum_{i\leq m} \lambda_i y_i^2 }}.
\end{align}
It is sufficient to find some $0<t_0<1$ such that
\begin{align}\label{eqn:scanerio_4_3}
\frac{(1-t_0)^2}{t_0(2-t_0)}\delta \xiab^T\Sigma_b^{-1}\xiab \leq  \sum_{i>m}\frac{v_i^2}{\lambda_i} -\xiab^T\Sigma_b^{-1}\xiab  + \log \frac{|\Sigma_a|}{|\Sigma_b|} - \br{2\sum_{i\leq m} y_i v_i +\sum_{i\leq m} \lambda_i y_i^2 },
\end{align}
then there definitely exists some $0<t<t_0$ satisfying (\ref{eqn:scanerio_4_2}).

We are going to give a lower bound for the right hand side of (\ref{eqn:scanerio_4_3}). Particularly, we need to lower bound $\sum_{i>m}\frac{v_i^2}{\lambda_i} -\xiab^T\Sigma_b^{-1}\xiab$. Denote $\tilde y = U^T(-\Sigma_a^{-1/2}\xiab)$. Then using the definition of $v$ and $\{\lambda_i\}$, we have 
\begin{align*}
 2\sum_{i\in[k]} \tilde y_i v_i +\sum_{i\in[k]} \lambda_i \tilde y_i^2  &=  2(-\Sigma_a^{-\frac{1}{2}}\xiab)^T \Sigma_a^\frac{1}{2} \Sigma_b^{-1}\xiab +(-\Sigma_a^{-\frac{1}{2}}\xiab)^T\br{\Sigma_a^\frac{1}{2}\Sigma_b^{-1}\Sigma_a^\frac{1}{2}- I_d }(-\Sigma_a^{-\frac{1}{2}}\xiab) \\
 & = -\xiab^T \Sigma_b^{-1} \xiab -\xiab^T \Sigma_a^{-1} \xiab.
\end{align*}
Then we have 
\begin{align}
\sum_{i>m}\frac{v_i^2}{\lambda_i} -\xiab^T\Sigma_b^{-1}\xiab &=\xiab^T \Sigma_a^{-1} \xiab  + \sum_{i>m}\lambda_i\br{\tilde y_i + \frac{v_i}{\lambda_i}}^2 + \br{ 2\sum_{i\leq m} \tilde y_i v_i +\sum_{i\leq m} \lambda_i \tilde y_i^2 } \nonumber\\
& \geq  \xiab^T \Sigma_a^{-1} \xiab + \br{ 2\sum_{i\leq m} \tilde y_i v_i +\sum_{i\leq m} \lambda_i \tilde y_i^2 }.\label{eqn:scanerio_4_4}
\end{align}
Hence, the right hand side of (\ref{eqn:scanerio_4_3}) can be lower bounded by
\begin{align}
& \geq  \xiab^T \Sigma_a^{-1} \xiab + \log \frac{|\Sigma_a|}{|\Sigma_b|} + \br{ 2\sum_{i\leq m} \tilde y_i v_i +\sum_{i\leq m} \lambda_i \tilde y_i^2 }  - \br{2\sum_{i\leq m} y_i v_i +\sum_{i\leq m} \lambda_i y_i^2 } \nonumber \\
&\geq   \xiab^T \Sigma_a^{-1} \xiab + \log \frac{|\Sigma_a|}{|\Sigma_b|} - 8\br{\sqrt{\delta'} \sqrt{d}\lambdamin^{-\frac{1}{2}}\dab^2 + \delta'\lambdamin^{-1}\dab^2} \nonumber \\
&\geq  \xiab^T \Sigma_a^{-1} \xiab + \log \frac{|\Sigma_a|}{|\Sigma_b|} -16\sqrt{\delta'd} \lambdamin^{-\frac{1}{2}}\dab^2 ,\label{eqn:scanerio_4_5}
\end{align}
where we use both $\tilde y,y\in S$ and the assumption that $\abs{v_i}\leq \sqrt{\delta'}\dab$ and $\abs{\lambda_i}\leq \delta'$ for any $i\leq m$. Then a sufficient condition for (\ref{eqn:scanerio_4_3})  is $t_0$ satisfies
\begin{align*}
\frac{(1-t_0)^2}{t_0(2-t_0)}\delta \xiab^T\Sigma_b^{-1}\xiab \leq \xiab^T \Sigma_a^{-1} \xiab + \log \frac{|\Sigma_a|}{|\Sigma_b|} +16\sqrt{\delta'd} \lambdamin^{-\frac{1}{2}}\dab^2.
\end{align*}
Since $\xiab^T\Sigma_b^{-1}\xiab$ is in the same order of $\dab^2$ and $ \log \frac{|\Sigma_a|}{|\Sigma_b|} \lesssim d=O(1)$, it can be achieved by $t_0 = \sqrt{\delta}$.

As a result, from (\ref{eqn:scanerio_4_0}) we have
\begin{align*}
 \max_{  y  \in \bar{\mathcal{B}}'_{a,b}(\delta) \cap S}  \norm{  y  -g(  y  )}   &\leq \max_{y  \in \bar{\mathcal{B}}'_{a,b}(\delta) \cap S} \norm{w(y)} \abs{t_0} \leq  \abs{t_0}  \max_{y  \in \bar{\mathcal{B}}'_{a,b}(\delta) \cap S} 2\sqrt{\norm{y}^2 + \frac{\norm{v}^2}{\delta'}} \\
 &\lesssim \sqrt{\delta} \br{\dab + \delta'^{-\frac{1}{2}} \dab} \leq 2\delta^\frac{1}{4}\dab.
\end{align*}

~\\
Combining the above four scenarios, we can see we all have $ \max_{  y  \in \bar{\mathcal{B}}'_{a,b}(\delta) \cap S}  \norm{  y  -g(  y  )}  \lesssim \delta^\frac{1}{4}\dab$ which is $o(1)\snr'_{a,b}$.  By the argument before the discussion of the four scenarios, we have  $\snr'_{a,b}(\delta) \geq (1-o(1))\snr'_{a,b}$ and the proof is complete. 
\end{proof}

%Then (\ref{eqn:t_satisfy}) becomes
%\begin{align*}
%(t^2  - 2t)\sum_{i>m}\lambda_i\br{y_i +\frac{v_i}{\lambda_i}}^2 + \delta\xiab^T\Sigma_b^{-1}\xiab =0.
%\end{align*}
%Since $\abs{v_i}< \sqrt{\delta'}\dab$ for all $i\leq m$, we have $\sum_{i>m} v_i^2 =\norm{v}^2 - \sum_{i\leq m} v_i^2 = \norm{v}^2 - d\delta'\dab^2 \geq \norm{v}^2/2$ which is in the order of $\dab^2$. Recall that $\lambda_i >\delta'$ for all $i>m$. Then we have $\sum_{i>m} (v_i/\lambda_i)^2 \geq $
%
%Let $t_0 = -\delta$. Then $0> t_0^2  + 2t_0 >-\delta$. 
%For any $y\in S$, using $\lambda_i >\delta'$ for all $i>m$, 
%we have 
%\begin{align*}
%(t_0^2  + 2t_0)\sum_{i>m}\lambda_i\br{y_i +\frac{v_i}{\lambda_i}}^2 + \delta\xiab^T\Sigma_b^{-1}\xiab &\leq -\delta\delta' \sum_{i>m}\br{y_i +\frac{v_i}{\lambda_i}}^2 + \delta O(\dab^2) \\
%&\leq -\delta\delta' \br{\sum_{i>m} \frac{v_i^2}{\lambda_i^2}}
%\end{align*}

\begin{lemma}\label{lem:2}
Assume $d=O(1)$ and $\lambda_{\min}\leq \lambda_1(\Sigma_a^*),\lambda_1(\Sigma_b^*) \leq \lambda_d(\Sigma_a^*), \lambda_d(\Sigma_b^*)\leq \lambda_{\max}$ where $\lambda_{\min},\lambda_{\max}>0$ are constants. Under the condition $\snr'_{a,b}\rightarrow\infty$,
for any positive sequence $\delta=o(1)$, there exists a $\tilde\delta=o(1)$ that  depends on $d,\dab,\lambdamin,\lambdamax$ such that
\begin{align*}
P_{1,2}(0) \geq \ebr{-\frac{1+\tilde\delta}{8}\snr_{a,b}^{'2}}.
\end{align*}
\end{lemma}
\begin{proof}
For convenience and conciseness, we will use  the notation $\theta_a,\theta_b,\Sigma_a,\Sigma_b$ instead of $\theta_a^*,\theta_b^*,\Sigma_a^*,\Sigma_b^*$ throughout the proof.
From Lemma \ref{lem:sameorder}, we know $\snr'_{a,b}$ is in the same order of $\dab$, which means $\dab\rightarrow\infty$.
Similar to the proof of Lemma \ref{lem:1},  denote $\lambda_1\leq \ldots \leq \lambda_d$ to be  the eigenvalues of $\Sigma_a^{\frac{1}{2}}(\Sigma_b)^{-1}\Sigma_a^{\frac{1}{2}}-I_d$ such that its eigen-decomposition can be written as 
$\Sigma_a^{\frac{1}{2}}(\Sigma_b)^{-1}\Sigma_a^{\frac{1}{2}}-I_d = \sum_{i=1}^d \lambda_i u_iu_i^T,$
where $\{u_i\}$ are orthogonal vectors. Denote  $U=(u_1,\ldots,u_d)$, 
$
v = U^T \Sigma_a^\frac{1}{2} \Sigma_b^{-1}\xiab
$.
Then denote
\begin{align*}
 \mathcal{B}'_{a,b}  & = \cbr{y\in \mathr^d: \sum_i y_i v_i + \frac{1}{2}\sum_i \lambda_i y_i^2 \leq -\frac{1}{2}\xiab^T\Sigma_b^{-1}\xiab  + \frac{1}{2}\log \frac{|\Sigma_a|}{|\Sigma_b|}},
\end{align*}
and its boundary
\begin{align*}
\bar{\mathcal{B}}'_{a,b} & = \cbr{y\in \mathr^d: \sum_i y_i v_i + \frac{1}{2}\sum_i \lambda_i y_i^2 = -\frac{1}{2}\xiab^T\Sigma_b^{-1}\xiab  + \frac{1}{2}\log \frac{|\Sigma_a|}{|\Sigma_b|}}.
\end{align*}
By the same argument as in the proof of Lemma \ref{lem:1}, $ \mathcal{B}'_{a,b} $ can be seen a reflection-rotation of $ \mathcal{B}_{a,b} $ by the transformation $y=U^tx$. Hence we have $\snr'_{a,b} =  \min_{  y  \in \mathcal{B}'_{a,b}} 2\norm{  y  }$ and we can work on $\mathcal{B}'_{a,b}$ instead of $\mathcal{B}_{a,b}$. Denote $\bar y\in \mathcal{B}'_{a,b}$ to be the one such that $2\norm{\bar y} = \snr'_{a,b}$. From the proof of  Lemma  \ref{lem:1} we also know $\bar y \in S$ which is defined as $S = \{  y  \in\mathr^d:\|  y  \|\leq 2\lambdamin^{-1/2}\dab\}$. In addition, we know $\bar y \in \bar{\mathcal{B}}'_{a,b}$.

We first give the main idea of the remaining proof.  Denote $p(y)$ to be the density function of $y\sim \mathn(0,I_d)$
We will construct a set $T\subset \mathr^d$ around $\bar y$ such that for any $y\in T$ we have $y\in \mathcal{B}'_{a,b}$ and $\norm{y-\bar y} =o(\dab)$. Then we have
\begin{align}\label{eqn:p12}
P_{1,2}(0)& \geq \abs{T} \inf_{y\in T} p(y) =\abs{T}  \frac{1}{(2\pi)^\frac{d}{2}} \ebr{-\frac{1}{2} \max_{y\in T} \norm{y}^2} \nonumber \\
 &=\abs{T}  \frac{1}{(2\pi)^\frac{d}{2}} \ebr{-(1+o(1))\frac{\snr_{a,b}^{'2}}{8}}.
\end{align}
Hence if $\log\abs{T}=o(\snr'_{a,b})$ then the proof will be complete. So it is all about constructing such $T$. We will consider four scenarios same as in the proof of Lemma \ref{lem:1}. Let $\delta=o(1)$  be some positive sequence going to 0 very slowly and denote $\delta'=\sqrt{\delta}$.

~\\
\emph{Scenario 1: $\abs{\lambda_1},\abs{\lambda_d}\leq \delta'$.} Define $w=v/\norm{v}$. We define $T$ as follows:
\begin{align*}
T=\cbr{y=\bar y + s:\norm{(I_d - ww^T)s} \leq \delta' \abs{w^T s}, w^Ts \in[-\delta\dab,0]}.
\end{align*}
Since $\bar y \in \bar{\mathcal{B}}'_{a,b}$, we have 
\begin{align}\label{eqn:p12_0}
\sum_i \bar y_i v_i + \frac{1}{2}\sum_i \lambda_i \bar  y_i^2 = -\frac{1}{2}\xiab^T\Sigma_b^{-1}\xiab  + \frac{1}{2}\log \frac{|\Sigma_a|}{|\Sigma_b|}.
\end{align}
It is obvious $\max_{y\in T}\norm{y -\bar y} \leq 2\delta \dab$. Hence we only need to show that for any $y\in T$, $y\in \mathcal{B}'_{a,b}$, i.e., 
\begin{align}\label{eqn:p12_01}
\sum_i (\bar y_i + s_i) v_i + \frac{1}{2}\sum_i \lambda_i  (\bar  y_i+s_i)^2 \leq -\frac{1}{2}\xiab^T\Sigma_b^{-1}\xiab  + \frac{1}{2}\log \frac{|\Sigma_a|}{|\Sigma_b|}.
\end{align}
From the above two displays, we need to show
\begin{align}\label{eqn:p12_1}
2\sum_i   s_i v_i +  \sum_i \lambda_is_i^2 + 2\sum_i\lambda_i \bar y_i s_i \leq 0.
\end{align}
Note that $s$ satisfies $\norm{s}\leq 2\abs{w^T s}$. 
\begin{align*}
2\sum_i   s_i v_i +  \sum_i \lambda_is_i^2 + 2\sum_i\lambda_i \bar y_i s_i  &\leq 2\norm{v} w^T s + \delta' \norm{s}^2 +  \delta' \norm{\bar y}\norm{s} \\
&\leq  2\norm{v} w^T s +  \delta'  \abs{w^T s}^2+  \delta' \norm{\bar y} \abs{w^Ts } \\
& = \abs{w^T s} \br{-2 \norm{v} + \delta'\abs{w^T s} + \delta'\norm{\bar y}}\\
&\leq 0,
\end{align*}
where we use the fact that $\norm{v},\norm{\bar y}$ are in the order of $\dab$. Hence, for any $y\in T$, we have  shown $y\in \mathcal{B}'_{a,b}$.  From Lemma \ref{lem:volume}, we have $\abs{T}\geq \ebr{d\log \frac{\delta\delta'\dab}{4} -\frac{d}{2}\log d} $. Since $d=O(1)$, $\dab\rightarrow\infty$, and $\delta$ goes to 0 slowly, (\ref{eqn:p12}) leads to $ P_{1,2}(0) \geq  \ebr{-(1+o(1))\frac{\snr_{a,b}^{'2}}{8}}$.

~\\
\emph{Scenario 2: $\lambda_1<-\delta'$.} Denote $e_1$ the first standard basis in $\mathr^d$. Define $T$ as
\begin{align*}
T=\cbr{y=\bar y + s:\norm{(I_d - e_1e_1^T)s} \leq \delta \abs{e_1^T s}, \text{sign}(v_1 + \lambda_1\bar y_1) e_1^Ts \in[-2\delta^\frac{1}{4}\dab, -\delta^\frac{1}{4}\dab]}.
\end{align*}
Here for the sign function we define $\text{sign}(0)=1$.
It is obvious $\max_{y\in T}\norm{y -\bar y} \leq 2\delta \dab$. Hence we only need to establish (\ref{eqn:p12_1}) to show that for any $y\in T$, $y\in \mathcal{B}'_{a,b}$. Note that 
\begin{align*}
2\sum_i   s_i v_i &+  \sum_i \lambda_is_i^2 + 2\sum_i\lambda_i \bar y_i s_i   = 2s_1 (v_1 + \lambda_1\bar y_1) + \lambda_1s^2 + 2\sum_{i\geq 2} s_i(v_i + \lambda_i \bar y_i) + 2\sum_{i\geq 2}\lambda_is_i^2\\
&\leq  2s_1 (v_1 + \lambda_1\bar y_1) + \lambda_1s^2 +2 \norm{(I-e_1e_1^T)s}\br{\norm{v} + \max_{j}\abs{\lambda_j} \norm{\bar y}} + 2  \max_{j}\abs{\lambda_j}   \norm{(I-e_1e_1^T)s}^2\\
&\leq  \br{ \lambda_1 +  \delta^{2}   \max_{j}\abs{\lambda_j}} s_1^2 - 2 \abs{v_1 + \lambda_1\bar y_1} \abs{s_1} + 2\br{\norm{v} + \max_{j}\abs{\lambda_j} \norm{\bar y}}  \delta \abs{s_1} \\
&\leq -\frac{\delta'}{2} s_1^2 + O(\dab) \delta \abs{s_1},
\end{align*}
where we use $ \max_{j}\abs{\lambda_j} =O(1)$ and $\norm{v} , \norm{\bar y}$ are in the order of $\dab$. It is easy to verify the right hand side is negative when $s_1\in[-2\delta^\frac{1}{4}\dab, -\delta^\frac{1}{4}\dab]$. From Lemma \ref{lem:volume}, we have $\abs{T}\geq \ebr{d\log \frac{\delta^\frac{5}{4}\dab}{4} -\frac{d}{2}\log d} $. Then (\ref{eqn:p12}) leads to the desired result.

~\\
\emph{Scenario 3: $\lambda_1 \geq -\delta'$ and there exists a $j\in[d]$ such that $\lambda_j \leq \delta'$ and $\abs{v_j}\geq \sqrt{\delta'}\dab$.} Denote $e_j$ the $j$th standard basis in $\mathr^d$. Define $T$ as
\begin{align*}
T=\cbr{y=\bar y + s:\norm{(I_d - e_je_j^T)s} \leq \delta' \abs{e_j^T s}, \text{sign}(v_j + \lambda_j\bar y_j) e_j^Ts \in[-\delta\dab,0]}.
\end{align*}
Again define $\text{sign}(0)=1$ and it is obvious $\max_{y\in T}\norm{y -\bar y} \leq 2\delta \dab$.  Now we are going to verify (\ref{eqn:p12_1}), i.e., to show $\lambda_j s_j^2 +2s_j(v_j + \lambda_j \bar y_j)\leq -\sum_{i\neq j}\lambda_i s_i^2 -2\sum_{i\neq j}s_i(v_i + \lambda_i \bar y_i) $. On one hand, we have
\begin{align*}
\lambda_j s_j^2 +2s_j(v_j + \lambda_j \bar y_j) & = \lambda_j s_j^2 -2\abs{s_j}\abs{v_j + \lambda_j \bar y_j} \\
&\leq -\delta'\delta\dab \abs{s_j} - 2 (\sqrt{\delta'}\dab -\delta'O(\dab)) \abs{s_j} \\
&\leq -\sqrt{\delta'}\dab\abs{s_j}.
\end{align*}
One the other hand, we have
\begin{align*}
-\sum_{i\neq j}\lambda_i s_i^2 -2\sum_{i\neq j}s_i(v_i + \lambda_i \bar y_i)  &\geq  -\max_{j}\abs{\lambda_j} \norm{(I-e_je_j^T)s}^2 -2  \norm{(I-e_je_j^T)s} \br{\norm{v} + \max_{j}\abs{\lambda_j}  \norm{\bar y}} \\
&\geq   -\max_{j}\abs{\lambda_j} \delta{'2} \abs{s_j}^2 - 2\delta' \abs{s_j} \br{\norm{v} + \max_{j}\abs{\lambda_j}  \norm{\bar y}} \\
&\geq  - 2\delta'\br{\delta \dab \max_{j}\abs{\lambda_j} + \norm{v} + \max_{j}\abs{\lambda_j}  \norm{\bar y}}  \abs{s_j} \\
&\geq -2\delta' O(\dab)\abs{s_j}\\
&\geq -\sqrt{\delta'}\dab\abs{s_j},
\end{align*}
we use $ \max_{j}\abs{\lambda_j} =O(1)$ and $\norm{v} , \norm{\bar y}$ are in the order of $\dab$. Hence (\ref{eqn:p12_1}) is established. From Lemma \ref{lem:volume}, we have $\abs{T}\geq \ebr{d\log \frac{\delta\delta'\dab}{4} -\frac{d}{2}\log d} $. Then (\ref{eqn:p12}) leads to the desired result.

~\\
\emph{Scenario 4: $\lambda_1 \geq -\delta'$  and $\abs{v_j}< \sqrt{\delta'}\dab$ for all $j\in[d]$ such that $\lambda_j \leq \delta'$.}  Denote $m\in[d]$ to be the integer such that $\lambda_j \leq \delta'$ for all $j\leq m$ and $\lambda_j >\delta'$ for all $j >m$. We can have $m<k$ otherwise this scenario can be reduced to Scenario 1.

Define $w\in\mathr^d$ to be unit vector such that
\begin{align*}
w_i  = \begin{cases}
\frac{\lambda_i\bar y_i + v_i }{\sqrt{\sum_{j>m}( \lambda_j\bar y_j + v_j )^2}},\text{ for all }i>m,\\
0,\text{ o.w..}
\end{cases}
\end{align*} 
Define
\begin{align*}
T=\cbr{y=\bar y + s:\norm{(I_d -  w w^T)s} \leq \delta' \abs{ w^T s},  w^Ts \in[-\delta\dab,0]}.
\end{align*}
Now we are going to verify (\ref{eqn:p12_1}), i.e., to show $2\sum_{i> m}s_i(v_i + \lambda_i \bar y_i) + \sum_{i}\lambda_i s_i^2 +2\sum_{i\leq m}s_i(v_i + \lambda_i \bar y_i) \leq 0$. On one hand, we have
\begin{align*}
2\sum_{i> m}s_i(v_i + \lambda_i \bar y_i) &= 2\sqrt{\sum_{j>m}(v_j + \lambda_j \bar y_j)^2}\sum_{i> m} s_i w_i  =  -2\sqrt{\sum_{j>m}(v_j + \lambda_j \bar y_j)^2}  \abs{w^Ts}\\
&\leq -2\sqrt{\delta'} \sqrt{\sum_{j>m}\br{ \bar y_j + \frac{v_j}{\lambda_j}}^2}  \abs{w^Ts}.
\end{align*}
We are going to give a lower bound for $\sum_{j>m}\br{ \bar y_j + \frac{v_j}{\lambda_j}}^2$. Note that  (\ref{eqn:p12_0}) can be written as
\begin{align*}
\sum_{i>m} \lambda_i \br{ \bar y_i + \frac{v_i}{\lambda_i}}^2 &= \sum_{i>m}\frac{v_i^2}{\lambda_i} -\xiab^T\Sigma_b^{-1}\xiab  + \log \frac{|\Sigma_a|}{|\Sigma_b|} - \br{2\sum_{i\leq m} \bar y_i v_i +\sum_{i\leq m} \lambda_i \bar y_i^2 }.
\end{align*}
Denote $\tilde y = U^T(-\Sigma_a^{-1/2}\xiab)$. Using (\ref{eqn:scanerio_4_4}), we have
\begin{align*}
\sum_{i>m} \lambda_i \br{ \bar y_i + \frac{v_i}{\lambda_i}}^2 &\geq \xiab^T \Sigma_a^{-1} \xiab + \br{ 2\sum_{i\leq m} \tilde y_i v_i +\sum_{i\leq m} \lambda_i \tilde y_i^2 }  + \log \frac{|\Sigma_a|}{|\Sigma_b|} - \br{2\sum_{i\leq m} \bar y_i v_i +\sum_{i\leq m} \lambda_i \bar y_i^2 }\\
&\geq  \xiab^T \Sigma_a^{-1} \xiab + \log \frac{|\Sigma_a|}{|\Sigma_b|} -16\sqrt{\delta'd} \lambdamin^{-\frac{1}{2}}\dab^2 ,\\
&\geq C\dab^2,
\end{align*}
for some constant $C>0$.
Here the second inequality is by the same argument as (\ref{eqn:scanerio_4_5}) and the last inequality uses the fact that $\xiab^T \Sigma_a^{-1} \xiab $ is in the order of $\dab^2$ and $d=O(1)$. Hence,
\begin{align*}
\sum_{i>m} \lambda_i \br{ \bar y_i + \frac{v_i}{\lambda_i}}^2 \leq -2\sqrt{C\delta'} \dab\abs{w^Ts}.
\end{align*}
On the other hand, we have
\begin{align*}
\sum_{i}\lambda_i s_i^2  +2\sum_{i\leq m}s_i(v_i + \lambda_i \bar y_i) &\leq  \max_j \abs{\lambda_j} \norm{s}^2 +2 \sqrt{\sum_{i\leq m}s_i^2} \br{\norm{v} +  \max_j \abs{\lambda_j} \norm{\bar y}} \\
&\leq 2\max_j \abs{\lambda_j}\abs{w^Ts}^2 + 2\br{\norm{v} +  \max_j \abs{\lambda_j} \norm{\bar y}}  \norm{(I-ww^T)s}\\
&\leq 2\br{\max_j \abs{\lambda_j} \delta \dab +\delta' \br{\norm{v} +  \max_j \abs{\lambda_j} \norm{\bar y}} }\abs{w^Ts} \\
&\leq O(\delta'\dab) \abs{w^Ts},
\end{align*}
where we use the properties of $s$ as $y\in T$. Summing the above two displays together, we have  (\ref{eqn:p12_1}) satisfied. From Lemma \ref{lem:volume}, we have $\abs{T}\geq \ebr{d\log \frac{\delta\delta'\dab}{4} -\frac{d}{2}\log d} $. Then (\ref{eqn:p12}) leads to the desired result.
\end{proof}

\begin{lemma}\label{lem:volume}
Consider   any positive integer $d$ and  any $0<r<1$, $t>0$.
Define a set $T= \cbr{y\in\mathr^d: \br{\sum_{i\geq 2}y_i^2}^{1/2} \leq r \abs{y_1}, y_1\in[-2t,-t]}$. Then we have
\begin{align*}
\abs{T} \geq \ebr{d\log\frac{rt}{2}  -\frac{d}{2}\log d }.
\end{align*}
\end{lemma}
\begin{proof}
Define a $d$-dimensional ball $B=\cbr{y\in\mathr^d: (y_1 + 1.5t)^2 + \sum_{i\geq 2}y_i^2 \leq (rt/2)^2}$. We can easily verify that $B\in T$. First, for all $y\in B$, we have $y_1\in[-2t,-t]$ as $r\in(0,1)$. Then, we have $\br{\sum_{i\geq 2}y_i^2}^{1/2} \leq rt/2 \leq r\abs{y_1}$.  As a result, by the expression of the  volume of a $d$-dimensional ball, we have 
\begin{align*}
\abs{T} \geq \abs{B}  = \frac{\pi^\frac{d}{2}}{\Gamma(\frac{d}{2} +1)} \br{\frac{rt}{2}}^d\geq  \frac{1}{d^\frac{d}{2}} \br{\frac{rt}{2}}^d  = \ebr{d\log\frac{rt}{2} -\frac{d}{2}\log d },
\end{align*}
where $\Gamma(\cdot)$ is the Gamma function.
\end{proof}

\bibliographystyle{plainnat}
\bibliography{reference}

\begin{thebibliography}{24}
\providecommand{\natexlab}[1]{#1}
\providecommand{\url}[1]{\texttt{#1}}
\expandafter\ifx\csname urlstyle\endcsname\relax
  \providecommand{\doi}[1]{doi: #1}\else
  \providecommand{\doi}{doi: \begingroup \urlstyle{rm}\Url}\fi

\bibitem[Abbe et~al.(2020)Abbe, Fan, and Wang]{AbbeFanWang20}
E.~Abbe, J.~Fan, and K.~Wang.
\newblock {An $\ell_p$-theory of PCA and spectral clustering}.
\newblock \emph{arxiv preprint}, 2020.

\bibitem[Bishop(2006)]{bishop2006pattern}
Christopher~M Bishop.
\newblock \emph{Pattern recognition and machine learning}.
\newblock springer, 2006.

\bibitem[Brubaker and Vempala(2008)]{brubaker2008isotropic}
S~Charles Brubaker and Santosh~S Vempala.
\newblock Isotropic pca and affine-invariant clustering.
\newblock In \emph{Building Bridges}, pages 241--281. Springer, 2008.

\bibitem[Dasgupta(2008)]{dasgupta2008hardness}
Sanjoy Dasgupta.
\newblock \emph{The hardness of k-means clustering}.
\newblock Department of Computer Science and Engineering, University of
  California~…, 2008.

\bibitem[Dempster et~al.(1977)Dempster, Laird, and Rubin]{dempster1977maximum}
Arthur~P Dempster, Nan~M Laird, and Donald~B Rubin.
\newblock Maximum likelihood from incomplete data via the em algorithm.
\newblock \emph{Journal of the Royal Statistical Society: Series B
  (Methodological)}, 39\penalty0 (1):\penalty0 1--22, 1977.

\bibitem[Fei and Chen(2018)]{fei2018hidden}
Yingjie Fei and Yudong Chen.
\newblock Hidden integrality of sdp relaxation for sub-gaussian mixture models.
\newblock \emph{arXiv preprint arXiv:1803.06510}, 2018.

\bibitem[Friedman et~al.(2001)Friedman, Hastie, and
  Tibshirani]{friedman2001elements}
Jerome Friedman, Trevor Hastie, and Robert Tibshirani.
\newblock \emph{The elements of statistical learning}, volume~1.
\newblock Springer series in statistics New York, 2001.

\bibitem[Gao and Zhang(2019)]{gao2019iterative}
Chao Gao and Anderson~Y Zhang.
\newblock Iterative algorithm for discrete structure recovery.
\newblock \emph{arXiv preprint arXiv:1911.01018}, 2019.

\bibitem[Giraud and Verzelen(2018)]{giraud2018partial}
Christophe Giraud and Nicolas Verzelen.
\newblock Partial recovery bounds for clustering with the relaxed $ k $ means.
\newblock \emph{arXiv preprint arXiv:1807.07547}, 2018.

\bibitem[Hsu et~al.(2012)Hsu, Kakade, Zhang, et~al.]{hsu2012tail}
Daniel Hsu, Sham Kakade, Tong Zhang, et~al.
\newblock A tail inequality for quadratic forms of subgaussian random vectors.
\newblock \emph{Electronic Communications in Probability}, 17, 2012.

\bibitem[Kalai et~al.(2010)Kalai, Moitra, and Valiant]{kalai2010efficiently}
Adam~Tauman Kalai, Ankur Moitra, and Gregory Valiant.
\newblock Efficiently learning mixtures of two gaussians.
\newblock In \emph{Proceedings of the forty-second ACM symposium on Theory of
  computing}, pages 553--562, 2010.

\bibitem[Laurent and Massart(2000)]{laurent2000adaptive}
Beatrice Laurent and Pascal Massart.
\newblock Adaptive estimation of a quadratic functional by model selection.
\newblock \emph{Annals of Statistics}, pages 1302--1338, 2000.

\bibitem[Lloyd(1982)]{lloyd1982least}
Stuart Lloyd.
\newblock Least squares quantization in pcm.
\newblock \emph{IEEE transactions on information theory}, 28\penalty0
  (2):\penalty0 129--137, 1982.

\bibitem[L{\"o}ffler et~al.(2019)L{\"o}ffler, Zhang, and
  Zhou]{loffler2019optimality}
Matthias L{\"o}ffler, Anderson~Y Zhang, and Harrison~H Zhou.
\newblock Optimality of spectral clustering for gaussian mixture model.
\newblock \emph{arXiv preprint arXiv:1911.00538}, 2019.

\bibitem[Lu and Zhou(2016)]{lu2016statistical}
Yu~Lu and Harrison~H Zhou.
\newblock Statistical and computational guarantees of lloyd's algorithm and its
  variants.
\newblock \emph{arXiv preprint arXiv:1612.02099}, 2016.

\bibitem[MacQueen et~al.(1967)]{macqueen1967some}
James MacQueen et~al.
\newblock Some methods for classification and analysis of multivariate
  observations.
\newblock 1967.

\bibitem[Ndaoud(2019)]{Ndaoud19}
M.~Ndaoud.
\newblock Sharp optimal recovery in the two component gaussian mixture model.
\newblock \emph{arXiv preprint}, 2019.

\bibitem[Pearson(1894)]{pearson1894contributions}
Karl Pearson.
\newblock Contributions to the mathematical theory of evolution.
\newblock \emph{Philosophical Transactions of the Royal Society of London. A},
  185:\penalty0 71--110, 1894.

\bibitem[Spielman and Teng(1996)]{spielman1996spectral}
Daniel~A Spielman and Shang-Hua Teng.
\newblock Spectral partitioning works: Planar graphs and finite element meshes.
\newblock In \emph{Proceedings of 37th Conference on Foundations of Computer
  Science}, pages 96--105. IEEE, 1996.

\bibitem[Titterington et~al.(1985)Titterington, Smith, and
  Makov]{titterington1985statistical}
D~Michael Titterington, Adrian~FM Smith, and Udi~E Makov.
\newblock \emph{Statistical analysis of finite mixture distributions}.
\newblock Wiley,, 1985.

\bibitem[Vempala and Wang(2004)]{VempalaWang04}
S.~Vempala and G.~Wang.
\newblock A spectral algorithm for learning mixture models.
\newblock \emph{J. Comput. Syst. Sci.}, 68\penalty0 (4):\penalty0 841--860,
  2004.

\bibitem[Von~Luxburg(2007)]{von2007tutorial}
Ulrike Von~Luxburg.
\newblock A tutorial on spectral clustering.
\newblock \emph{Statistics and computing}, 17\penalty0 (4):\penalty0 395--416,
  2007.

\bibitem[Wang et~al.(2020)Wang, Yan, and Diaz]{wang2020efficient}
Kaizheng Wang, Yuling Yan, and Mateo Diaz.
\newblock Efficient clustering for stretched mixtures: Landscape and
  optimality.
\newblock \emph{arXiv preprint arXiv:2003.09960}, 2020.

\bibitem[Wu et~al.(2008)Wu, Kumar, Quinlan, Ghosh, Yang, Motoda, McLachlan, Ng,
  Liu, Philip, et~al.]{wu2008top}
Xindong Wu, Vipin Kumar, J~Ross Quinlan, Joydeep Ghosh, Qiang Yang, Hiroshi
  Motoda, Geoffrey~J McLachlan, Angus Ng, Bing Liu, S~Yu Philip, et~al.
\newblock Top 10 algorithms in data mining.
\newblock \emph{Knowledge and information systems}, 14\penalty0 (1):\penalty0
  1--37, 2008.

\end{thebibliography}

%\begin{thebibliography}{9}
%\bibitem{ref:Lu2016} Lu, Y., \& Zhou, H. H. (2016). Statistical and computational guarantees of lloyd's algorithm and its variants. arXiv preprint arXiv:1612.02099.
%\bibitem{ref:GaoandZhang} Gao, C., \& Zhang, A. Y. (2019). Iterative Algorithm for Discrete Structure Recovery. arXiv preprint arXiv:1911.01018.
%\end{thebibliography}
\end{document}